\documentclass[11pt]{book}
\usepackage[latin1]{inputenc}
\usepackage[english]{babel}
\usepackage{amsmath}
\usepackage[ntheorem]{empheq}
\usepackage{makeidx}
\usepackage{amssymb}
\usepackage{graphics}
\usepackage{amsthm}
\usepackage{nomencl}
\usepackage{mathrsfs}
\usepackage{titleps}
\usepackage{yhmath}
\usepackage{imakeidx}
\usepackage[margin=3cm]{geometry}
\makeindex[title=Notation]
\makenomenclature
\numberwithin{equation}{section}
\usepackage{etoolbox}
\makeatletter
\patchcmd{\l@chapter}{1.0em}{0.8em}{}{}
\makeatother

\title{Centralizers and conjugacy classes in finite classical groups}
\author{Giovanni De Franceschi}

\renewcommand{\maketitle}{

}

\begin{document}
\newpagestyle{main}{

  \sethead
  [\thepage][][\slshape\thechapter.\enspace\chaptertitle]
  {\slshape\thesection.\enspace\sectiontitle}{}{\thepage}
}
\newpagestyle{intro}{

  \setfoot
  [][\thepage][]
  {}{\thepage}{}
}
\newpagestyle{end}{

  \sethead
  [\thepage][][\slshape\chaptertitle]
  {\slshape\sectiontitle}{}{TESTESTESTESTEST\thepage}
}

\pagenumbering{gobble}

\newtheorem{theorem}[equation]{Theorem}
\newcommand{\GL}{\mathrm{GL}}
\newcommand{\SL}{\mathrm{SL}}
\newcommand{\U}{\mathrm{U}}
\newcommand{\Sp}{\mathrm{Sp}}
\newcommand{\Or}{\mathrm{O}}
\newcommand{\SO}{\mathrm{SO}}
\newcommand{\SU}{\mathrm{SU}}
\newcommand{\N}{\mathrm{N}}
\newcommand{\C}{\mathscr{C}}
\newcommand{\im}{\mathrm{im}}
\newcommand{\Spec}{\mathscr{S}}
\newcommand{\tr}{\mathrm{t}}
\newcommand{\rk}{\mathrm{rk}\,}
\newcommand{\ged}{generalized elementary divisor }
\newcommand{\geds}{generalized elementary divisors }
\newcommand\lreqn[2]{\noindent\makebox[\textwidth]{$\displaystyle#1$\hfill(#2)}\vspace{2ex}}
\newtheorem{lemma}[equation]{Lemma}
\newtheorem{definition}[equation]{Definition}
\newtheorem{proposition}[equation]{Proposition}
\newtheorem{corollary}[equation]{Corollary}
\theoremstyle{definition}
\newtheorem{remark}[equation]{Remark}
\newtheorem{example}[equation]{Example}

\begin{center}
\hrule
       \vspace{10mm}
       {\huge Centralizers and conjugacy classes in finite classical groups \par}%
       
       \vspace{10mm}
              {\Large \sc Giovanni De Franceschi \par }
       \vspace{10mm}
{Department of Mathematics, Technische Universit\"{a}t, Kaiserslautern, Germany}
       \vspace{10mm}

       \hrule
\end{center}

\chapter*{Abstract}
\addcontentsline{toc}{chapter}{Abstract}
Let $\C$ be a classical group defined over a finite field. We present comprehensive theoretical solutions to the following closely related problems:
\begin{itemize}
\item List a representative for each conjugacy class of $\C$.
\item Given $x \in \C$, describe the centralizer $C_{\C}(x)$ of $x$ in $\C$, by giving its group structure and a generating set.
\item Given $x,y \in \C$, establish whether $x$ and $y$ are conjugate in $\C$ and, if they are, find explicit $z \in \C$ such that $z^{-1}xz = y$.
\end{itemize}
We also formulate practical algorithms to solve these problems and have implemented them in {\sc Magma}.

\tableofcontents

\mainmatter
\pagestyle{main}
\chapter{Introduction and Background}   \label{IntroductionChapter}

Let $\C$ be a classical group defined over a finite field. We consider the following closely related problems:
\begin{itemize}
\item List a representative for each conjugacy class of $\C$.
\item Given $x \in \C$, describe the centralizer $C_{\C}(x)$ of $x$ in $\C$, by giving its group structure and a generating set.
\item Given $x,y \in \C$, establish whether $x$ and $y$ are conjugate in $\C$ and, if they are, find explicit $z \in \C$ such that $z^{-1}xz = y$.
\end{itemize}
We present comprehensive theoretical solutions to all three problems. Their solution is often a necessary and vital component of algorithms for computational group theory: as just one example, all three are vital in computing the character table of a classical group.  Hence we seek explicit solutions which can be employed widely. To achieve this outcome, we use our theoretical solutions to formulate practical algorithms to solve the problems.  In parallel to our theoretical work, we have developed in \textsc{Magma} \cite{MAGMA} complete implementations of these algorithms.

\section{Summary of results}   \label{Sect12}
In Chapter \ref{LinearChapter} we introduce classical groups and fix the notation. Before analyzing the classical groups, we discuss the three problems for linear groups. Although most of the results are well known, this is important to introduce the approach used in the rest of the article. We show that the three problems can be solved separately for semisimple and unipotent elements, and the solution in the general case can be obtained from the solutions in these two particular cases.

One of the most significant papers on conjugacy classes in classical groups is that of Wall \cite{GEW}. There he solved the problems in a more general context, considering isometry groups over division rings, and he applied the results of his analysis to the finite groups of isometries. In particular, he established whether $x \in \GL(n,q)$ belongs to a certain group $\C$ of isometries, decided whether two elements are conjugate in $\C$, computed the cardinality of each class, and the total number of conjugacy classes in $\C$. In exploiting his results for our purposes, we face two difficulties. The conjugacy problem is discussed in terms of hermitian invariants of forms (see \cite[\S 2.4]{GEW}); these are hard to compute explicitly. Moreover, the analysis in even characteristic differs from subsequent works.

We use a different approach inspired by other work on this topic. Our first step is to establish which conjugacy classes in $\GL(n,q)$ have elements in a certain group $\C$ of isometries. This problem was considered by Britnell \cite[Chap.\ 5]{JBRIT}, Milnor \cite{Milnor} and Wall \cite[\S 2.6,\,\S 3.7]{GEW}. We mainly refer to \cite[Chap.\ 5]{JBRIT}, where Britnell solves the membership problem in symplectic and orthogonal groups of odd characteristic. In Chapter \ref{membershipChapter} we report his results, extending them to hermitian and quadratic forms in every characteristic. 

In Chapter \ref{semisimpleChap} we analyze in detail the semisimple case. Our starting point is the brief paper of Wall \cite{SCCSS} where he discusses the semisimple conjugacy classes in symplectic groups of odd characteristic. In Section \ref{semisimpleconjclasses} we extend this work to all sesquilinear and quadratic forms and then to special and Omega groups. The structure of the centralizer appears in \cite[Chap.\ 3]{PGPG} and \cite[\S 1]{FS}.

The unipotent case has been studied extensively by Liebeck \& Seitz in \cite{LS}. Gonshaw, Liebeck \& O'Brien \cite{GLOB} list explicitly representatives for the unipotent conjugacy classes in classical groups. In Chapter \ref{unipotentChap} we summarize both the results of \cite{GLOB} and the centralizer structures listed in \cite[Chap.\ 7]{LS}, listing only information essential for our task.

In Chapter \ref{generalChap} we use the information about semisimple and unipotent classes to solve the three problems in the general case. Section \ref{SectListGen} describes how to list all conjugacy classes in every classical group. In Section \ref{SectCentrGen} we use  our theoretical analysis to describe explicitly how to get a generating set for the centralizer of every element in a classical group. In Section \ref{SectConjGen} we describe when two elements in a classical group are conjugate and show how to build explicitly a conjugating element.

We expect to soon make publicly available our {\sc Magma} implementation of the resulting algorithms.

\section{Extending our results to an arbitrary finite group}            \label{Sect14}
We expect that our solution to these problems for classical groups will be useful in solving them for an arbitrary finite group $G$. Existing algorithms follow the \textquotedblleft soluble radical model\textquotedblright \, \cite[Chap.\ 10]{lifting}. A practical algorithm to construct the necessary data structure for this model is described in \cite{BHLO} and its  highly optimised implementation is available in {\sc Magma}. Recall that $G$ has a characteristic series
$$
1 \leqslant L \leqslant S \leqslant P \leqslant G,
$$
where
\begin{itemize}
\item $L$ is the solvable radical of $G$;
\item $S/L$ is the socle of $G/L$ and $S/L \cong \prod_i T_i^{d_i}$, where the $T_i$ are non-abelian non-isomorphic simple groups;
\item $P/S \leqslant \prod_i \mathrm{Out}(T_i)^{d_i}$ is solvable;
\item ${G/P \leqslant \prod_i\mathrm{Sym}(d_i)}$.
\end{itemize}
Observe that $G/L$ is a direct product $\prod_i W_i$, with $W_i = \mathrm{Aut}(T_i) \wr \mathrm{Sym}(d_i)$. The solution of the three problems for the $T_i$ allows us to solve them in $G$.
\begin{itemize}
\item The problem to extend the solutions from $T_i$ to $\mathrm{Aut}(T_i) \wr \mathrm{Sym}(d_i)$, and then to $G/L$, was solved by Hulpke \cite{Hulpke2000} and Cannon \& Holt \cite{CannonHolt}.
\item Since $L$ is solvable, there exists a series
$$
L = N_1 \vartriangleright N_2 \vartriangleright \cdots \vartriangleright N_r=1
$$
with $N_i/N_{i+1}$ elementary abelian. The solution of the problem in $G/N_{i+1}$ can be obtained from that in $G/N_i$. This procedure is described in \cite{Hulpke2000, Hulpke2013} for conjugacy classes and \cite[\S 8.8]{lifting} for centralizers.
\end{itemize}
Thus solving the problems in finite groups reduces to their solution for finite simple groups. Constructive recognition algorithms, for example \cite{black}, allow us to map elements from the classical group to its central quotient, so it is readily feasible to use our results to solve the problems in the corresponding finite simple group.

\section*{Acknowledgments}
The content of this article was part of my PhD thesis \cite{thesis}. I am extremely grateful to my supervisors Jianbei An and Eamonn O'Brien for their patience, their guidance and the thoughtful attention they gave me. I thank John Britnell, Alexander Hulpke, Martin Liebeck, Cheryl Praeger and Don Taylor for useful discussions.

\chapter{Preliminaries}   \label{LinearChapter}
In this chapter we define the classical groups over finite fields and discuss centralizers and conjugacy classes in the general linear and special linear groups. This serves to fix both the notation and the general approach used in the rest of the article. The only independent result is the exhibition of a generating set for the centralizer of a unipotent element. All the rest is well known, so it is just sketched.
\section{Finite classical groups}
\subsection{Sesquilinear and quadratic forms}
Let $F= \mathbb{F}_q$ or $\mathbb{F}_{q^2}$ and let $\overline{\lambda} := \lambda^q$ for all $\lambda \in F$. So, $\lambda \mapsto \overline{\lambda}$\index{$ a \mapsto \overline{a}$} is an automorphism of $F$ of order $1 $ or $2$. Let $V$ be an $n$-dimensional $F$-vector space and let $G = \GL(n,F)$. For $v \in V$, we denote by $\overline{v} \in V$ the vector obtained by applying the automorphism $\lambda \mapsto \overline{\lambda}$ to all entries of $v$. The same definition holds for $\overline{X}$, with $X \in M_n(F)$.

A \textit{reflexive sesquilinear form}\index{form!sesquilinear} is a function $\beta: V \times V \rightarrow F$ such that:
\begin{itemize}
\item $\beta(u_1+u_2,v) = \beta(u_1,v)+\beta(u_2,v)$ for all $u_1,u_2,v \in V$;
\item $\beta(u,v_1+v_2) = \beta(u,v_1)+\beta(u,v_2)$ for all $u,v_1,v_2 \in V$;
\item $\beta(au,bv) = a\overline{b}\beta(u,v)$ for all $u,v \in V$ and $a,b \in F$;
\item $\beta(u,v)=0$ if, and only if, $\beta(v,u)=0$.
\end{itemize}
If the automorphism $\lambda \mapsto \overline{\lambda}$ is the identity, then the form is \textit{bilinear}\index{form!bilinear}. The \textit{Gram matrix} of $\beta$ is the matrix $B = (\beta(v_i,v_j))$, where $v_1, \dots, v_n$ is any basis for $V$.

For every non-empty subset $S \subseteq V$, denote $S^{\bot} = \{ v \in V \, | \, \beta(u,v) =0 \, \forall u \in S\}$. If $W$ is a subspace of $V$, then $W$ is \textit{non-degenerate}\index{subspace!non-degenerate} if $W \cap W^{\bot} = \{0\}$, namely if for every non-zero $u \in W$ there exists $v \in W$ such that $\beta(u,v) \neq 0$. The subspace $W$ is \textit{totally isotropic}\index{subspace!totally isotropic} if $W \subseteq W^{\bot}$, namely if $\beta(u,v)=0$ for all $u,v \in W$. The form $\beta$ is \textit{non-degenerate}\index{form!non-degenerate} if $V^{\bot} = \{ 0 \}$. This is equivalent to the condition $\det{B} \neq 0$.\\

We consider on $V$ three types of non-degenerate sesquilinear forms:
\begin{itemize}
\item \textit{Alternating}\index{form!alternating}: $\lambda = \overline{\lambda}$ and $\beta(v,v)=0$ for all $\lambda \in F$ and $v \in V$. In this case $B = -B^{\tr}$ and $B$ has a zero diagonal.
\item \textit{Symmetric}\index{form!symmetric}: $\lambda = \overline{\lambda}$ and $\beta(u,v)=\beta(v,u)$ for all $\lambda \in F$ and $u,v \in V$. In this case $B = B^{\tr}$.
\item \textit{Hermitian}\index{form!hermitian}: $\lambda \mapsto \overline{\lambda}$ has order $2$ and $\beta(u,v) = \overline{\beta(v,u)}$ for all $u,v \in V$. In this case $B = {\overline{B}}^{\tr}$.
\end{itemize}
From now on, in either case $\lambda \mapsto \overline{\lambda}$ has order 1 or 2, we indicate the matrix ${\overline{B}}^{\tr}$ with $B^*$. \label{Tstarsymbol}\index{$B^*$}
\begin{remark} \label{sumofforms} Let $V_1, \dots, V_k$ be vector spaces and let $\beta_i$ be an alternating form on $V_i$ with matrix $B_i$ for every $i=1, \dots, k$. Let $V = V_1 \oplus \cdots \oplus V_k$. Each $v \in V$ can be written uniquely as $v = v_1+\cdots+v_k$, with $v_i \in V_i$. The form $\beta = \beta_1 \oplus \cdots \oplus \beta_k$ on $V$ defined by $\beta(u,v) = \sum_{i=1}^k \beta_i(u_i,v_i)$ is an alternating form on $V$ with matrix $B_1 \oplus \cdots \oplus B_k$, where we denote by $A \oplus B$ the block diagonal join of the matrices $A$ and $B$\index{$A\oplus B$}. Moreover, $\beta$ is non-degenerate if, and only if, each $\beta_i$ is; in such a case, $V_i^{\bot} = \bigoplus_{j \neq i} V_j$ for every $i$. The same holds on replacing alternating forms by symmetric or hermitian forms.
\end{remark}
A \textit{quadratic form} is a function $Q: V \rightarrow F$ satisfying the following conditions:
\begin{itemize}
\item $Q(av) = a^2Q(v)$ for every $a \in F$ and $v \in V$;
\item the function $\beta_Q(u,v):=Q(u+v)-Q(u)-Q(v)$ is a bilinear form.
\end{itemize}
The form $\beta_Q$ \label{betaQsymbol}\index{$\beta_Q$} is the \textit{bilinear form associated to}\index{form!bilinear associated to $Q$} $Q$. Given a $n \times n$ matrix $A$, the function $Q(v) := vAv^{\tr}$ also defines a quadratic form. In such a case, the matrix of $\beta_Q$ is $A+A^{\tr}$. Note that two matrices $A_1, A_2$ define the same quadratic form if, and only if, $A_1-A_2$ is an alternating matrix (that is, it is skew-symmetric and the diagonal is zero). Moreover, a symmetric form $\beta$ is the bilinear form associated to a unique quadratic form $Q$ if, and only if, $q$ is odd. In such a case, $Q(v) = \frac{1}{2}\beta(v,v)$. A subspace of $V$ is \textit{non-degenerate} (resp.\ \textit{totally isotropic}) if it is non-degenerate (resp.\ totally isotropic) for $\beta_Q$. A subspace $W$ is \textit{totally singular}\index{subspace!totally singular} if $Q(v)=0$ for all $v \in V$. A totally singular subspace is always totally isotropic, while the reverse implication holds only in odd characteristic. A quadratic form $Q$ is \textit{non-degenerate} or \textit{non-singular}\index{form!non-singular} if $\beta_Q$ is non-degenerate.\\

Two sesquilinear forms $\beta_1$, $\beta_2$ on $V$ are \textit{congruent}\index{forms!congruent} if there exists $T \in \GL(V)$ such that $\beta_1(uT,vT) = \beta_2(u,v)$ for every $u,v \in V$. In terms of matrices, if $B_1$ and $B_2$ are the matrices of $\beta_1$ and $\beta_2$ respectively, this conditions is equivalent to the existence of $T \in \GL(V)$ such that $TB_1T^* = B_2$. A similar definition holds for quadratic forms: $Q_1$ and $Q_2$ are \textit{congruent} if there exists $T \in \GL(V)$ such that $Q_1(vT) = Q_2(v)$ for all $v \in V$. An important result is the following.

\begin{theorem}\label{basis}
Let $V \cong F^n$.
\begin{itemize}
\item If $n$ is even, then there is one congruence class of non-degenerate alternating forms on $V$. If $n$ is odd, then there are no non-degenerate alternating forms on $V$.
\item All non-degenerate hermitian forms on $V$ are congruent.
\item If $n$ is odd and $q$ is odd, then there are two congruence classes of non-singular quadratic forms on $V$. If $Q$ is a representative for one of the congruence classes, then $\lambda Q$ is a representative for the other class, where $\lambda$ is a non-square in $F$. If $n$ is odd and $q$ is even, then there are no non-singular quadratic forms on $V$.
\item If $n=2k$ is even, then there are two congruence classes of non-singular quadratic forms on $V$. The two classes are distinguished by the dimension of a maximal totally singular subspace of $V$, that is $k$ in one case and $k-1$ in the second case.
\end{itemize}
\end{theorem}
\begin{proof}
For alternating forms, see \cite[p.\ 69]{PG}. For hermitian forms, see \cite[p.\ 116]{PG}. For quadratic forms, see \cite[\S 2.5.3]{KL} or \cite[p.\ 139]{PG}.
\end{proof}
Let $Q$ be a non-singular quadratic form on a vector space $V$ of dimension $n=2k$. The dimension of a totally singular subspace is the \textit{Witt index}, and it is well-defined (see \cite[7.4]{PG}). If the Witt index of $Q$ is $k$, then $Q$ has \textit{plus type}; if the Witt index of $Q$ is $k-1$, then $Q$ has \textit{minus type}.

\subsection{Isometry groups}
Let $V$ be a vector space with a sesquilinear form $\beta$. A matrix $X \in \GL(V)$ is an \textit{isometry}\index{isometry} for $\beta$ if $\beta(uX,vX) = \beta(u,v)$ for every $u,v \in V$. If $B$ is the matrix of $\beta$, then  $X \in \GL(V)$ is an isometry if, and only if, $XBX^* = B$.

If $Q$ is a quadratic form on the vector space $V$, then an \textit{isometry} for $Q$ is $X \in \GL(V)$ such that $Q(vX)=Q(v)$ for every $v \in V$. If $A$ is a matrix of $Q$, then $X \in \GL(V)$ is an isometry for $Q$ if, and only if, $XAX^{\tr}-A$ is an alternating matrix.

The set of isometries for a sesquilinear or a quadratic form is a subgroup of $\GL(V)$, denoted by $\C(\beta)$, $\C(B)$, $\C(Q)$ or $\C(A)$, \label{groupofisometriessymbol}\index{$\C$}with $\beta,B,Q,A$ as above. Isometries can be defined for every sesquilinear or quadratic form, but we will consider non-degenerate reflexive or non-singular forms only.

A simple computation shows that, for every $T \in \GL(V)$, if $X \in \C(B)$, then $TXT^{-1} \in \C(TBT^*)$; in other words, isometry groups for congruent forms are conjugate subgroups in $\GL(V)$. This leads to the definition of the classical groups.

\begin{definition}
Let $V \cong F^n$.
\begin{itemize}
\item Let $F = \mathbb{F}_q$. The group of linear isomorphisms of $V$ is the \emph{linear group} and it is denoted by $\GL(n,F)$ or $\GL(n,q)$.
\item Let $F = \mathbb{F}_q$. If $\beta$ is a non-degenerate alternating form on $V$, then $\C(\beta)$ is the \emph{symplectic group}\index{group!symplectic} and it is denoted by $\Sp(n,F)$ or $\Sp(n,q)$\label{symplecticsymbol}\index{$\Sp(n,q)$}. By Theorem $\ref{basis}$, all symplectic groups are conjugate in $\GL(V)$, so in particular they are isomorphic. Clearly, these are defined only if $n$ is even.
\item Let $F = \mathbb{F}_{q^2}$. If $\beta$ is a non-degenerate hermitian form on $V$, then $\C(\beta)$ is the \emph{unitary group}\index{group!unitary} and it is denoted by $\U(n,F)$ or $\U(n,q)$\label{unitarysymbol}\index{$\U(n,q)$}. By Theorem $\ref{basis}$, all unitary groups are conjugate in $\GL(V)$, so in particular they are isomorphic.
\item Let $F = \mathbb{F}_q$. If $Q$ is a non-singular quadratic form on $V$, then $\C(Q)$ is the \emph{orthogonal group}\index{group!orthogonal}. We use different notation according to their type in Theorem $\ref{basis}$. If $Q$ has plus type, then $\C(Q)$ is denoted by $\Or^+(n,F)$ or $\Or^+(n,q)$\label{orthogonalplussymbol}\index{$\Or^{\epsilon}(n,q)$}. If $Q$ has minus type, then $\C(Q)$ is denoted by $\Or^-(n,F)$ or $\Or^-(n,q)$\label{orthogonalminussymbol}. If $V$ has odd dimension, then $\C(Q)$ is denoted by $\Or(n,F)$, $\Or(n,q)$, $\Or^{\circ}(n,F)$ or $\Or^{\circ}(n,q)$\label{orthogonalsymbol}. We do not need to distinguish between the two congruence classes of $Q$ because $\C(Q) = \C(\lambda Q)$ for every $\lambda \in F$. The groups $\Or^+(n,q)$ and $\Or^-(n,q)$ are not isomorphic. We often refer to these groups as $\Or^{\epsilon}(n,q)$ with $\epsilon \in \{+,-,\circ\}$. 
\end{itemize}
The groups $\Sp(n,q)$, $\U(n,q)$ and $\Or^{\epsilon}(n,q)$ are the \emph{isometry groups}.
\end{definition}
Note that, if $Q$ is a quadratic form, then $\C(Q) \leqslant \C(\beta_Q)$. If $q$ is odd, then equality holds, and $\C(\beta_Q)$ is an orthogonal group. If $q$ is even, then $\C(Q)$ is strictly contained in $\C(\beta_Q)$, and $\C(\beta_Q)$ is a symplectic group.

\begin{definition}
 The \emph{special linear group}\index{group!special linear} $\SL(n,q)$\label{speciallinearsymbol} is the subset of all elements of $\GL(n,q)$ having determinant $1$.
 
The \emph{special orthogonal group}\index{group!special orthogonal} $\SO^{\epsilon}(n,q)$\label{specialorthogonalsymbol}\index{$\SO^{\epsilon}(n,q)$} is the subset of all elements of $\Or^{\epsilon}(n,q)$ having determinant $1$.

The \emph{special unitary group}\index{group!special unitary} $\SU(n,q)$\label{specialunitarysymbol}\index{$\SU(n,q)$} is the subset of all elements of $\U(n,q)$ having {determinant}~{$1$}.
\end{definition}
It is well known that $\SO^{\epsilon}(n,q)$ is a normal subgroup of $\Or^{\epsilon}(n,q)$ of index 1 (if $q$ is even) or 2 (if $q$ is odd) and $\SU(n,q)$ is a normal subgroup of $\U(n,q)$ of index $q+1$. Every element of $\Sp(n,q)$ has determinant $1$.

We introduce the spinor norm\index{spinor norm} $\theta$ on orthogonal groups. For $x \in \Or^{\epsilon}(n,q)$, let $V_x \leqslant V$\label{Vxsymbol}\index{$V_x$} be the image of $\mathbf{1}_V-x$, where $\mathbf{1}_V$ is the identity on $V$\index{$\mathbf{1}_V$}.
\begin{itemize}
\item If $q$ is even, then the \emph{spinor norm} is defined as $\theta(x) := \dim{V_x} \pmod{2}$.
\item Let $q$ be odd. For $x \in \Or^{\epsilon}(n,q)$, its \textit{Wall form}\index{Wall form} is the form on $V_x$ defined by
\begin{eqnarray}      \label{WallForm}
\chi_x(u,v) := \beta(w,v),
\end{eqnarray}
where $w \in V$ is such that $u=w-wx$. If $x \in \Or^{\epsilon}(n,q)$, then $\chi_x$ is a non-degenerate bilinear form on $V_x$, and it does not depend on the choice of $w$ (see \cite[11.32]{PG}). Let $B_x$ be the matrix of $\chi_x$. For $x \in \Or^{\epsilon}(n,q)$, the \emph{spinor norm} $\theta(x)$ is defined as $\det{B_x} \pmod{\mathbb{F}_q^{*2}}$, namely the discriminant of $\chi_x$. For consistency with even characteristic, we set $\theta(x) \in \mathbb{F}_2$, with $\theta(x)=0$ if, and only if, $\det{B_x}$ is a square in $\mathbb{F}_q^*$.
\end{itemize}
The spinor norm is a surjective homomorphism from $\SO^{\epsilon}(n,q)$ to $\mathbb{F}_2$ (see \cite[11.43 and 11.50]{PG}).
\begin{definition}  \label{Omegadef}\index{group!Omega}
The \emph{Omega group} $\Omega^{\epsilon}(n,q)$ is the kernel of the spinor norm on $\SO^{\epsilon}(n,q)$.
\end{definition}
The group $\Omega^{\epsilon}(n,q)$\label{Omegasymbol}\index{$\Omega^{\epsilon}(n,q)$} is the unique subgroup of $\SO^{\epsilon}(n,q)$ of index 2, except for $\SO^+(4,2)$, which has three subgroups of index 2 (see \cite[2.5.7]{KL}).

\section{Conjugacy classes in $\GL(n,q)$}
Let $V$ be an $n$-dimensional vector space on $F = \mathbb{F}_q$ and let $X \in \GL(n,q)$. Let $f_1(t)^{e_1} \cdots f_h(t)^{e_h}$ be the minimal polynomial of $X$, where $f_i(t) \neq t$ are distinct monic irreducible polynomials. We can write
$$
V = V_1 \oplus \cdots \oplus V_h,
$$
where $V_i = \ker{f_i(X)^{e_i}}$ is the \textit{eigenspace} corresponding to $f_i(t)$ for every $i$. By \cite[4.5.1]{PFGG}, every $V_i$ can be written as a direct sum of $F[t]$-submodules
\begin{eqnarray}\label{decompCSM}
V_i = V_{i,1} \oplus \cdots \oplus V_{i,k_i},
\end{eqnarray}
where $X$ acts cyclically on $V_{i,j}$ with minimal polynomial $f_i(t)^{e_{i,j}}$, for some $1 \leq e_{i,j} \leq e_i$. The polynomials $f_1(t)^{e_{1,1}}, \dots, f_1(t)^{e_{1,k_1}}, \dots, f_h(t)^{e_{h,1}}, \dots, f_h(t)^{e_{h,k_h}}$ are the \textit{elementary divisors} of $X$. If $f$ is a power of an irreducible polynomial, then the \textit{multiplicity} of $f$ as an elementary divisor of $X$ is the number of times $f$ appears in the list of elementary divisors of $X$. Since the decomposition in (\ref{decompCSM}) is unique up to rearranging the factors, the list of elementary divisors of $X$ is well-defined. The main result is the following.
\begin{theorem}
Two matrices $X,Y \in \GL(n,q)$ are conjugate if, and only if, they have the same elementary divisors.
\end{theorem}
\begin{proof}
See \cite[6.7.3]{PFGG}.
\end{proof}
As a representative for each conjugacy class of $\GL(n,q)$, we choose the Jordan form. For every monic irreducible polynomial $f$ of degree $d$ and positive integer $e$, the \textit{Jordan block}\index{Jordan block} of order $e$ relative to $f$ is the block matrix
$$
\begin{pmatrix}
C & \mathbb{I} && \\ & \ddots & \ddots & \\ && \ddots & \mathbb{I} \\ &&& C
\end{pmatrix},
$$
where $C$ is the companion matrix of $f$, $\mathbb{I}$\index{$\mathbb{I}$, $\mathbb{I}_n$} is the $d \times d$ identity matrix and $C$ appears $e$ times. A Jordan block is \textit{unipotent}\index{Jordan block!unipotent} if $f(t)=t-1$. For every $V_{i,j}$ as in (\ref{decompCSM}) there is a basis such that the matrix of the restriction of $X$ to $V_{i,j}$ is the Jordan block of order $e_{i,j}$ relative to $f_i$; hence, there exists a basis of $V$ such that the matrix of $X$ is a diagonal join of Jordan blocks. This matrix is the \textit{Jordan form} of $X$.

\section{Centralizers in $\GL(n,q)$}
In this section we show how the centralizer of every matrix in $\GL(n,q)$ is computed. The results are all well-known, except for the generation of centralizers of unipotent elements; our work is motivated by that of Murray \cite{Murray}, but our generating set is independent. The strategy will be useful in the next chapters.\\

For every $x,z \in G=\GL(n,q)$, the relation $C_G(x^z) = C_G(x)^z$ implies that to compute the centralizer of an element in $G$ it is sufficient to compute the centralizer of any element in its conjugacy class. We choose the Jordan form.\\

\begin{lemma}     \label{SplitCentr}
Every element of $C_G(x)$ fixes the subspaces $V_i$.
\end{lemma}
\begin{proof}
If $v \in V_i$, then $vf_i(x)=0$. Hence, for every $y \in C_G(x)$,
$$
(vy)f_i(x) = v(yf_i(x)) = v(f_i(x)y) = (vf_i(x))y = 0y=0,
$$
using the fact that $y$ commutes with $f_i(x)$. Thus $vy$ belongs to $\ker{f_i(x)} = V_i$.
\end{proof}
The previous lemma implies that if $X = \bigoplus_{i=1}^h X_i$, where $X_i$ is the restriction of $X$ to $V_i$, then $C_{\GL(n,q)}(X)$ consists of all matrices of the form $Y=\bigoplus_{i=1}^h Y_i$, where $Y_i$ commutes with $X_i$.

\begin{definition}
	Let $G = \GL(n,q)$, with $n$ a positive integer and $q=p^a$ prime power, and let $r$ be the order of $x \in G$. If $\gcd(r,p)=1$, then $x$ is \emph{semisimple}\index{semisimple}. If $r$ is a $p$-power, then $x$ is \emph{unipotent}\index{unipotent}.
\end{definition}
\begin{lemma} \label{splitSU}
	Every $x \in G$ can be written uniquely as a product of a semisimple and a unipotent element which commute with each other. In symbols $x = su=us$ for $s$ semisimple and $u$ unipotent.
\end{lemma}
\begin{proof}
	If $x$ has order $p^{\alpha}m$ for $\gcd(m,p)=1$, then take $s=x^k$ and $u=x^{p^{\alpha}m+1-k}$, where $k$ is an integer such that $p^{\alpha}|k$ and $k \equiv 1 \pmod{m}$.
	For the uniqueness, if $x=s_1u_1 = s_2u_2$ are two decompositions of $x$, then we would have $s_2^{-1}s_1 = u_2u_1^{-1}$, where the left hand side has order coprime to $p$ and the right hand side has order a power of $p$. The only possibility is $s_1=s_2$ and $u_1=u_2$.
\end{proof}
\begin{definition}
	We call $x=su=us$ the \emph{Jordan decomposition}\index{Jordan decomposition} of $x$.
\end{definition}
For example, if $x$ is a Jordan block,
$$
x = \left( \begin{array}{cccc}
C & \mathbb{I} & & \\
& \ddots & \ddots & \\
& & \ddots & \mathbb{I}\\
& & & C \end{array}\right),
$$
then the Jordan decomposition of $x$ is $x=su=us$, where
$$
s = \left( \begin{array}{cccc}
C &  & & \\
& \ddots & & \\
& & \ddots & \\
& & & C \end{array}\right) \: \mbox{ and } \: u = \left( \begin{array}{cccc}
\mathbb{I} & C^{-1} & & \\
& \ddots & \ddots & \\
& & \ddots & C^{-1}\\
& & & \mathbb{I} \end{array}\right).
$$
If $x$ is a diagonal join of Jordan blocks, then $s$ and $u$ are the diagonal joins of the semisimple and the unipotent parts of every single block.

Clearly $C_G(x) = C_G(s) \cap C_G(u)$: the containment $\supseteq$ is obvious because if an element centralizes $s$ and $u$, then it also centralizes their product $su = x$, and $\subseteq$ follows since $s$ and $u$ commute each other. Thus, the problem to find the centralizer of $x \in G$ can be solved separately for unipotent and semisimple elements.

Since intersection is a difficult operation, it is convenient to compute $C_G(s)$ and then, using the fact that $u \in C_G(s)$, compute $C_{C_G(s)}(u)$, as we will see below.

\subsection{Centralizer of a semisimple element}
Let $x \in \GL(n,q)$ be a semisimple element. By Lemma \ref{SplitCentr}, we can assume that $x$ has a unique elementary divisor $f(t) \in \mathbb{F}_q[t]$.

If $\deg{f}=1$, then $x$ is a scalar matrix, so every matrix commutes with $x$.

Now suppose that $f(t) \in \mathbb{F}_q[t]$ is irreducible of degree $r>1$. Let $E = \mathbb{F}_q[t]/(f)$ be the splitting field of $f$ over $\mathbb{F}_q$ and let $\lambda \in E$ be a root of $f$. Every element of $E$ can be described as $\phi(\lambda)$ for some polynomial $\phi(t) \in \mathbb{F}_q[t]$ of degree smaller than $r$. For every positive integer $m$, there is a canonical embedding of $\GL(m,q^r)$ into $\GL(mr,q)$ sending the matrix $(\phi_{ij}(\lambda))$ into the block matrix $(\phi_{ij}(C))$, where $C$ is the companion matrix of $f$ (see \cite[2.1.4]{PGPG}). This embedding is important in the next chapters.

\begin{example}
	Let $J$ be the Jordan block of order $e$ relative to the polynomial $f$. If $\lambda$ is a root of $f$ in its splitting field $E$ over $\mathbb{F}_q$, then $J$ is the embedding into $\GL(er,q)$ of the matrix
	$$
	\begin{pmatrix}
	\lambda & 1 & & \\ & \ddots & \ddots & \\ & & \ddots & 1 \\ & & & \lambda
	\end{pmatrix} \in \GL(e,E),
	$$
	where $r = \deg{f}$.
\end{example}
Now, let us go back to the case where $x$ is semisimple with a unique elementary divisor $f(t)$ of degree $r>1$ and multiplicity $m$. We can suppose that the matrix of $x$ is a diagonal join of $m$ copies of $C$, the companion matrix of $f$. If $\lambda \in E$ is a root for $f$, then $x$ is the embedding into $\GL(mr,q)$ of the scalar matrix $\lambda \mathbb{I}_m \in \GL(m,E)$. Every matrix of $\GL(m,E)$ commutes with $\lambda \mathbb{I}_m$, so its embedding into $\GL(mr,q)$ commutes with $x$. On the other hand, these are the only matrices in $\GL(mr,q)$ that commute with $x$ (see \cite[3.1.9]{PGPG}). This leads to the following theorem for semisimple elements.
\begin{theorem} \label{largerfield}
	Let $x \in \GL(n,q)$ be semisimple with characteristic polynomial $f_1(t)^{e_1} \cdots f_h(t)^{e_h}$, for $f_i \in \mathbb{F}_q[t]$ irreducible. Then
	$$
	C_{\GL(n,q)}(x) \cong \bigoplus_{j=1}^{h} \GL(e_j,q^{r_j}),
	$$
	with $r_j = \deg{f_j}$.
\end{theorem}

\subsection{Centralizer of a unipotent element} \label{Cue}
Let $x \in G=\GL(n,q)$ be unipotent of order $p^b$. Following the approach of \cite[\S 2.3]{Murray}, it is convenient to work in the matrix algebra $M = M_n(\mathbb{F}_q)$\index{$M_n(F)$} and find the centralizer $C_M(x)$ of $x$ in $M$. The centralizer of $x$ in $G$ is the set of invertible elements of $C_M(x)$.

Since the function $a \mapsto a^p$ is an automorphism of $\mathbb{F}_q$, the unique eigenvalue of $x$ is $1$, so the Jordan form of $x$ is
$$
\left(\begin{array}{cccc}
J_{\lambda_1} & & & \\
& J_{\lambda_2} & & \\
& & \ddots & \\
& & & J_{\lambda_k}
\end{array}\right),
$$
where $J_{\lambda_i}$ is the unipotent Jordan block of dimension $\lambda_i$ and $\lambda_1 + \cdots + \lambda_k = n$. We suppose $\lambda_1 \leq \lambda_2 \leq \cdots \leq \lambda_k$.

Take an element $y$ centralizing $x$ and write the matrix of $y$ as
\begin{eqnarray}  \label{BlocksBij}
\begin{pmatrix}
B_{11} & \cdots & B_{1k}\\
\vdots & \ddots & \vdots \\
B_{k1} & \cdots & B_{kk}
\end{pmatrix},
\end{eqnarray}
where $B_{ij}$ is a block of dimension $\lambda_i \times \lambda_j$ for every $i,j$.

Denote by $X_{c\times d}^a$ the $c \times d$ matrix whose $(i,j)$-entry is 1 if $j-i = a$ and 0 otherwise. A straightforward computation shows that the conditions $xy=yx$ is equivalent to $B_{ij} = \sum_{a=\lambda_j-\lambda_i}^{\lambda_j-1} b_aX_{\lambda_i\times \lambda_j}^a$ (resp.\ $\sum_{a=0}^{\lambda_j-1}b_aX_{\lambda_i\times \lambda_j}^a$) if $\lambda_i\leq \lambda_j$ (resp.\ $\lambda_i > \lambda_j$) for every $i,j$, where the $b_a$ run over $\mathbb{F}_q$.

For every integer $\lambda_s$, let $\mathbb{F}_q[t]_{\lambda_s} = \mathbb{F}_q[t]/(t^{\lambda_s})$ be the truncated polynomial algebra. The multiplication
$$
{\mathbb{F}_q[t]_{\lambda_s} \times \mathbb{F}_q[t]_{\lambda_{s'}} \rightarrow \mathbb{F}_q[t]_{\lambda_{s'}}}
$$
is defined by multiplying the two polynomials and removing all of the monomials of degree greater than $\lambda_{s'}$. There is an algebra isomorphism between the quotient algebra
$$
\mathbb{F}_q[t]_{\lambda} = \left( \begin{array}{cccc}
\mathbb{F}_q[t]_{\lambda_1} & t^{\lambda_2-\lambda_1}\mathbb{F}_q[t]_{\lambda_2} & \cdots & t^{\lambda_k-\lambda_1}\mathbb{F}_q[t]_{\lambda_k} \\
\mathbb{F}_q[t]_{\lambda_1} & \mathbb{F}_q[t]_{\lambda_2} & \cdots & t^{\lambda_k-\lambda_2}\mathbb{F}_q[t]_{\lambda_k}\\
\vdots & \vdots & \ddots & \vdots \\
\mathbb{F}_q[t]_{\lambda_1} & \mathbb{F}_q[t]_{\lambda_2} & \cdots & \mathbb{F}_q[t]_{\lambda_k}
\end{array} \right),
$$
and $C_M(x)$ defined by sending $t^a$ in the $(i,j)$-entry into $X^a_{c \times d}$ in the $(i,j)$-block and extending by linearity. This allows us to work on $\mathbb{F}_q[t]_{\lambda}$ to describe the centralizer of $x$.

Searching for invertible elements of $C_M(x)$ is equivalent to searching for invertible elements of $\mathbb{F}_q[t]_{\lambda}$. The $(\mu,\nu)$-entry of an arbitrary element of $C_M(x)$ corresponds to a pair $(\lambda_{\mu},\lambda_{\nu})$. Since the $\lambda_{\mu}$ are not necessarly distinct, it is convenient to introduce the following notation: let $h = |\{\lambda_1, \dots, \lambda_k\}|$ and let $l_i$ be the multiplicity of $\lambda_i$ in the set $\{\lambda_1, \dots, \lambda_k\}$, for $1 \leq i \leq h$. Let us assemble the entries sharing the same values $(\lambda_i,\lambda_j)$ in a unique $l_i\times l_j$ block and write an arbitrary element of $\mathbb{F}_q[t]_{\lambda}$ as a block matrix
\begin{eqnarray}\label{LF}
A = \left(\begin{array}{ccc} A_{11} & \cdots & A_{1h}\\ \vdots & \ddots & \vdots \\ A_{h1} & \cdots & A_{hh}\end{array}\right)
\end{eqnarray}
with $A_{ij} = \sum_{s=\lambda_j-\lambda_i}^{\lambda_j-1} t^sA_{ij}^{(s)}$, where the  $A_{ij}^{(s)}$ are $l_i \times l_j$ matrices with entries in $\mathbb{F}_q$ and $t^s=0$ whenever $s<0$.

If we write an arbitrary $B \in \mathbb{F}_q[t]_{\lambda}$ as $B=B_0+tB_1+\cdots + t^{\lambda_k-1}B_{\lambda_k-1}$ with
$$
B_s = \left(\begin{array}{ccc} A_{11}^{(s)} & \cdots & A_{1h}^{(s)}\\ \vdots & \ddots & \vdots \\ A_{h1}^{(s)} & \cdots & A_{hh}^{(s)}\end{array}\right),
$$
then it is clear that $B$ is invertible if and only if $B_0$ is invertible. Since
$$
B_0 = \left(\begin{array}{ccc} A_{11}^{(0)} & \cdots & A_{1h}^{(0)}\\  &  \ddots & \vdots \\ 0 &  & A_{hh}^{(0)}\end{array}\right),
$$
$B$ is invertible if and only if $A_{ii}^{(0)}$ is invertible, equivalently $A_{ii}^{(0)} \in \GL(l_i,q)$.

This allows us to get more information about the structure of $C_G(x)$ for unipotent $x$. Let $R$ be the subgroup of $\mathbb{F}_q[t]_{\lambda}^*$ consisting of the matrices of the form
$$
\left(\begin{array}{ccc} A_{11}^{(0)} &  & \\  &  \ddots & \\  &  & A_{hh}^{(0)}\end{array}\right)
$$
with $A_{ii}^{(0)} \in \GL(l_i,q)$. Let $U$ be the subgroup of $\mathbb{F}_q[t]_{\lambda}^*$ consisting of the matrices of the form
$$
\left(\begin{array}{cccc}
1+tU_{11} & U_{12} & \cdots & U_{1h}\\
U_{21} & \ddots & \ddots & \vdots \\
\vdots & \ddots & \ddots & U_{h-1,h}\\
U_{h1} & \cdots & U_{h,h-1} & 1+tU_{hh} \end{array}\right),
$$
where $U_{ij}$ is an arbitrary $l_i\times l_j$ block with entries in $t^{\lambda_j-\lambda_i}\mathbb{F}_q[t]_{\lambda_i}$ (with $t^{\lambda_j-\lambda_i}=1$ if $\lambda_i>\lambda_j$), equivalently the matrices of the form of Equation (\ref{LF}) with identical constant term in the blocks on the diagonal. It is easy to see that $R \cap U$ is trivial. Moreover, an arbitrary element of $\mathbb{F}_q[t]_{\lambda}$ can be written as a product of an element of $R$ and an element of $U$, since the multiplication by an element of $R$ acts on rows or columns by multiplying them by an invertible matrix. Finally, $U$ is normal in the unit group of $\mathbb{F}_q[t]_{\lambda}$. To check this, write an element of $U$ as $B_0 + tB_1$, where $B_0$ is a constant matrix and $B_1$ is a matrix with coefficients in $\mathbb{F}_q[t]$; now $A \in R$ is an upper triangular block matrix and $B_0$ is an upper unitriangular block matrix, so $AB_0A^{-1}$ is an upper unitriangular block matrix and thus $A(B_0+tB_1)A^{-1} = AB_0A^{-1}+t(AB_1A^{-1})$ is an element of $U$. Thus we have proved the following.
\begin{theorem} \label{Card}
	Let $x = \bigoplus_{i=1}^h J_{\lambda_i}^{\oplus l_i}$ be a unipotent element of $\GL(n,q)$, where $J_{\lambda_i}^{\oplus l_i}$ is the direct sum of $l_i$ copies of $J_{\lambda_i}$, with $\lambda_1 < \cdots < \lambda_h$. Then $C_{\GL(n,q)}(x) =U \rtimes R$, where $R \cong \prod_{i=1}^h \GL(l_i,q)$ and
	\begin{eqnarray} \label{CentStruct}
	|U| = q^{\gamma} \mbox{ with } \gamma = 2\sum_{i<j}\lambda_il_il_j + \sum_i (\lambda_i-1)l_i^2.
	\end{eqnarray}
\end{theorem}
\begin{proof}
	Working in $\mathbb{F}_q[t]_{\lambda}$ instead of $C_G(x)$, the groups $R$ and $U$ are the subgroups described above, and we have proved that $C_G(x) \cong U \rtimes R$. The equality $R \cong \prod_{i=1}^h \GL(l_i,q)$ is trivial by the definition of $R$. It remains to compute the cardinality of $U$, namely how many choices are there for the $U_{ij}$. The term $1+tU_{ii}$ equals $1+tA_{ii}^{(1)} + \cdots + t^{\lambda_i-1}A_{ii}^{(\lambda_i-1)}$ and $A_{ii}^{(s)}$ is a matrix in $M_{l_i}(\mathbb{F}_q)$ which can be chosen arbitrarily; thus for $U_{ii}$ there are $q^{(\lambda_i-1)l_i^2}$ choices and these give the second sum in (\ref{CentStruct}).
	
	If $i \neq j$, then $U_{ij}$ is $t^{\lambda_j-\lambda_i}A_{ij}^{(\lambda_j-\lambda_i)} + \cdots + t_{\lambda_j-1}A_{ij}^{(\lambda_j-1)}$ if $i<j$ or $A_{ij}^{(0)} + \cdots + t^{\lambda_j-1}A_{ij}^{(\lambda_j-1)}$, if $j>i$, and every matrix $A_{ij}^{(s)}$ is an arbitrary $l_i\times l_j$ matrix with entries in $\mathbb{F}_q$. Thus the number of choices is $q^{\lambda_il_il_j}$ if $i<j$, or $q^{\lambda_jl_il_j}$ if $i>j$. Summing over all $i \neq j$, we get the first sum in  (\ref{CentStruct}).
\end{proof}
Let $x = \bigoplus_{i=1}^k J_{\lambda_i}^{\oplus a_i}$, where $J_{\lambda_i}$ is the unipotent Jordan block of dimension $\lambda_i$. We now describe a generating set for $C_G(x)$. It is easier to work in the truncated polynomial algebra $\mathbb{F}_q[t]_{\lambda}$.

We proceed to construct a generic element of $\mathbb{F}_q[t]_{\lambda}$. Let $\omega$ be a primitive element of $\mathbb{F}_q$.
\begin{enumerate}
	\item We need to generate the subgroup $R = \prod_{i=1}^k \GL(a_i,q)$. Generators for the direct factors are described in \cite{DET}.
	
	\item The unit group of $\mathbb{F}_q[t]_{\lambda_k}$ is generated by
	$$
	\{\omega\} \cup \{ 1+\omega^it^j: 0 \leq i \leq [\mathbb{F}_q:\mathbb{F}_p], 1 \leq j \leq \lambda_k-1\}.
	$$
	Let $y = \mu_0 + \mu_1t + \cdots + \mu_{\lambda_k-1}t^{\lambda_k-1}$, with $\mu_i = \sum_{j=0}^{r-1} b_{ij}\omega^j$ for $b_{ij} \in \mathbb{F}_p$ and $\mu_0 \neq 0$. We can suppose that $\mu_0 = 1$, since $y = \mu_0 (1+(\mu_1/\mu_0)t + \cdots (\mu_{\lambda_k-1}/\mu_0)t^{\lambda_k-1})$ and $\mu_0$ is a power of $\omega$. Thus
	$$
	\prod_{j=0}^{r-1} (1+\omega^jt)^{b_{1j}} = 1 + \mu_1t + y_1t^2
	$$
	for some $y_1 \in \mathbb{F}_q[t]_{\lambda_k}$. If $y_1 = (a_0 +a_1t + \cdots)$, with $a_0 \neq 0$, and $\mu_2' = \mu_2/a_0 = \sum_{j=0}^{r-1} b_{2j}'\omega^j$, then
	$$
	(1 + \mu_1t + y_1t^2)\prod_{j=0}^{r-1}(1+\omega^jt^2)^{b_{2j}'} = 1+\mu_1t+\mu_2t^2+y_2t^3
	$$
	for some $y_2 \in \mathbb{F}_q[t]_{\lambda_k}$. We proceed in this way acting every time on the coefficient of largest degree and we get the complete expression for $y$.
	
	\item The elements in a single block $A_{ii}$ in (\ref{LF}) are generated by the generators of $\GL(\lambda_i,q)$ and by elements of the form
	$$
	\left(\begin{array}{cccc}
	y(t) & & & \\
	& 1 & & \\
	& & \ddots & \\
	& & & 1 \end{array} \right),
	$$
	where $y(t)$ is a generator of $\mathbb{F}_p[t]_{\lambda_i}^*$ as in the previous point. We prove that these elements are sufficient in the case $\lambda_i = 2$; the other cases follow the arguments for block matrices (the case $\lambda_i=1$ is trivial). Conjugating elements of the form
	$$
	\left( \begin{array}{cc} f & 0 \\ 0 & 1 \end{array}\right)
	$$
    by $c = \scriptsize{\left( \begin{array}{cc} 0 & 1 \\ 1 & 0 \end{array}\right)} \in \GL(2,q)$, we get the elements of the form
	$$
	\left( \begin{array}{cc} 1 & 0 \\ 0 & f \end{array}\right),
	$$
	with $f = f(t) \in \mathbb{F}_q[t]_2^*$. For every $f,g \in \mathbb{F}_q[t]_2^*$, 
	$$
	\left( \begin{array}{cc} f & 1 \\ 1 & g \end{array}\right) = \left( \begin{array}{cc} f & 0 \\ 0 & 1 \end{array}\right)\left( \begin{array}{cc} 1 & 0 \\ 1 & 1 \end{array}\right)\left( \begin{array}{cc} 1 & 0 \\ 0 & 1-fg \end{array}\right)\left( \begin{array}{cc} 1 & -1 \\ 0 & 1 \end{array}\right)\left( \begin{array}{cc} 1 & 0 \\ 0 & -f^{-1} \end{array}\right).
	$$
	Finally, a generic matrix can be written as
	$$
	\left( \begin{array}{cc} f & g \\ h & l \end{array}\right)= \left( \begin{array}{cc} fh^{-1} & 1 \\ 1 & lg^{-1} \end{array}\right)\left( \begin{array}{cc} h & 0 \\ 0 & 1 \end{array}\right)\left( \begin{array}{cc} 1 & 0 \\ 0 & g \end{array}\right)
	$$
	if $h$ and $g$ are invertible; otherwise we reduce to this case by conjugating with $c$.
	
	\item Now we construct the blocks $B_{ij}$ described in (\ref{BlocksBij}) with $\lambda_i \neq \lambda_j$. Let us consider only upper unitriangular matrices $y$, namely the block matrices where $B_{ii}=1$ for every $i=1, \dots, k$ and $B_{ij}=0$ for every $i>j$. In the notation of the algebra $\mathbb{F}_q[t]_{\lambda}$,
	\begin{eqnarray}   \label{Formy}
	y = \left(\begin{array}{ccc} 1 & & *\\ & \ddots & \\ 0 & & 1 \end{array}\right),
	\end{eqnarray}
	where the entry in position $(i,j)$ is a polynomial in $\mathbb{F}_q[t]_{\lambda_j}$ for any $i<j$. To generate this subgroup it is sufficient to take those elements with 1 on the diagonal and in the position $(i,i+1)$ for exactly one $i$ and 0 elsewhere. It is sufficient to add to our generating set only such elements where $\lambda_i \neq \lambda_{i+1}$, because the others are already generated as shown in the previous discussion. For convenience denote by $E_{ij}$ the matrix with the polynomial 1 in the entry $(i,j)$ and $0$ elsewhere, and write $\mathbb{I}$ for the identity matrix in $\mathbb{F}_q[t]_{\lambda}$. We have generators of the form $\sum_{j=1}^k f_j(t)E_{jj}$ for $f_j(t) \in \mathbb{F}_q[t]_{\lambda_j}$ and of the form $\mathbb{I} + E_{i,i+1}$ for $i = 1, \dots, k-1$.
	
	For every $i >j$ and $f(t) \in \mathbb{F}_q[t]_{\lambda_j}$, the matrix $\mathbb{I} + E_{ij}$ can be written as the iterated commutator 
	$$
	\mathbb{I} + E_{ij} = [[\dots[[\mathbb{I}+E_{i,i+1},\mathbb{I}+E_{i+1,i+2}],\mathbb{I}+E_{i+2,i+3}]\dots],\mathbb{I}+E_{j-1,j}],
	$$
	where $[x,y] = x^{-1}y^{-1}xy$. For $f(t)$ invertible, the element $\mathbb{I}+f(t)E_{ij}$ can be obtained by conjugation as follows:
	$$
	\mathbb{I}+f(t)E_{ij} = (\mathbb{I}+(f(t)-1)E_{ii})(\mathbb{I}+E_{ij})(\mathbb{I}+(f(t)-1)E_{ii})^{-1}.
	$$
	If $f(t)$ is not invertible, then it can be written as $f(t)=f_1(t)+f_2(t)$ with $f_1(t)$ and $f_2(t)$ invertible (it is sufficient to choose $f_1$ and $f_2$ with non-zero constant term), so $\mathbb{I}+f(t)E_{ij} = (\mathbb{I}+f_1(t)E_{ij})(\mathbb{I}+f_2(t)E_{ij})$. A generic element of the form (\ref{Formy}) can be written as
	$$
	\mathbb{I}+\sum_{i<j}f_{ij}(t)E_{ij} = \left(\prod_{i=1}^{k-1} (\mathbb{I}+f_{ik}(t)E_{ik})\right)\prod_{i=2}^{k-1}\left(\prod_{j=1}^{i-1} (\mathbb{I}+f_{k-i,k-j}(t)E_{k-i,k-j}) \right).
	$$
	It is trivial to get a generic upper triangular matrix starting from an upper unitriangular matrix and a diagonal matrix. Finally, to get a generic matrix we proceed inductively on the dimension. A matrix of dimension 2 can be obtained as follows:
	\begin{eqnarray} \label{above}
	\left(\begin{array}{cc} a & b\\ c & d \end{array}\right) = \frac{1}{ad} \left(\begin{array}{cc} ad-bc & ab\\ 0 & ad \end{array}\right)\left(\begin{array}{cc}a & 0\\ c & d \end{array}\right)
	\end{eqnarray}
	under the hypothesis $ad \neq 0$ (otherwise we work with the conjugate by the matrix $\scriptsize{\left( \begin{array}{cc} 0 & 1 \\ 1 & 0 \end{array}\right)}$). Thus we proceed by induction using Equation (\ref{above}) with block matrices and using the conjugation mentioned above.
\end{enumerate}
\begin{remark} \label{differentUni}
Write $x = 1 + N$, where 1 is the identity matrix in $\GL(n,q)$ and $N$ is the \emph{nilpotent part} of $x$. Given a nonzero $\mu \in \mathbb{F}_q^*$, the previous computation gives the same result if we replace $N$ by $\mu N$, so the Jordan blocks of $x$ have the form
	$$
	\left(\begin{array}{cccc}
	1 & \mu & & 0\\
	& \ddots & \ddots & \\
	& & \ddots & \mu \\
	& & & 1
	\end{array}\right).
	$$
	In fact, for each matrix $y$ the following holds:
	$$
	(1+\mu N)y = y(1+\mu N) \Leftrightarrow \mu Ny = y\mu N \Leftrightarrow Ny = yN \Leftrightarrow (1+N)y = y(1+N) \Leftrightarrow xy = yx,
	$$
	so the elements of the form $1+\mu N$ for $\mu \in \mathbb{F}_q^*$ have the same centralizer. This will be useful in the next section.
\end{remark}

\subsection{Centralizer of a generic element}
Using the result of the two previous sections, we can compute the centralizer of an arbitrary $x \in \GL(n,q)$ in its Jordan form.
\begin{theorem}   \label{GLCentralizer}
Let $x \in \GL(n,q)$ with minimal polynomial $f_1(t)^{e_1} \cdots f_h(t)^{e_h}$, with $f_i(t) \in \mathbb{F}_q[t]$ irreducible of degree $d_i$. Suppose $X$ has Jordan form
$$
x = \bigoplus_{i=1}^h ( B_{\lambda_{i,1}}^{\oplus l_{i,1}} \oplus \cdots \oplus B_{\lambda_{i,k_i}}^{\oplus l_{i,k_i}}),
$$
where $B_{\lambda_{i,j}}$ is the Jordan block of dimension $\lambda_{i,j}$ relative to $f_i(t)$, with $\lambda_{i,1} < \cdots < \lambda_{i,k_i} = e_i$. Then $C_G(X) = U \rtimes R$, where $R \cong \prod_{i=1}^h \left( \prod_{j=1}^{k_i} \GL(l_{i,j},q^{d_i})\right)$ and $|U| = q^{\gamma}$ with 
$$
\gamma = \sum_{i=1}^h d_i\left( 2\sum_{j<\ell} \lambda_{i,j}l_{i,j}l_{i,\ell} + \sum_j (\lambda_{i,j}-1)l_{i,j}^2\right).
$$
\end{theorem}
\begin{proof}
Let $x_i$ be the restriction of $x$ to $\ker(f_i(x)^{e_i})$. By Lemma \ref{SplitCentr},
$$
C_{\GL(n,q)}(x) = \prod_{i=1}^h C_{G_i}(x_i),
$$
where $G_i \cong \GL(e_i,q^{d_i})$. Let $x_i = s_iu_i$ be the Jordan decomposition of $x_i$. Applying Theorem \ref{largerfield}, we deduce that $C_{G_i}(x_i) = C_{C_{G_i}(s_i)}(u_i)$, and so is isomorphic to $C_{\GL(e_i,q^{d_i})}(\tilde{u}_i)$, where $u_i$ is the embedding into $\GL(e_id_i,q)$ of a unipotent $\tilde{u}_i \in \GL(e_i,q^{d_i})$. Now, using Theorem \ref{Card} and Remark \ref{differentUni}, we find that $C_{G_i}(x_i) = U_i \rtimes  R_i$ with $R_i = \prod_{j=1}^{k_i} \GL(l_{i,j},q^{d_i})$ and $|U_i| = {(q^{d_i})}^{\gamma_i}$, where
$$
\gamma_i = 2\sum_{j<\ell} \lambda_{i,j}l_{i,j}l_{i,\ell} + \sum_j (\lambda_{i,j}-1)l_{i,j}^2.
$$
The claim follows from the fact that $\prod_i \left( U_i \rtimes R_i \right) = \left(\prod_i U_i \right) \rtimes \left(\prod_i R_i \right)$.
\end{proof}

\section{Centralizers and conjugacy classes in $\SL(n,q)$}     \label{SLconjclasses}
For completeness we include the solutions for the three problems for the special linear group. Some of the results are well-known (see, for example, \cite[\S 3.2.3]{PGPG} for the conjugacy problem), so we just state them and sketch the proofs. We discuss in more detail how to generate the centralizer of an arbitrary element in $\SL(n,q)$.
\begin{theorem} \label{MainResult2}
Write $S=\SL(n,q)$ and $G=\GL(n,q)$. Let $x \in \SL(n,q)$ with minimal polynomial $f_1(t)^{e_1} \cdots f_h(t)^{e_h}$, where $f_i(t) \in \mathbb{F}_q[t]$ is irreducible of degree $d_i$. Suppose $x$ has Jordan form
$$
x = \bigoplus_{i=1}^h (B_{\lambda_{i,1}}^{l_{i,1}} \oplus \cdots \oplus B_{\lambda_{i,k_i}}^{l_{i,k_i}}),
$$
where $B_{\lambda_{i,j}}$ is the Jordan block of dimension $\lambda_{i,j}$ relative to $f_i(t)$, with $\lambda_{i,1} < \cdots < \lambda_{i,k_i} = e_i$. Now $C_S(x) = U \rtimes (R \cap S)$, where $U$ and $R$ are as in Theorem $\ref{GLCentralizer}$. Moreover
$$
|C_S(x)| = \frac{d}{q-1}|C_G(x)|
$$
with $d = \mathrm{gcd}(\lambda_{1,1}, \dots, \lambda_{h,k_h}, q-1)$.
\end{theorem}
\begin{proof}
Every element of $U$ has determinant 1, so we only need to compute $R \cap S$. Every $x \in R$ can be identified by $(\widetilde{x}_{i,j})$ for $\widetilde{x}_{i,j} \in \GL(l_{i,j},q^{d_i})$, and a straightforward computation shows that $\det{x} = \prod_{i,j} \N_i (\det{\widetilde{x}_{i,j}})^{\lambda_{i,j}}$, where $\N_i$ is the norm $\mathbb{F}_{q^{d_i}}^* \rightarrow \mathbb{F}_q^*$. Let $\omega$ be a primitive element for $\mathbb{F}_q$. We write $\N_i(\det{\widetilde{x}_{i,j}})$ as $\omega^{a_{i,j}}$ for some $a_{i,j} \in \mathbb{Z}_{q-1}$. Note that, by surjectivity of the norm, $a_{i,j}$ assumes all values in $\mathbb{Z}_{q-1}$ when $\widetilde{x}_{i,j}$ runs over $\GL(l_{i,j},q^{d_i})$. Let $k = \sum_i k_i$. The index $|C_G(x):C_S(x)|$ is equal to the index in $\mathbb{Z}_{q-1}^k$ (as an additive group) of the kernel of the homomorphism defined by $(a_{1,1}, \dots, a_{h,k_h}) \rightarrow \sum_{i,j}a_{i,j}\lambda_{i,j}$, that is equal to $(q-1)/d$.
\end{proof}
\begin{theorem}
Let $x \in \SL(n,q)$ and let $d = \gcd(\lambda_1, \dots, \lambda_k,q-1)$, where $\lambda_1, \dots, \lambda_k$ are the dimensions of the Jordan blocks of $x$. Let $\omega$ be a primitive element of $\mathbb{F}_q$ and let $z \in \GL(n,q)$ have determinant $\omega$. The conjugacy class of $x$ in $\GL(n,q)$ splits into $d$ distinct classes in $\SL(n,q)$, whose representatives are $x, x^z, x^{z^2}, \dots, x^{z^{d-1}}$. Moreover, $x^{y_1}$ and $x^{y_2}$ are conjugate in $\SL(n,q)$ if, and only if, $\det{(y_1^{-1}y_2)}$ is a power of $\omega^d$ in $\mathbb{F}_q$.
\end{theorem}
\begin{proof}
The previous theorem implies
\begin{eqnarray}  \label{splitsGLSL}
|x^S| = \frac{|S|}{|C_S(x)|} = \frac{|G|/(q-1)}{d|C_G(x)|/(q-1)} = \frac{|G|}{d|C_G(x)|} = \frac{|x^G|}{d}.
\end{eqnarray}
Now, if $\det{(y_1^{-1}y_2)}$ is a power of $\omega^d$ in $\mathbb{F}_q$, then there exists $z$ in the centralizer of $x$ such that $\det{z}^{-1} = \det{(y_1^{-1}y_2)}$. It follows that $y = y_1^{-1}zy_2 \in S$ and $(x^{y_1})^y = x^{y_2}$.
\end{proof}
The generation of $C_S(x)$ is similar to that of $C_G(x)$ described in Section \ref{Cue}.
\begin{enumerate}
\item 
Let us describe how to generate $R \cap S$. For each $i,j$, let $\widetilde{G}_{i,j}$ be the unique subgroup of $\GL(l_{i,j},q^{d_i})$ of index $q-1$ and let $G_{i,j}=\varphi(\widetilde{G}_{i,j})$, where $\varphi$ is the homomorphism between $\prod_{i=1}^h \left( \prod_{j=1}^{k_i} \GL(l_{i,j},q^{d_i})\right)$ and $R$. The direct product of the $G_{i,j}$ can be generated using the generators of the $\widetilde{G}_{i,j}$. Now take $\widetilde{H}_{i,j} \in \GL(l_{i,j},q^{d_i})$ having determinant of maximum order and define $H_{i,j} = \varphi(\widetilde{H}_{i,j})$ and $\omega^{d_{i,j}}=\det(H_{i,j})$. Let $k = \sum_i k_i$ and let $Z(0)$ be the subgroup of $\mathbb{Z}_{q-1}^k$ of the solutions $(x_{1,1}, \dots, x_{h,k_h})$ of the equation $\sum_{i,j} x_{i,j}d_{i,j}=0$ in $\mathbb{Z}_{q-1}$. If $Z(0)$ is generated by $\{ (a_{\mu,1,1}, \dots, a_{\mu,h,k_h}) \, : \, 1 \leq \mu \leq r\}$, then we add to the generating set for $R \cap S$ the matrices
$$
\bigoplus_{i,j} H_{i,j}^{a_{\mu,i,j}}, \quad 1 \leq \mu \leq r.
$$

\item Step 2 is identical to that in the case $C_G(x)$.

\item Following the argument for the case $C_G(x)$, we need to prove that every element of determinant 1 of the form $\binom{f \: g}{h \: l}$ can be obtained by using elementary operations on the elements of the form $\binom{f \: 0}{0 \: 1}$ and elements of $\SL(2,q)$. Recall that $f,g,h,l$ are polynomials in $\mathbb{F}_q[t]_{\lambda}$ and the determinant depends only on the constant terms. For convenience, write a polynomial as $f = f_0+f_1$, where $f_0 \in \mathbb{F}_q$ is the constant term and $f_1 \in t\mathbb{F}_q[t]_{\lambda}$.

The element
$$
\left(\begin{array}{cc} 1 & \\ & f\end{array}\right)
$$
with $f_0 =1$ is obtained by conjugating $\binom{f \: 0}{0 \: 1}$ by the matrix $\binom{0 \: -\!1}{1 \:\: 0} \in \SL(2,q)$.\\
For $f=f_0+f_1$,
$$
\left(\begin{array}{cc} 1 & f\\ 0 & 1\end{array}\right) = \left(\begin{array}{cc} 1 & -1\\ 0 & 1\end{array}\right) \left(\begin{array}{cc} 1+f_1 & 0\\ 0 & 1\end{array}\right)\left(\begin{array}{cc} 1 & 1\\ 0 & 1\end{array}\right)\left(\begin{array}{cc} (1+f_1)^{-1} & 0\\ 0 & 1\end{array}\right)\left(\begin{array}{cc} 1 & f_0\\ 0 & 1\end{array}\right)
$$
and
$$
\left(\begin{array}{cc} f & -1\\ 1 & 0\end{array}\right) = \left(\begin{array}{cc} 1 & f\\ 0 & 1\end{array}\right)\left(\begin{array}{cc} 0 & -1\\ 1 & 0\end{array}\right).
$$
For polynomials $f,g,l$,
$$
\left(\begin{array}{cc} f & g\\ 0 & l\end{array}\right) = \left(\begin{array}{cc} 1 & gl^{-1}\\ 0 & 1\end{array}\right)\left(\begin{array}{cc} 1+f_1/f_0 & 0\\ 0 & 1\end{array}\right)\left(\begin{array}{cc} 1 & 0\\ 0 & 1+l_1/l_0\end{array}\right)\left(\begin{array}{cc} f_0 & 0\\ 0 & l_0\end{array}\right).
$$
Finally, a generic element $\binom{f \: g}{h \: l}$ with $f_0 \neq 0$ and $f_0l_0-h_0g_0 = 1$ can be written as
$$
\left(\begin{array}{cc} f & g\\ h & l\end{array}\right) = \left(\begin{array}{cc} f & 0\\ h & l-ghf^{-1}\end{array}\right)\left(\begin{array}{cc} 1 & gf^{-1}\\ 0 & 1\end{array}\right).
$$

\item By contrast with $C_G(x)$, we cannot obtain the elements of the form $\mathbb{I}+f(t)E_{ij}$ by conjugation, because we do not have the elements $\mathbb{I} +(f(t)-1)E_{ii}$ (the determinant is not 1 usually), so we need to add generators of the form $\mathbb{I}+\omega^jE_{i,i+1}$ for $1 \leq j \leq \deg{\mathbb{F}_q}$.
\end{enumerate}

\newpage
\chapter{Membership of classical groups}   \label{membershipChapter}
In Chapter \ref{LinearChapter} we described centralizers and conjugacy classes in linear groups. We now focus on the (remaining) classical groups: symplectic, orthogonal and unitary. Let $\C$ be a classical group. The first step in listing conjugacy classes of $\C$ is to determine which classes of $\GL(V)$ have elements in $\C$. This follows the approach of  Britnell \cite[\S 5.1]{JBRIT} and Milnor ${\cite[\S 3]{Milnor}}$.
\section{Membership of classical groups} \label{membership}
Let $F = \mathbb{F}_q$ in the orthogonal and symplectic case, $F = \mathbb{F}_{q^2}$ in the unitary case. Let $\lambda \mapsto \overline{\lambda}$ be a field automorphism of order 1 (in symplectic and orthogonal case) or 2 (in unitary case). Let $V \cong F^n$ be a vector space.

Let $f(t) = a_0 + a_1t+ \cdots + a_{d-1}t^{d-1}+t^d \in F[t]$ be monic with $a_0 \neq 0$. Denote by $\overline{f}(t)$ the polynomial $\overline{a_0} + \overline{a_1}t + \cdots + \overline{a_{d-1}}t^{d-1}+t^d$. The \textit{dual polynomial}\index{dual polynomial} $f^*$ of $f$ is defined by \index{$f^*$}
$$
f^*(t) := \overline{a_0}^{-1}t^d\overline{f}(t^{-1}). \label{dualpolsymbol}
$$
If $\lambda_1, \dots, \lambda_d$ are the roots of $f$ (in an extension of $F$), then the roots of $f^*$ are $\overline{\lambda_1}^{-1}, \dots, \overline{\lambda_d}^{-1}$. An easy computation shows that $(f^*)^* = f$ and $(fg)^* = f^*g^*$ for monic polynomials $f$ and $g$. In particular, $f$ is irreducible if, and only if, $f^*$ is.

The next two results are essentially \cite[5.1 and 5.2]{JBRIT}. They investigate the elementary divisors and the $F[t]$-module structure of $V$ induced by an element of a classical group. The statements are proved for sesquilinear forms, but they can be easily extended to quadratic forms of even characteristic by recalling that $\Or(n,q) \leqslant \Sp(n,q)$ for $q$ even. 
\begin{theorem} \label{necessarycondition}
	Let $\beta$ be a non-degenerate alternating, hermitian or symmetric form on $V$. Let $X \in \C(\beta)$ and let $U$ be an $X$-invariant subspace of $V$ with $U \cap U^{\bot}=\{0\}$. If $X$ has minimal polynomial $f$ on $U$, then $f=f^*$.
\end{theorem}
\begin{proof}
	First observe that for every $i \in \mathbb{N}$, $UX^i = U$, since $X_{|U}$ is bijective. Moreover, $\alpha U = U$ for every $\alpha \in F^*$. Now $\beta(uX,v)=\beta(u,vX^{-1})$ for every $u,v \in U$. Write $f(t) = \sum_{i=0}^d a_it^i$, with $a_0 \neq 0$ and $a_d=1$. By definition of the minimal polynomial, $Uf(X)=\{0\}$, so
	\begin{eqnarray*}
		\{0\} & = & \beta(Uf(X),U)\\
		& = & \beta(U(\sum_i a_iX^i),U) \\
		& = & \beta(U,U\sum_i \overline{a_i}X^{-i}) \\
		& = & \beta(U,\overline{a}_0^{-1}UX^d\sum_i \overline{a_i}X^{-i}) \\
		& = & \beta(U,Uf^*(X)).
	\end{eqnarray*}
	Since $\beta$ is non-degenerate on $U$, the identity $\{0\} = \beta(U,Uf^*(X))$ implies $Uf^*(X)=\{0\}$, so $f$ divides $f^*$. Since $f$ and $f^*$ are both monic of degree $d$, it follows $f=f^*$.
\end{proof}
\begin{theorem} \label{orthogonaldecomposition}
	Let $\beta$ be non-degenerate on $V$ and let $X \in \C(\beta)$. There exists a decomposition
	\begin{eqnarray*}
		V = \bigoplus_i U_i
	\end{eqnarray*}
	where $U_i \bot U_j$ for every $i \neq j$. For every $i$, either $X$ acts cyclically on $U_i$ with minimal polynomial $f^e$ and $f=f^*$ irreducible, or $U_i = W \oplus W^*$ and $X$ acts cyclically on $W$ (resp.\ $W^*$) with minimal polynomial $f^e$ (resp.\ $f^{*e}$) and $f$ irreducible.
\end{theorem}
\begin{proof}
	Let $U$ be one of the summands in the decomposition of $V$ into cyclic submodules given in (\ref{SplitCentr}). Let $X$ act cyclically on $U$ with minimal polynomial $f^e$, where $f$ is irreducible and $e \in \mathbb{N}$. Let $W$ be an $X$-invariant subspace of $V$ such that $V = U \oplus W$, and let $U^* = W^{\bot}$. Since $\beta$ is non-degenerate, $U^* \cap U^{\bot}= (U^{\bot} \cap W^{\bot}) = (U \oplus W)^{\bot}=V^{\bot} = \{0\}$, so for each non-zero $v \in U^*$ there exists $u \in U$ such that $\beta(u,v) \neq 0$. Moreover, for every non-zero $u \in U$, there exists $v \in U^*$ such that $\beta(u,v) \neq 0$ (otherwise $u \in U^{*\bot} = W$, a contradiction). Let us distinguish two cases.
	\begin{enumerate}
		\item $U \cap U^* \neq \{0\}$. Since $U \cap U^*$ is $X$-invariant, $Uf(X)^{e-1} \subseteq U \cap U^*$ (because $X$ acts cyclically on $U$ with minimal polynomial $f^e$). If $U \cap U^{\bot} \neq \{0\}$, then for the same reason $Uf(X)^{e-1} \subseteq U \cap U^{\bot}$, so $Uf(X)^{e-1} \subseteq U^* \cap U^{\bot} = \{ 0 \}$, a contradiction. Thus $U \cap U^{\bot} = \{0\}$ and $\beta$ is non-degenerate on $U$, so $f=f^*$ by Theorem \ref{necessarycondition}.
		\item $U \cap U^{\bot} \neq \{0\}$, so $U \cap U^*=\{0\}$ by previous case. We want to show that $\beta$ is non-degenerate on $U \oplus U^*$, namely for every $u \in U$ and $v \in U^*$ with $u+v \neq 0$ there exists $z \in U \oplus U^*$ with $\beta(u+v,z) \neq 0$. If $v=0$, take $z \in U^*$ such that $\beta(u,z) \neq 0$; if $v \neq 0$ and $u \in U^{\bot}$, take $z \in U$ such that $\beta(z,v) \neq 0$ (it exists in both cases since $U^{\bot} \cap U^* = \{0\}$). Suppose $u \notin U^{\bot}$ and $v \neq 0$. Let $a = \min\{n \in \mathbb{N} \, | \, uf(X)^n \in U^{\bot}\}$. By our choice of $u$, $1 \leq a \leq e$. Consider
		$$
		(u+v)f(X)^a = uf(X)^a+vf(X)^a.
		$$
		If $uf(X)^a=0$, then $uf(X)^{a-1} \in \ker{f(X)} = Uf(X)^{e-1} \subseteq U \cap U^{\bot}$, contradicting the minimality of $a$. If $vf(X)^a=0$, then there exists $w \in U^*$ such that $\beta(uf(X)^a,w) \neq 0$, so $\beta((u+v)f(X)^a,w) \neq 0$, and choose $z = w\overline{f}(X^{-1})^a$. If $vf(X)^a \neq 0$, then $vf(X)^a \in U^*$ and there exists $w \in U$ such that $\beta(w,v\overline{f}(X^{-1})^a) \neq 0$. Since $uf(X)^a \in U^{\bot}$, it follows that $\beta((u+v)f(X)^a,w) \neq 0$, so again choose $z = w{\overline{f}}(X^{-1})^a$.
		
		Now if $\beta$ is non-degenerate on $U$, then choose $U_1=U$. The minimal polynomial of $X$ on $U$ is $f^e$, and it satisfies $f=f^*$ by Theorem \ref{necessarycondition}.
		
		If $\beta$ is non-degenerate on $U \oplus U^*$, then choose $U_1 = U \oplus U^*$. For every $n$, 
		\begin{eqnarray*}
			\beta(Uf(X)^n,U^*)=\{0\} & \Leftrightarrow & \beta(U,U^*f^*(X)^n) = \{0\}\\
			Uf(X)^n = \{0\} & \Leftrightarrow & U^*f^*(X)^n = \{0\},
		\end{eqnarray*}
		using $U^{\bot} \cap U^* = \{0\}$. This proves that the minimal polynomial of $X$ on $U^*$ is $f^{*e}$, and the action of $X$ on $U^*$ is cyclic because $\dim{U} = \dim{U^*}$, so we are in the second case of the hypothesis.
	\end{enumerate}
	Finally, we can write $V = U_1 \oplus U_1^{\bot}$. Since $U_1^{\bot}$ is $X$-invariant, we can repeat the argument on $U_1^{\bot}$, and by induction on $\dim{V}$ we establish the claim.
\end{proof}
From Theorems \ref{necessarycondition} and \ref{orthogonaldecomposition}, we get information about the elementary divisors of $X \in \C(\beta)$ and the $F[t]$-module structure induced on $V$. Given $X \in \C(\beta)$, let $V = \bigoplus_{i=1}^k V_i$ be the decomposition of $V$ into cyclic $F[t]$-submodules described in (\ref{decompCSM}). Here, $X$ acts cyclically on $V_i$ with minimal polynomial $f_i^{e_i}$, where $f_i$ is irreducible and $e_i \in \mathbb{N}$. By Theorem \ref{orthogonaldecomposition}, we can suppose that, for every $i=1, \dots, k$, either $V_i \bot V_j$ for all $j \neq i$ or there exists $i' \neq i$ such that $(V_i \oplus V_{i'}) \bot V_j$ for all $j \neq i,i'$. In the first case, put $i' = i$. With this notation, $f_i = f_{i'}^*$ and $e_i = e_{i'}$. Each of the $U_i$ in Theorem \ref{orthogonaldecomposition} coincides with a certain $V_i$ (if $i =i'$) or with $V_i \oplus V_{i'}$ (if $i \neq i'$). An immediate consequence is that if $f^e$ is an elementary divisor of $X$ of multiplicity $m$, with $f \neq f^*$ and $e$ a positive integer, then $f^{*e}$ is also an elementary divisor of $X$ of multiplicity $m$.
\begin{remark} In general, a direct sum decomposition of $V$ into cyclic submodules $\bigoplus_i V_i$ need not satisfy the property of Theorem \ref{orthogonaldecomposition}. This theorem only ensures that such a decomposition exists.
\end{remark}

Following \cite{Cikunov, FS, Milnor}, we introduce the following notation.
\begin{definition}  \label{ThreeCases}\index{$\Phi$}\index{$\Phi_1, \Phi_2, \Phi_3$}
	Let $F = \mathbb{F}_{q^2}$ in the unitary case or $\mathbb{F}_q$ in the symplectic and orthogonal cases. Define:
	\begin{eqnarray*}
		\Phi_1 & := & \{ f: \, f \in F[t] \; | \; f=f^* \mbox{ monic irreducible}, \, \deg{f}=1 \};\\
		\Phi_2 & := & \{ f: \, f \in F[t] \; | \; f = gg^*, \, g \neq g^*, \, g \mbox{ monic irreducible} \};\\
		\Phi_3 & := & \{ f: \, f \in F[t] \; | \; f=f^* \mbox{ monic irreducible}, \, \deg{f}>1 \}.
	\end{eqnarray*}
	Define $\Phi := \Phi_1 \cup \Phi_2 \cup \Phi_3$. Given $X \in \C(\beta)$, $f \in \Phi$ and $m$ a positive integer, $f^m$ is a \emph{generalized elementary divisor} of $X$ if either $f \in \Phi_1 \cup \Phi_3$ and $f^m$ is an elementary divisor of $X$, or $f \in \Phi_2$, $f=gg^*$ and $g^m$ is an elementary divisor of $X$ (and so ${g^*}^m$ is). \index{generalized elementary divisor}
\end{definition}
\begin{remark} Note that $\Phi_1 = \{ t \pm 1 \}$ in the symplectic and orthogonal case and $\Phi_1$ is the set of polynomials of the form $t-\lambda^{q-1}$ with $\lambda \in F$ in the unitary case.
\end{remark}
\begin{remark} If $f \in \Phi_3$, then $f$ has even degree in the symplectic and orthogonal cases and odd degree in the unitary case. In the symplectic and orthogonal cases each root $\lambda$ of $f$ can be paired with $\lambda^{-1}$, so the number of roots of $f$ is clearly even ($\lambda \neq \lambda^{-1}$ because $f \neq t\pm 1$).
	
	Suppose we are in the unitary case. The map $\alpha \mapsto \alpha^{-q}$ acts as a permutation on the set $Z(f)$ of the roots of $f$. Consider an orbit $\mathcal{O} = \{ \lambda, \lambda^{-q}, \lambda^{q^2}, \dots, \lambda^{(-q)^r} \}$ under this action and its suborbit $\mathcal{O'} = \{\lambda, \lambda^{q^2}, \lambda^{q^4}, \dots \}$. The polynomial
	$$
	g(t) = \prod_{\mu \in \mathcal{O'}} (t-\mu) = (t-\lambda)(t-\lambda^{q^2})(t-\lambda^{q^4}) \cdots 
	$$
	is a divisor of $f$ and belongs to $\mathbb{F}_{q^2}[t]$ because its coefficients are fixed by the field automorphism $\alpha \mapsto \alpha^{2q}$. Since $f$ is irreducible, the only possibility is that $g=f$ and $\mathcal{O'}=Z(f)$. It follows that $\mathcal{O}$ is the unique orbit in $Z(f)$, and $\mathcal{O}=\mathcal{O'}$ if, and only if, $|\mathcal{O}|$ is odd.
\end{remark}
By analogy with the linear case, our analysis of conjugacy classes and centralizers in classical groups can be simplified to the situation where $X$ has minimal polynomial $f^e$, for $e$ a positive integer, and $f$ satisfies one of the three cases in Definition \ref{ThreeCases}.

Now we state the main result of this section. Recall that matrices $X,Y$ are \textit{similar}\index{similar elements} if they are conjugate in $\GL(V)$, namely if they have the same elementary divisors with the same multiplicity; in such a case we write $X \sim Y$\label{similarSymbol}\index{$x \sim y$}.
\begin{theorem} \label{elementsofCG}
	Let $X \in \GL(V)$. 
	\begin{itemize}
		\item There exists a non-degenerate hermitian form $\beta$ such that $X \in \C(\beta)$ if, and only if, $X \sim X^{*-1}$.
		\item There exists a non-degenerate alternating form $\beta$ such that $X \in \C(\beta)$ if, and only if, $X \sim X^{-1}$ and every elementary divisor $(t\pm 1)^{2k+1}$ with $k$ a non-negative integer has even multiplicity.
		\item If $q$ is odd, then there exists a non-singular quadratic form $Q$ such that $X \in \C(Q)$ if, and only if, $X \sim X^{-1}$ and every elementary divisor $(t\pm 1)^{2k}$ with $k$ a positive integer has even multiplicity. Moreover, in even dimension, $Q$ can be of both plus type and minus type if, and only if, $X$ has at least one elementary divisor $(t \pm 1)^{2k+1}$ for some non-negative integer $k$. If this is not the case, then $Q$ has plus type (resp.\ minus type) if, and only if, $\sum_{f,e}e m(f^e)$ is even (resp.\ odd), where the sum runs over all polynomials $f\in \Phi_3$ and $e \in \mathbb{N}^+$, and $m(f^e)$ is the multiplicity of $f^e$ as an elementary divisor of $X$.
		\item If $q$ is even, then there exists a non-singular quadratic form $Q$ such that $X \in \C(Q)$ if, and only if, $X \sim X^{-1}$ and every elementary divisor $(t+ 1)^{2k+1}$ with $k$ a non-negative integer has even multiplicity. Moreover, $Q$ can be of both plus type and minus type if, and only if, $X$ has at least one elementary divisor $(t+1)^k$ for some positive integer $k$. If this is not the case, then $Q$ has plus type (resp.\ minus type) if, and only if, $\sum_{f,e}e m(f^e)$ is even (resp.\ odd), where the sum runs over all polynomials $f\in \Phi_3$ and $e \in \mathbb{N}^+$, and $m(f^e)$ is the multiplicity of $f^e$ as an elementary divisor of $X$.
	\end{itemize}
\end{theorem}
\textit{}\\
The proof of this theorem requires several steps. We first prove sufficiency.

We have already proved in Theorem \ref{necessarycondition} that if $X \in \C(\beta)$ or $\C(Q)$, then $X \sim X^{*-1}$. So, we need to prove that if $X$ is an element of a symplectic group (resp.\ orthogonal group of odd characteristic), then every elementary divisor $(t \pm 1)^{2k+1}$ (resp.\ $(t \pm 1)^{2k}$) has even multiplicity.\\

In the symplectic case, let $\beta$ be a non-degenerate alternating form and let $k$ be a fixed non-negative integer. Let $U_1, \dots, U_{\ell}$ be the $F[t]$-submodules of $V$ on which $X$ acts with minimal polynomial $(t-1)^{2k+1}$, and $\ell$ is the multiplicity of $(t-1)^{2k+1}$ as an elementary divisor of $X$. Let $U = U_1 \oplus \cdots \oplus U_{\ell}$. If $\ell$ is odd, then by Theorem \ref{orthogonaldecomposition} $\beta_U$ would be a non-degenerate alternating form on $U$, that is impossible because $U$ has odd dimension. The case $(t+1)^{2k+1}$ is analogous. If $q$ is even, then $\Or(n,q) \leqslant \Sp(n,q)$, so the same condition holds for elements of $\Or(n,q)$.

Now consider the orthogonal case in odd characteristic. Let $\beta$ be a non-degenerate symmetric form and $k$ a fixed positive integer. Similar to the symplectic case, if $U_1, \dots, U_{\ell}$ are the $F[t]$-submodules of $V$ on which $X$ acts with minimal polynomial $(t-1)^{2k}$, then $\beta$ must be non-degenerate on $U = U_1 \oplus \cdots \oplus U_{\ell}$. By Theorem \ref{orthogonaldecomposition}, we can suppose that for every $i$, either $U_i \bot U_j$ for all $j \neq i$, or there exists $i'$ such that $(U_i \oplus U_{i'}) \bot U_j$ for all $j \neq i,i'$. So, if $\ell$ is odd, there must exist at least one $U_i$ such that $\beta_{|U_i}$ is non-degenerate, and we can consider just the case $\ell=1$ to exclude the odd multiplicity. The following proposition, due to Huppert \cite[2.2]{Huppert}, concludes the proof of the sufficiency in Theorem \ref{elementsofCG} (the case $(t+1)^{2k}$ is analogous).
\begin{proposition}\label{Prop414}
	If $X$ acts cyclically on $V$ with minimal polynomial $(t-1)^{2k}$, then there are no non-degenerate symmetric forms $\beta$ such that $X \in \C(\beta)$.
\end{proposition}
\begin{proof}
	Put $m=2k$ and choose a basis $v_1, \dots, v_m$ for $V$ such that $v_1X=v_1$ and $v_iX = v_{i-1}+v_i$ for all $i=2, \dots, m$. Suppose by contradiction that $\beta$ is non-degenerate on $V$. Hence $\langle v_1 \rangle^{\bot}$ is an invariant $F[t]$-submodule of $V$ of dimension $m-1$, and the only possibility is $\langle v_1 \rangle^{\bot} = \langle v_1, \dots, v_{m-1} \rangle$; this implies $\beta(v_1,v_m) \neq 0$. Now,
	\begin{eqnarray*}
		0 \neq \beta(v_1,v_m) & = & \beta(v_1,v_1(X-1)^{m-1}) = \beta(v_1(X^{-1}-1)^{m-1},v_1)\\
		& = & \beta((-1)^{m-1}v_1X^{-m+1}(X-1)^{m-1},v_1)\\
		& = & (-1)^{m-1} \beta(v_mX^{-m+1},v_1)\\
		& = & (-1)^{m-1}\beta(v_m,v_1)\\
		& = & -\beta(v_1,v_m),
	\end{eqnarray*}
	that is impossible in odd characteristic.
\end{proof}
Now let us prove the reverse implication in Theorem \ref{elementsofCG}. Recall from Remark \ref{sumofforms} that if $\beta_1, \dots, \beta_k$ are sesquilinear forms of the same type and $X_i \in \C(\beta_i)$, then $X_1 \oplus \cdots \oplus X_k \in \C(\beta_1 \oplus \cdots \oplus \beta_k)$, and the same holds for quadratic forms. So, to prove that there exists a sesquilinear or quadratic form preserved by $X$, we suppose that $X$ has a unique \ged of the type described in Definition \ref{ThreeCases}. Write simply $\C$ for $\C(\beta)$ or $\C(Q)$. \\ \label{paginacasi}
\\
\textbf{Case 1:} $f \in \Phi_1$. Now $X$ has a unique (generalized) elementary divisor of the form $(t-\lambda)^e$ of multiplicity $m$, for $\lambda \in F$ and $e,m$ positive integers. Clearly it is sufficient to consider the case $\lambda=1$ because if $X \in \C$ acts with minimal polynomial $(t-1)^e$, then $\lambda X$ acts with minimal polynomial $(t-\lambda)^e$ and belongs to $\C$. We exhibit explicit unipotent elements of $\C$ in Chapter \ref{unipotentChap}.

\textit{}\\
\textbf{Case 2:} $f \in \Phi_2$, $f=gg^*$. Let $d= \deg{g}$. In an appropriate basis, $X$ has block diagonal matrix $X = Y \oplus Y^{*-1}$, with $Y$ a Jordan block relative to $g^e$. It is immediate that $X$ is an isometry for the form with matrix $B$ or, in the orthogonal case, for the quadratic form of matrix $A$ with
$$
B = \left(\begin{array}{cc} \mathbb{O} & \mathbb{I} \\ \varepsilon \mathbb{I} & \mathbb{O} \end{array}\right), \quad A = \left(\begin{array}{cc} \mathbb{O} & \mathbb{I} \\ \mathbb{O} & \mathbb{O} \end{array}\right),
$$
where $\mathbb{O}$\index{$\mathbb{O}$, $\mathbb{O}_n$} is the $de \times de$ zero matrix, $\mathbb{I}$ is the $de \times de$ identity matrix, and $\varepsilon = -1$ in the symplectic case and $1$ otherwise. 

Note that $V= U \oplus W$, where $X$ acts cyclically on $U$ (resp.\ $W$) with minimal polynomial $g^e$ (resp.\ $g^{*e}$), and $U$ and $W$ are totally isotropic.\\
\\
\textbf{Case 3:} $f \in \Phi_3$. We extend \cite[Thm.\ 5.4]{JBRIT} to the unitary case and even characteristic. Let $X$ have a unique (generalized) elementary divisor $f^e$, with $f=f^*$ irreducible of degree $d$. In this case $X$ acts cyclically on $V$ with minimal polynomial $f^e$. Let $U_1,U_2 \cong V$ and write $U = U_1 \oplus U_2$. Using the same argument as in Case 2, where we did not use the fact $f \neq f^*$, we can suppose that there exists a non-degenerate form $\beta$ on $U$ such that $X \in \C(\beta)$ and $U_1, U_2$ are totally isotropic. Now suppose there exists $v \in U$ and $i \in \mathbb{Z}$ such that $\beta(v,vX^if(X)^{e-1}) \neq 0$. Let $W$ be the cyclic $F[t]$-submodule of $U$ generated by $v$. The minimal polynomial of $X$ on $W$ is $f^h$ for some $h \leq e$. If $h < e$, then $vf(X)^{e-1}=0$, so $\beta(v,vX^if(X)^{e-1})=\beta(v,0)=0$, a contradiction. Moreover, $W$ is non-degenerate. Assume not: if $vg(X) \in W^{\bot}$ for some $g \in F[t]$, then write $g(t) = f(t)^m g'(t)$ for $m \leq e-1$ and $g'$ coprime to $f$. Since $W \cap W^{\bot}$ is $X$-invariant, $vg(X)f(X)^{e-1-m}X^i \in W \cap W^{\bot}$, but in this case $0 \neq \beta(v, vX^if(X)^{e-1}) = \beta(vg'(X),vg(X)f(X)^{e-1-m}X^i)=0$, a contradiction. Thus $X \in \C(\beta_{|W})$, but $W$ is isomorphic to $V$ as $F[t]$-module, so such a form must exist on $V$.

Therefore, the case $f \in \Phi_3$ is solved except when $\beta(v,vX^if^{e-1}(X))=0$ for every $v \in V$ and $i \in \mathbb{Z}$. Suppose this is the case. Write $v = v_1+v_2$ with $v_i \in U_i$. By Case 2, we can assume that $U_1$ and $U_2$ are totally isotropic, so the condition is equivalent to
$$
\beta(v_1,v_2X^if(X)^{e-1})+\beta(v_2,v_1X^if(X)^{e-1})=0
$$
for all $v_1 \in U_1$, $v_2 \in U_2$ and $i \in \mathbb{Z}$. By a straightforward computation, this condition implies the following sequence of identities:
\begin{eqnarray*}
	\beta(v_1,v_2[X^if(X)^{e-1}+\varepsilon X^{-i}\overline{f}(X^{-1})^{e-1}]) & = & 0,\\
	\beta(v_1,v_2[X^if(X)^{e-1}+\varepsilon \overline{f(0)}^{e-1}X^{-i-d(e-1)}f(X)^{e-1}]) & = & 0, \quad (\mbox{using } f=f^*)\\
	\beta(v_1,v_2f(X)^{e-1}[X^i+\varepsilon \overline{f(0)}^{e-1}X^{-i-d(e-1)}]) & = & 0,\\
	\beta(v_1,v_2f(X)^{e-1}[X^{d(e-1)+2i}+\varepsilon \overline{f(0)}^{e-1}]) & = & 0.
\end{eqnarray*}
Since $\beta$ is non-degenerate on $U_1 \oplus U_2$, the left hand side must vanish on all occasions. This is the case only if $f(t)$ divides $t^{d(e-1)+2i}+\varepsilon \overline{f(0)}^{e-1}$ for every $i$ such that $d(e-1)+2i$ is non-negative. But $i$ can be chosen such that $d(e-1)+2i \leq 2$, so this never occurs in the unitary case, because $\deg{f} \geq 3$. In the orthogonal and symplectic cases $d$ is even, so choosing $i = 1-d(e-1)/2$ we get $f(t)$ divides $t+\varepsilon \overline{f(0)}^{e-1}$ of degree 1, which is impossible. This concludes the proof of the existence of $\beta$ when $f \in \Phi_3$. In the quadratic case, the argument is the same, by considering the quadratic form $Q$ on the module $W$ and $\beta_Q$ instead of $\beta$.

Before we prove Theorem \ref{elementsofCG} we need some more facts.
\begin{proposition}  \label{typePhi3}
	Let $f \in F[t]$ be an irreducible polynomial with $f= f^*$ and $\deg{f}>1$, and let $X$ be its companion matrix. If $Q$ is a non-singular quadratic form such that $X \in \C(Q)$, then $Q$ has minus type.
\end{proposition}
\begin{proof}
	We will show in the next chapter that $C_{\C(Q)}(X) \cong \U(1,q^{d/2})$, where $d = \deg(f)$. In particular, $|C_{\C(Q)}(X)| = 1+q^{d/2}$. If $Q$ has plus type, then $C_{\C(Q)}(X)$ is a subgroup of $\Or^+(d,q)$. But, with the exception of $(d,q) = (2,3)$ and $(6,2)$, this is impossible because $1+q^{d/2}$ does not divide the cardinality of $\Or^+(d,q)$ (see \cite[Chapter 11]{PG}). The only polynomials for $(2,3)$ and $(6,2)$ are $t^2+1 \in \mathbb{F}_3[t]$ and $t^6+t^3+1 \in \mathbb{F}_2[t]$; in both cases we check directly that their companion matrices preserve only non-singular quadratic forms of minus type.
\end{proof}
The next result is due to Huppert \cite{Huppert}. We just sketch the proof, omitting details.
\begin{proposition}\label{HuppertResults}
	Let $q$ be odd. Let $\beta$ be a non-degenerate symmetric form on $V$ and $X \in \C(\beta)$. Let $U$ be one of the direct summands in the decomposition of $V$ described in Theorem $\ref{orthogonaldecomposition}$. If $U = W \oplus W^*$, where $X$ acts cyclically on $W$ (resp.\ $W^*$) with minimal polynomial $f^e$ (resp.\ $f^{*e}$), then we can suppose that $W$ and $W^*$ are totally isotropic.
\end{proposition}
\begin{proof}
	If $f \neq f^*$, then the proposition follows directly from Lemma \ref{L26w} (whose proof is independent of the results of this chapter). Suppose $f=f^*$. If $f(t) = (t \pm 1)^{2k}$ for some integer $k$, then see \cite[2.4]{Huppert}. In the other cases, from \cite[2.1]{Huppert} there exists an $F[t]$-submodule $W_1$ of $U$ on which $X$ acts with minimal polynomial $f^e$ and such that the restriction of $\beta$ on $W_1$ is non-degenerate. In such a case we can write $U = W_1 \oplus W_2$ with $W_2 = U \cap W_1^{\bot}$, and we can replace $U$ by two distinct direct summands in Theorem \ref{orthogonaldecomposition}.
\end{proof}
\begin{proof}[Proof of Theorem $\ref{elementsofCG}$]
	We have proved the theorem in the symplectic and unitary case. We have also proved that there exists a quadratic form $Q$ such that $X \in \C(Q)$ if, and only if, $X \sim X^{-1}$ and, for every positive integer $k$, the polynomials $(t\pm 1)^{2k}$ (resp.\ $(t\pm1)^{2k-1}$) have even multiplicity as an elementary divisors of $X$ if $q$ is odd (resp.\ even). It remains to prove the last assertion about the type of $Q$. Recall that, if $Q_1, \dots, Q_k$ are quadratic forms of even dimensions, then $\bigoplus_i Q_i$ has plus type (resp.\ minus type) if, and only if, the number of $Q_i$ having minus type is even (resp.\ odd).
	
	Let $V$ have even dimension. In the analysis of Case 1, we proved that if $X$ has elementary divisors $(t \pm 1)^{2k+1}$ (in the case $q$ odd) or $(t \pm 1)^k$ (case $q$ even), then a non-singular quadratic form of both types can be defined on the submodules of $V$ relative to these elementary divisors. Thus $X$ belongs to appropriate copies of both $\Or^+(V)$ and $\Or^-(V)$. If $q$ is odd, then an elementary divisor $(t \pm 1)^{2k}$ does not affect the type of $Q$: in fact, if $U_1, \dots, U_{2m}$ are the cyclic submodules of $V$ on which $X$ acts with minimal polynomial $(t \pm 1)^{2k}$ and $U = \bigoplus_{i=1}^{2m} U_i$, then the restriction of $\beta$ to each $U_i$ is degenerate by Proposition \ref{Prop414}, so, up to index rearrangements, we can suppose that $\beta$ is non-degenerate on $U_i \oplus U_{i+m}$ and $(U_i \oplus U_{i+m}) \bot (U_j\oplus U_{j+m})$ for every $i \neq j$. By Proposition \ref{HuppertResults}, we can suppose that the $U_i$ are totally singular. So $U_1 \oplus \cdots \oplus U_m$ is a totally singular subspace of dimension $\frac{1}{2} \dim{U}$. This proves that every quadratic form on $U$ preserved by $X$ has plus type.
	
	The same argument holds for elementary divisors of the form $f^e$ with $f \neq f^*$. If $U = \ker{f(X)^e}$ and $U^*=\ker{f^*(X)^e}$, then we can assume that $U$ and $U^*$ are totally singular by Proposition \ref{HuppertResults}, so every quadratic form on $U \oplus U^*$ preserved by $X$ has plus type.
	
	Suppose finally that $X$ acts cyclically on $V$ with minimal polynomial $f^e$, with $f = f^*$ irreducible of degree $d$. If $X$ belongs to $\C(Q)$ for some quadratic form $Q$, so does its semisimple part $S$. In such a case $V$ is the direct sum of cyclic $F[t]$-submodules $U_1, \dots, U_e$ on which $S$ acts cyclically with minimal polynomial $f$. By Proposition \ref{HuppertResults} we can assume that $U_i \bot U_j$ for every $i \neq j$ and that every $Q_i$ is non-singular, where $Q_i = Q_{|U_i}$. By Proposition \ref{typePhi3}, every $Q_i$ has minus type, so it is easy to see that $Q$ has plus type (resp.\ minus type) if, and only if, $e$ is even (resp.\ odd).
\end{proof}

\newpage
\chapter{Centralizers and conjugacy classes of semisimple elements}   \label{semisimpleChap}
Our strategy to investigate conjugacy classes in a classical group $\C$ is similar to the linear case, so we analyze separately semisimple and unipotent elements. In this chapter we consider the semisimple case. In Chapter \ref{membershipChapter} we established which conjugacy classes of $\GL(V)$ have elements in $\C$, so to list conjugacy classes of $\C$ it remains to establish whether each class of $\GL(V)$ splits into more classes in $\C$ and to show explicit representatives. This extends the work of Wall \cite{SCCSS} to all groups of isometries.

\section{Semisimple conjugacy classes in classical groups}     \label{semisimpleconjclasses}
\subsection{Conjugacy classes in isometry groups of sesquilinear forms} \label{semisimplesesquilinear}
Let $V \cong F^n$ be a vector space and let $\beta$ be a non-degenerate alternating, symmetric or hermitian form on $V$, where $F = \mathbb{F}_{q^2}$ if $\beta$ is hermitian, $F=\mathbb{F}_q$ otherwise. We assume that $q$ is odd if $\beta$ is symmetric. Let $\C = \C(\beta)$ be the isometry group of $\beta$.

The aim of this section is to decide whether or not two similar semisimple elements of $\C$ are conjugate in $\C$. We use the strategy of Wall \cite{SCCSS} for the symplectic case of odd characteristic and extend it to all sesquilinear forms.
\begin{theorem}\label{MainConjinC}
	Let $\beta$ be a reflexive sesquilinear form, let $\C=\C(\beta)$ be its group of isometries and let $X_1, X_2 \in \C$ be semisimple. If $\C$ is a symplectic or a unitary group, then $X_1$ and $X_2$ are conjugate in $\C$ if, and only if, they are similar. If $\C$ is an orthogonal group, then $X_1$ and $X_2$ are conjugate in $\C$ if, and only if, they are similar and the forms induced by $\beta$ on the eigenspaces of the eigenvalues $+1$ and $-1$ have the same type.
\end{theorem}
The proof of this theorem requires some preliminary work. For every form $\beta$ and every $T \in G:=\GL(n,F)$, define the form $\beta T$ by
$$
\beta T(u,v) := \beta (uT^{-1},vT^{-1})
$$
for all $u,v \in V$. It is easy to check that $\beta (ST) = (\beta S)T$ for all $S,T \in G$ and that $\beta T = \beta$ if, and only if, $T \in \C(\beta)$. For fixed $\beta$ and $X \in \C(\beta)$, consider the set $\mathscr{L}$ of all forms $\gamma$ congruent to $\beta$. The centralizer $C_G(X)$ of $X$ in $G$ acts on $\mathscr{L}$ through the action $(Y,\gamma) \mapsto \gamma Y$, so we consider the set of orbits $L = \{ [\beta_1 = \beta], [\beta_2], \dots, [\beta_h]\}$ of $\mathscr{L}$ under the action of $C_G(X)$.

Also $\C$ acts on the conjugacy class $X^G$ of $X$ in $G$ via the action $(g, X^h) \mapsto X^{hg}$, and we consider the set of orbits $K = \{ [X^{T_1=1}], [X^{T_2}], \dots, [X^{T_k}] \}$ of $X^G$ under the action of $\C$.
\begin{lemma}
	$|L| = |K|$.
\end{lemma}
\begin{proof}
	Define $\mathcal{M} = \{ (\gamma,Y) \in \mathscr{L} \times X^G \, | \, Y \in \mathscr{C}(\gamma) \}$. The group $G$ acts on $\mathcal{M}$ through the action $(T, (\gamma,Y)) \mapsto (\gamma T, Y^T)$. Let $M$ be the set of orbits of $\mathcal{M}$ under the action of $G$. We define two functions $M \rightarrow L$ and $L \rightarrow M$ by
	$$
	\begin{array}{ccc}
	M & \longrightarrow & L \\
	\Delta & \longmapsto & \{ \gamma \in \mathscr{L} \, | \, (\gamma, X) \in \Delta \};
	\end{array} \quad \mbox{ and } \quad \begin{array}{ccc}
	L & \longrightarrow & M \\
	\{ \gamma_1, \dots, \gamma_{\ell} \} & \longmapsto & \mbox{orbit containing } (\gamma_1,X).
	\end{array}
	$$
	These two functions are mutual inverses; hence, there is a bijection between $M$ and $L$, so $|M| = |L|$. A similar argument applied to the functions
	$$
	\begin{array}{ccc}
	M & \longrightarrow & K \\
	\Delta & \longmapsto & \{ Y \in X^G \, | \, (\beta, Y) \in \Delta \};
	\end{array} \quad \mbox{ and } \quad \begin{array}{ccc}
	K & \longrightarrow & M \\
	\{ Y_1, \dots, Y_k \} & \longmapsto & \mbox{orbit containing } (\beta,Y_1).
	\end{array}
	$$
	shows that $|M|=|K|$. In conclusion, $|L|=|M|=|K|$.
\end{proof}
The lemma allows us to rewrite the problem of conjugacy classes in $\C$ in terms of congruence classes in $\mathscr{L}$. Given $X \in \C$, we will analyze all possible forms preserved by $X$ and their congruence classes in $\mathscr{L}$. \\

By Theorem \ref{orthogonaldecomposition}, if $[f_1, \dots, f_h]$ is the list of elementary divisors of $X$, then $[f_1, \dots, f_h] = [f_1^*, \dots, f_h^*]$, so if $f_i \neq f_i^*$, then $f_i$ and $f_i^*$ have the same multiplicity. Hence, the list of elementary divisors of $X$ can be rewritten to ensure $f_i \in \Phi$, as in Definition \ref{ThreeCases}. \\

Let $f(t) = f_1(t)^{e_1} \cdots f_h(t)^{e_h}$ be the minimal polynomial of $X$. By choosing an appropriate basis for $V$, we suppose that $X$ has a shape
$$
\left( \begin{array}{cccc} X_1 & & & \\ & X_2 & & \\ & & \ddots & \\ & & & X_h \end{array}\right),
$$
where $X_i$ is the matrix of the restriction of $X$ to $\ker(f_i(X)^{e_i})$. The matrix $B$ of the form preserved by $X$ is 
$$
\left(\begin{array}{ccc} B_{11} & \cdots & B_{1h} \\ \vdots & \ddots & \vdots \\ B_{h1} & \cdots & B_{hh} \end{array}\right),
$$
where $X_iB_{ij}X_j^* = B_{ij}$, and $B_{ij} = \varepsilon B_{ji}^*$ for every $1 \leq i,j \leq h$, $\varepsilon= -1$ in the symplectic case and $\varepsilon=1$ in the other cases. More generally,
\begin{eqnarray}\label{7w}
g(X_i)B_{ij} = B_{ij}g(X_j^{*-1})
\end{eqnarray}
for every $g \in F[t]$.
\begin{lemma}\label{L26w}
	Let $X_i$ be the block corresponding to the elementary divisor $f_i$. If $f_i \neq f_j^*$, then $B_{ij}=0$.
\end{lemma}
\begin{proof}
	The polynomial $f_i^{e_i}$ is the minimal polynomial of $X_i$, so, taking $f=f_i$ in (\ref{7w}), we get $0 = f_i(X_i)^{e_i}B_{ij} = B_{ij}f_i(X_j^{*-1})^{e_i}$. Since $f_i$ and $f_j^*$ are irreducible, they are coprime, so $f_i(X_j^{*-1})^{e_i}$ is non-singular. Hence $B_{ij}f_i(X_j^{*-1})^{e_i}=0$ implies $B_{ij}=0$.
\end{proof}

The previous lemma implies that the matrix of $B$ has block diagonal shape
$$
\begin{pmatrix} B_1 & & \\ & \ddots & \\ & & B_h \end{pmatrix},
$$
and the problem can be solved separately for individual blocks $X_i$. We may therefore assume that there exists $f \in \Phi$ such that every \ged of $X$ is a power of $f$.
Hence, suppose that $X$ is a semisimple element of $\C(\beta)$, so $e_i=1$ for every $i$.\\
\\
\textbf{Case 1}: $f \in \Phi_1$. The matrix $X$ is a scalar matrix, so it commutes with every other matrix, and forms $B_1$ and $B_2$ preserved by $X$ are congruent in $C_G(X)$ if and only if they are congruent in $G$.\\
\\
\textbf{Case 2}: $f \in \Phi_2$, $f=gg^*$. In such a case, $V = \ker(g(X)) \oplus \ker(g^*(X))$. If $X_1$ is the restriction of $X$ to $\ker(g(X))$ and $X_2$ is the restriction of $X$ to $\ker(g^*(X))$, then $X_2$ is similar to $X_1^{*-1}$ and we can suppose that
\begin{eqnarray} \label{case2}
X = \left( \begin{array}{cc} X_1 & 0 \\ 0 & X_1^{*-1} \end{array}\right), \quad B = \left( \begin{array}{cc} B_{11} & B_{12}\\ \varepsilon B_{12}^* & B_{22} \end{array}\right).
\end{eqnarray}
By Lemma \ref{L26w}, $B_{11} = B_{22}=0$. The identity $XBX^*=B$ implies $X_1B_{12}X_1^{-1}=B_{12}$, so $X_1$ commutes with $B_{12}$. If we take
$$
Y = \left( \begin{array}{cc} B_{12} & 0 \\ 0 & \mathbb{I} \end{array}\right), \quad J = \left(\begin{array}{cc} 0 & \mathbb{I} \\ \varepsilon \mathbb{I} & 0 \end{array} \right),
$$
where $\mathbb{I}$ is the identity matrix of the same dimension as $X_1$, then $Y$ commutes with $X$ and $YJY^* = B$. The matrix $J$ does not depend on $B$, so given forms $B$ and $B'$ preserved by $X$ and $Y,Y'$ constructed as above, $B = (YJY^*) = (YY'^{-1})B'(YY'^{-1})^*$ and $YY'^{-1} \in C_G(X)$. This proves that all forms are mutually congruent and the conjugacy class of $X$ in $G$ remains a unique conjugacy class in $\C$.

To compute the centralizer of $X$ in $\C$, we suppose that the form is $J$ and $X_1$ is a block diagonal matrix
$$
\left( \begin{array}{ccc} C & & \\ & \ddots & \\ & & C\end{array} \right),
$$
where $C$ is the companion matrix of $f$. By Lemma \ref{SplitCentr}, $Y \in C_G(X)$ has the shape
$$
\left( \begin{array}{cc} Y_1 & 0 \\ 0 & Y_2 \end{array}\right),
$$
where $Y_1$ commutes with $X_1$ and $Y_2$ with $X_1^{*-1}$. The centralizer of $X_1$ is isomorphic to $\GL(h,q^d)$, where $h$ is the multiplicity of $f$ as a \ged of $X_1$ and $d = \deg(f)$. The condition $Y \in \C$ implies $YBY^*=B$, so, by a straightforward calculation, $Y_2 = Y_1^{*-1}$. Hence $Y_1$ can be an arbitrary element of the centralizer of $X_1$ and it determines uniquely $Y_2$. Thus, if $f \in \Phi_2$, then $C_{\C}(X) \cong \GL(h,q^{2d})$ (in the unitary case) or $\GL(h,q^d)$ (in the symplectic and orthogonal case).\\
\\
\textbf{Case 3}: $f \in \Phi_3$. If the unique \ged $f$ has multiplicity $h$, then we assume that $X$ has a block diagonal shape
$$
\left( \begin{array}{ccc} R & & \\ & \ddots & \\ & & R\end{array}  \right),
$$
where $R$ has minimal polynomial $f$ (e.g.\ $R$ can be the companion matrix of $f$) and appears $h$ times. By Theorem \ref{largerfield}, the centralizer of $X$ in $G$ is isomorphic to $\GL(h,\mathbb{F}_{q^d})$, with $d = \deg{f}$ (in the orthogonal and symplectic case) or $d=2\deg{f}$ (in the unitary case), and $Y$ commutes with $X$ if and only if $Y$ is non-singular and has the block matrix shape
$$
\left( \begin{array}{ccc} & \vdots & \\ \cdots & f_{ij}(R) & \cdots \\ & \vdots & \end{array} \right)
$$
for some $f_{ij} \in F[t]$. The matrix of the form is a block matrix
$$
B = \left( \begin{array}{ccc} B_{11} & \cdots & B_{1h} \\ \vdots & \ddots & \vdots \\ B_{h1} & \cdots & B_{hh} \end{array} \right).
$$
The equation $XBX^*=B$ is equivalent to
\begin{eqnarray}\label{15w}
RB_{ij}R^* = B_{ij} \mbox{ for } 1 \leq i,j \leq h.
\end{eqnarray}
Since $f = f^*$, $R^*$ is similar to $R^{-1}$, so there exists $T \in \GL(n/h,F)$ such that
\begin{eqnarray}\label{16w}
R^* = T^{-1}R^{-1}T.
\end{eqnarray}
Thus (\ref{15w}) may be rewritten as
$$
R(B_{ij}T^{-1}) = (B_{ij}T^{-1})R.
$$
This shows that $B_{ij}T^{-1}$ belongs to the centralizer of $R$ and $B_{ij} = f_{ij}(R)T$ for a certain polynomial $f_{ij} \in F[t]$. We obtain the equation
\begin{eqnarray}\label{18w}
B = H \mathcal{T},
\end{eqnarray}
where $H = (f_{ij}(R)) \in C_G(X)$ and $\mathcal{T} = T \oplus \cdots \oplus T$.
\begin{lemma}\label{L210w}
	The matrix $T$ can be chosen such that $T = \varepsilon T^*$.
\end{lemma}
\begin{proof}
	If $R^* = T^{-1}R^{-1}T$, then $T$ can be replaced by $g(R)T$ for some $g \in F[t]$. So the aim is to prove that there exists $g(t)$ such that $g(R)T = \varepsilon(g(R)T)^*$. Applying the transpose-conjugate to the equation $RTR^* = T$, we get $RT^*R^* = T^*$, whence $R^* = T^{*-1}R^{-1}T^*$. Comparing with (\ref{16w}), we deduce that $T^*T^{-1}$ commutes with $R$, so $T^* = \phi(R)T$ for some $\phi \in F[t]$.
	
	If $T = -\varepsilon T^*$ and $RT = -\varepsilon (RT)^*$, then
	$$
	RT = -\varepsilon (RT)^* = (-\varepsilon T^*)R^* = TR^* = R^{-1}T
	$$
	by (\ref{16w}). But this implies $R=R^{-1}$ and so $R^2=1$, contradicting the assumption that $f(t) \neq t \pm 1$. Thus at least one of $T \neq -\varepsilon T^*$ and $RT \neq -\varepsilon (RT)^*$ holds.
	
	If $T \neq -\varepsilon T^*$, then we choose $g(t) = 1+\varepsilon\phi(t)$ and deduce that
	$$
	g(R)T = (1+\varepsilon\phi(R))T = T+\varepsilon T^*
	$$
	is non-singular, satisfying the hypothesis of the lemma. If $RT \neq -\varepsilon (RT)^*$, take $\psi \in F[t]$ such that $\psi(R) = R^{-1}$ and take $g(t) = t+\varepsilon\psi(t)\phi(t)$. We deduce that
	$$
	g(R)T = RT+\varepsilon R^{-1}T^* = RT+\varepsilon T^*R^* = RT+\varepsilon (RT)^*
	$$
	is non-singular and satisfies the hypothesis of the lemma. 
\end{proof}

We saw that if $Y \in C_G(X)$, then $Y$ is a block matrix $(\phi_{ij}(R))$, so it can be identified with a matrix in $\GL(h,E)$, where $E = F[t]/(f)$. The mapping $\phi_{ij}(R) \mapsto \phi_{ij}(R^{-1})$ is a field automorphism of $E$ of order 2 (since $R \neq R^{-1}$). For $Y = (\phi_{ij}(R)) \in \GL(h,E)$, define
\begin{eqnarray}        \label{dag-symbol}\index{$Y^{\dag}$}
Y^{\dag}:=(\phi_{ji}(R^{-1})).
\end{eqnarray}
The map $Y \mapsto Y^{\dag}$ is a sort of \textquotedblleft transpose conjugate" in $\GL(h,E)$, where the transposition is applied to the blocks $\phi_{ij}(R)$ and not to the single entries, and the conjugation is with respect to the automorphism $\phi_{ij}(R) \mapsto \phi_{ij}(R^{-1})$.
\begin{theorem}\label{L29w}
	Let $B = H\mathcal{T}$ as in $(\ref{18w})$. Now $B = \varepsilon B^*$ if, and only if, $H = H^{\dag}$. Moreover, if $Y \in C_G(X) \cong \GL(h,E)$, then $YBY^* = YHY^{\dag}\mathcal{T}$.
\end{theorem}
\begin{proof}
	Applying the standard transpose-conjugate to (\ref{18w}), it becomes $B^* = \mathcal{T}^* H^*$. By Lemma \ref{L210w} we can suppose that $\mathcal{T} = \varepsilon \mathcal{T}^*$. Using this assumption and Equation (\ref{16w}), we deduce that
	\begin{eqnarray*}
		\varepsilon B^* & = & \varepsilon \mathcal{T}^*H^*\\
		& = & (\varepsilon T^* f_{ji}(R^*))_{i,j} \\
		& = & (Tf_{ji}(R^*))_{i,j} \\
		& = & (f_{ji}(R^{-1})T)_{i,j}\\
		& = & H^{\dag}\mathcal{T}.
	\end{eqnarray*}
	So $B = \varepsilon B^*$ if, and only if, $H \mathcal{T} = H^{\dag}\mathcal{T}$. Since $\mathcal{T}$ is invertible, this holds if, and only if, $H = H^{\dag}$.
	
	Consider the second assertion. Write $Y$ as a block matrix $(\phi_{ij}(R))$ for $\phi_{ij} \in F[t]$. Now
	\begin{empheq}{equation*}
	\begin{split}
	YBY^* = YH\mathcal{T}Y^* & =  (\phi_{ij}(R))(f_{ij}(R)T)(\phi_{ij}(R))^*\\
	& =  (\sum_{\lambda,\mu} \phi_{i \lambda}(R) f_{\lambda \mu}(R)T\phi_{j\mu}(R^*))\\
	& =  (\sum_{\lambda,\mu} \phi_{i \lambda}(R)f_{\lambda \mu}(R)\phi_{j\mu}(R^{-1})T)\\
	& =  YHY^{\dag}\mathcal{T}. \qedhere
	\end{split}
	\end{empheq}
\end{proof}
The last theorem allows us to resolve the case $f = f^*$. If $B_1$ and $B_2$ are forms preserved by $X$, with $B_1=H_1\mathcal{T}$ and $B_2 = H_2\mathcal{T}$, then $H_1$ and $H_2$, considered as matrices in $\GL(h,E)$, are hermitian matrices. Thus they are congruent and there exists $Y \in \GL(h,\mathbb{F}_{q^d}) \cong C_G(X)$ such that $H_1 = YH_2Y^{\dag}$. By Theorem \ref{L29w} this implies $B_1 = H_1\mathcal{T} = YH_2Y^{\dag}\mathcal{T} = YB_2Y^*$. There is only one orbit of congruent forms under the action of $C_G(X)$, and the conjugacy class of $X$ in $G$ remains only one conjugacy class in $\C$.

If $Y \in C_G(X)$, then $Y \in \C$ if, and only if, $YBY^* = B$, that is $YHY^{\dag}\mathcal{T} = H\mathcal{T}$, so $YHY^{\dag} = H$. Hence the centralizer of $X$ in $\C$ is isomorphic to $\U(h,E)$.\\
\begin{proof}[Proof of Theorem $\ref{MainConjinC}$]
	Let $A_1, A_2$ be similar semisimple elements of $\C=\C(\beta)$. There exist $T_1, T_2 \in G$ such that $A_1^{T_1} = A_2^{T_2} = X$, where
	$$
	X = \left( \begin{array}{ccc} X_1 & & \\ & \ddots & \\ & & X_h \end{array} \right),
	$$
	and $X_i$ is the matrix of the restriction of $X$ to $\ker{f_i(X)}$ for $f_i \in \Phi$. Now $X = A_1^{T_1} \in \C(\beta T_1)$, and similarly $X \in \C(\beta T_2)$. If $B$ is the matrix of $\beta$ and $B_i = T_i^{-1}BT_i^{*-1}$ is the matrix of $\beta T_i$, for $i=1,2$, then $A_1$ and $A_2$ are conjugate in $\C$ if and only if $B_1, B_2$ are congruent in $C_G(X)$. By Lemma \ref{L26w}, $B_i$ has a block diagonal shape
	$$
	\left( \begin{array}{ccc} B_{i,1} & & \\ & \ddots & \\ & & B_{i,h} \end{array}\right)
	$$
	for every $i=1,2$, and every $Y \in C_G(X)$ has a block diagonal shape
	$$
	\left( \begin{array}{ccc} Y_1 & & \\ & \ddots & \\ & & Y_h \end{array}\right)
	$$
	with $Y_iX_i = X_iY_i$ for every $i=1, \dots, h$. Thus the equation $YB_1Y^*=B_2$ holds if, and only if, $Y_iB_{1,i}Y_i^* = B_{2,i}$ for every $i=1, \dots, h$.
	
	From our analysis of the three cases, we deduce that if $B_1$ and $B_2$ are of the same type, then there exists $Y_i$ in the centralizer of $X_i$ such that $Y_iB_{1,i}Y_i^* = B_{2,i}$. Thus, in the symplectic and unitary cases, there always exists $Y \in C_G(X)$ such that $YB_1Y^*=B_2$ (because the forms $B_{1,i}$ and $B_{2,i}$ are of alternating and hermitian type respectively).
	
	The orthogonal case is not so immediate. If $B_1$ and $B_2$ are symmetric forms, then the forms $B_{1,i}$ and $B_{2,i}$ may be either of plus or minus type. In particular, by Theorem \ref{elementsofCG} if $f_i \in \Phi_2$, then $B_{1,i}$ has plus type; if $f_i \in \Phi_3$, then $B_{1,i}$ has minus type. If $f_i = t \pm 1$, then $B_{1,i}$ can be of both types. It follows that, if at most one of $t+1$ and $t-1$ is an elementary divisor for $X$, then the types of $B_{1,i}$ and $B_{2,i}$ are uniquely determined, and they coincide, so $B_1$ and $B_2$ are congruent in $C_G(X)$. If both $t+1$ and $t-1$ are elementary divisors for $X$, then the corresponding forms can assume both types, since only the type of the sum of the two forms is well-defined. If the types are different, then the two forms are not congruent and $A_1$ and $A_2$ are not conjugate in $\C$. In such a case, the conjugacy class of $A_1$ in $G$ splits into two distinct classes in $\C$, with representatives $A_1$ and $A_1^T$, where $T \in G$ is such that $B_1 = TB_2T^*$.
\end{proof}
\subsection{Conjugacy classes in special groups}  \label{sectSpecial}
In this section, let $\Spec$ denote either $\SO^{\epsilon}(n,q)$ or $\SU(n,q)$\label{specialgroupsymbol}\index{$\Spec$} for certain $n \in \mathbb{N}$, $q$ is a prime power (assume $q$ odd in orthogonal case) and $\epsilon = +, - $ or $\circ$.
\begin{theorem} \label{ConjugacySpecial}
	The conjugacy class of a semisimple $X \in \SU(n,q)$ coincides with the conjugacy class of $X$ in $\U(n,q)$; equivalently, two semisimple elements of $\SU(n,q)$ are conjugate if, and only if, they are conjugate in $\U(n,q)$. The conjugacy class of a semisimple $X \in \SO^{\epsilon}(n,q)$ coincides with the class of $X$ in $\Or^{\epsilon}(n,q)$ if $X^2-1$ is singular, otherwise it splits into two distinct classes in $\SO^{\epsilon}(n,q)$ with representatives $X$ and $X^Z$, for $Z \in \Or^{\epsilon}(n,q)$, $\det{Z} = -1$.
\end{theorem}
\begin{proof}
	If $X$ is a semisimple element of $\Spec$, then
	\begin{eqnarray}\label{S1eq}
	|X^{\Spec}| = \frac{|\Spec|}{|C_{\Spec}(X)|} = |X^{\C}|\frac{|C_{\C}(X):C_{\Spec}(X)|}{k},
	\end{eqnarray}
	where $k=2$ in the orthogonal case and $k=q+1$ in the unitary case. Thus, to describe the conjugacy class of $X$ in $\Spec$ is equivalent to describing the centralizer $C_{\Spec}(X)$ of $X$ in $\Spec$, and this can be solved separately for each \ged $f$ of $X$. So, we suppose first that $X$ has a unique \ged $f\in \Phi$.\\
	\\
	\textbf{Case 1}: $f \in \Phi_1$. Now $X$ is a scalar matrix, so the centralizer of $X$ in $\Spec$ is $\Spec$.\\
	\\
	\textbf{Case 2}: $f \in \Phi_2$, $f = gg^*$. As described in (\ref{case2}), in an appropriate basis the matrix of $X$ has the shape
	$$
	\left( \begin{array}{cc} X_1 & \\ & X_1^{*-1} \end{array}\right),
	$$
	and every $Y \in C_{\C}(X)$ has matrix
	$$
	\left( \begin{array}{cc} Y_1 & \\ & Y_1^{*-1} \end{array}\right)
	$$
	with $Y_1 \in \GL(m,q^d)$, where $m$ is the multiplicity of $f$ as a \ged and $d = \deg{g}$.
	
	In the orthogonal case $\C = \Or^{\epsilon}(n,q)$, $\Spec = \SO^{\epsilon}(n,q)$, $\det(Y) = \det(Y_1)\det(Y_1^{-1}) = 1$, so $C_{\C}(X) \subseteq \Spec$, and $C_{\Spec}(X) = C_{\C}(X)$. Therefore, $|X^{\C}| = 2|X^{\Spec}|$, so the conjugacy class of $X$ in $\Or^{\epsilon}(n,q)$ splits into two distinct conjugacy classes in $\SO^{\epsilon}(n,q)$, with representatives $X$ and $X^Z$, for $Z \in \Or^{\epsilon}(n,q)$, $\det{Z} = -1$.
	
	In the unitary case, $\det(Y) = \det(Y_1)\det(Y_1)^{-q} = \det(Y_1)^{1-q}$; thus $Y \in \Spec$ if, and only if, $\det{Y_1}$ has order a multiple of $q-1$. This shows that $|C_{\C}(X): C_{\Spec}(X)| = q+1 = |\C: \Spec|$, so $|X^{\C}| = |X^{\Spec}|$. Thus, two semisimple elements of $\SU(n,q)$ are conjugate in $\SU(n,q)$ if, and only if, they are conjugate in $\U(n,q)$.\\
	\\
	\textbf{Case 3}: $f \in \Phi_3$. If $d= \deg{f}$ and $m$ is the multiplicity of $f$ as a \ged of $X$, then by Theorem \ref{L29w} the centralizer of $X$ in $\C$ is isomorphic to $\U (m,q^{d/2}) \leqslant \GL(m,q^d)$. Let us distinguish the two cases.
	
	\textit{Orthogonal case}: $\C = \Or^{\epsilon}(md,q)$, so $d=2d'$ is even (otherwise $f \neq f^*$) and $C_{\C}(X) \cong \U(m,q^{d'})$. If $y \in \U(m,q^{d'})$, then $\det{y}$ has order a multiple of $q^{d'}+1$, so $\det{y} = \lambda^{q^{d'}-1}$ for some $\lambda \in \mathbb{F}_{q^d}$. If $Y$ is the corresponding element in $\GL(md,q)$, then $\det{Y} = \mathrm{N}_{q^d|q}(\det{y}) = (\lambda^{q^{d'}-1})^{(q^d-1)/(q-1)} =1$, since the exponent of $\lambda$ is a multiple of $q^d-1 = |\mathbb{F}_{q^d}^*|$. This proves that $C_{\Spec}(X) = C_{\C}(X)$, so $|X^{\C}| = 2|X^{\Spec}|$.
	
	\textit{Unitary case}: $\C = \U(md,q)$, so $d$ is odd and $C_{\C}(X) \cong \U(m,q^d)$. Again, if $y \in \U(m,q^d)$, then $\det{y} = \lambda^{q^d-1}$ for some $\lambda \in \mathbb{F}_{q^{2d}}$. Hence, if $Y$ is the corresponding element in $\GL(md,q^2)$, then $\det{Y} = \mathrm{N}_{q^{2d}|q^2}(\det{y}) = (\lambda^{q^{d}-1})^{(q^{2d}-1)/(q^2-1)} =\mathrm{N}_{q^{2d}|q^2}(\lambda)^{q^d-1}$. Thus $Y \in \Spec$ if, and only if, $\mathrm{N}_{q^{2d}|q^2}(\lambda)^{q+1}=1$. It follows that $|C_{\C}(X) : C_{\Spec}(X)| = q+1$ and $|X^{\C}| = |X^{\Spec}|$.\\
	
	These three cases imply that $|C_{\C}(X):C_{\Spec}(X)| = |\C: \Spec|$ except when $\C$ is orthogonal and $X$ has no elementary divisors $t \pm 1$. The statement of the theorem easily follows.
\end{proof}

\subsection{Conjugacy classes in $\Or^{\epsilon}(\MakeLowercase{2m,2^k})$} \label{sectQuad}
Recall that a quadratic form $Q$ can be represented by a matrix $A = (a_{ij})$ such that $Q(v) = vAv^{\tr}$ for all $v \in V$. Recall that $\C(Q) \leqslant \C(\beta_Q)$, so $\Or^{\epsilon}(2m,2^k) \leqslant \Sp(2m,2^k)$ and we reduce the quadratic case in even characteristic to the symplectic case.
\begin{theorem} \label{ConjQuadratic}
	Two semisimple $X,Y \in \Or^{\epsilon}(2m,2^k)$ are conjugate in $\Or^{\epsilon}(2m,2^k)$ if, and only if, they are conjugate in $\Sp(2m,2^k)$, so if, and only if, they are similar. Moreover, if $X+1$ is non-singular, then $C_{\Or^{\epsilon}(2m,2^k)}(X) = C_{\Sp(2m,2^k)}(X)$.
\end{theorem}
\begin{proof}
	By analogy with the other cases, we can assume that $X \in \Or^{\epsilon}(2m,2^k)$ has a unique \ged $f \in \Phi$.
	
	If $f \in \Phi_1$, then $f(t)=t+1$ and $X = \mathbb{I}_{2m}$, so the centralizer of $X$ is $\Or^{\epsilon}(2m,2^k)$ and the conjugacy class of $X$ contains only $X$.
	
	Suppose $f \in \Phi_2 \cup \Phi_3$. Let $\Or^{\epsilon}(2m+1,2^k)$ be the group of isometries for the quadratic form of matrix
$$
\widehat{A} = \left(\begin{array}{c|ccc} 1 & & & \\ \hline & & & \\ & & A & \\ & & & \end{array}\right),
$$
where $A$ is the matrix of the quadratic form preserved by $\Or^{\epsilon}(2m,2^k)$. Every matrix in $\Or^{\epsilon}(2m+1,2^k)$ has shape
\begin{eqnarray} \label{formoddquad}
\widehat{X} = \left( \begin{array}{c|ccc} 1 & & 0 & \\ \hline & & & \\ v & & X & \\ & & & \end{array} \right),
\end{eqnarray}
where $v$ is a $2m$-dimensional column vector. Note that if $v=0$, then $X$ is an isometry for $\C(A) = \Or^{\epsilon}(2m,2^k)$. It is well-known (see \cite[14.1]{CGSA}) that the function $\widehat{X} \mapsto X$ is an isomorphism between $\Or^{\epsilon}(2m+1,2^k)$ and $\C(A+A^{\tr}) \cong \Sp(2m,2^k)$, and that the inverse of this isomorphism maps $X$ to a matrix of the form (\ref{formoddquad}), where $v=0$ if, and only if, $X \in \Or^{\epsilon}(2m,2^k)$. We now use this isomorphism to describe the centralizer of $X$ in $\Or^{\epsilon}(2m,2^k)$. The centralizer of $X$ in $\Sp(2m,2^k)$ is isomorphic to the centralizer of $\widehat{X}$ in $\Or^{\epsilon}(2m+1,2^k)$, where
	$$
	\widehat{X} = \left( \begin{array}{c|ccc} 1 & & & \\ \hline & & & \\ & & X & \\ & & & \end{array}\right).
	$$
	Since $t+1$ is not an elementary divisor of $X$, every element of the centralizer of $\widehat{X}$ in ${\Or^{\epsilon}(2m+1,2^k)}$ has the block diagonal shape
	$$
	\widehat{Y} = \left( \begin{array}{c|ccc} 1 & & & \\ \hline & & & \\ & & Y & \\ & & & \end{array}\right)
	$$
	for some $Y$. Since the vector $v$ in Equation (\ref{formoddquad}) is zero, $Y \in \Or^{\epsilon}(2m,2^k)$. Conversely, it is clear that for every $Y \in C_{\Or^{\epsilon}(2m,2^k)}(X)$, the corresponding $\widehat{Y}$ is in the centralizer of $\widehat{X}$ in $\Or^{\epsilon}(2m+1,2^k)$. This proves that $|C_{\Or^{\epsilon}(2m,2^k)}(X)| = |C_{\Sp(2m,2^k)}(X)|$, so they coincide since the first is obviously contained in the second.
	
	Similarly, if $X_1,X_2 \in \Or^{\epsilon}(2m,2^k)$ are conjugate in $\Sp(2m,2^k)$, then $\widehat{X}_1$ and $\widehat{X}_2$ are conjugate in $\Or^{\epsilon}(2m+1,2^k)$ and there exists
	$$
	\widehat{Z} = \left( \begin{array}{c|ccc} 1 & & & \\ \hline & & & \\ v_z & & Z & \\ & & & \end{array}\right)
	$$
	such that $\widehat{Z}^{-1}\widehat{X}_1\widehat{Z} = \widehat{X}_2$. If $t+1$ is not an elementary divisor of $X_1$, then $\widehat{Z}$ preserves the diagonal block structure of $\widehat{X}_1$ and $\widehat{X}_2$, that is $v_z=0$. This shows that $Z \in \Or^{\epsilon}(2m,2^k)$, so $X_1$ and $X_2$ are conjugate in $\Or^{\epsilon}(2m,2^k)$.
\end{proof}

\subsection{Conjugacy classes in $\Omega^{\epsilon}(\MakeLowercase{n,q})$} \label{OmegaClasses}
In this section we will analyse the case $n>2$ since the groups $\Omega^{\epsilon}(2,q)$ are cyclic.
\begin{theorem} \label{semisimpleOmega}
	Let $n \geq 3$ and let $x \in \Omega^{\epsilon}(n,q)$ be semisimple. If $q$ is odd, or if $q$ is even and $x+\mathbf{1}_V$ is singular, then the conjugacy class of $x$ in $\Omega^{\epsilon}(n,q)$ coincides with the conjugacy class of $x$ in $\SO^{\epsilon}(n,q)$. If $q$ is even and $x+\mathbf{1}_V$ is non-singular, then the conjugacy class of $x$ in $\SO^{\epsilon}(n,q)$ splits into two distinct conjugacy classes in $\Omega^{\epsilon}(n,q)$, with representatives $x$ and $x^z$, where $z \in \SO^{\epsilon}(n,q) \setminus \Omega^{\epsilon}(n,q)$.
\end{theorem}
\begin{proof}
	Denote $\Or^{\epsilon}(n,q)$, $\SO^{\epsilon}(n,q)$ and $\Omega^{\epsilon}(n,q)$ by $\C$, $\Spec$ and $\Omega$ respectively. For every $x \in \C$, we denote the subspace $\im(\mathbf{1}_V-x)$ by $V_x$ (or by $V_X$ if $X$ is the matrix of $x$). Let us distinguish the cases $q$ even and odd. We use \cite[4.1.9(iii)]{KL} in our analysis of the case of $q$ odd and $f \in \Phi_2$; the remainder is independent.\\
	\\
	\textit{Case $q$ even}. Let $\theta(x) = \dim{V_x} \pmod{2}$ be the spinor norm. Let $x \in \Omega$ be semisimple. The strategy to analyse the conjugacy class and the centralizer of $x$ in $\Omega$ is the same as in the previous sections: determine whether $C_{\C}(x)$ contains $y \notin \Omega$. Clearly, $y \in \Omega$ if, and only if, $\theta(y)=0$: namely, if the sum of the multiplicities of all elementary divisors $(t+1)^r$ in $y$ is even. Let us suppose that $x$ has a unique \ged $f \in \Phi$.
	\begin{itemize}
		\item $f \in \Phi_1$. Now $x$ is the identity. The centralizer of $x$ in $\C$ is $\C$, and $y$ is conjugate to $x$ if, and only if, $x=y$.
		\item $f \in \Phi_2$, $f=gg^*$. Now $C_{\C}(x)$ is the set of all block diagonal matrices $Y=Y_1 \oplus Y_1^{*-1}$. Since $(t+1)^r$ is self-reciprocal for every $r$, the multiplicities of $(t+1)^r$ in $Y_1$ and in $Y_1^{*-1}$ are the same, so the multiplicity of $(t+1)^r$ in $Y$ is always even. Hence $C_{\C}(x) \subseteq \Omega$.
		
		\item $f \in \Phi_3$. If $f=f^*$ has degree $2d$ and multiplicity $m$ as an elementary divisor of $x$, then $C_{\C}(x)$ is isomorphic to $\U(m,q^d)$. If $(t+1)^r$ has multiplicity $\mu$ as an elementary divisor of $y \in \U(m,q^d)$, then the corresponding matrix in $C_{\C}(x)$ has $(t+1)^r$ as an elementary divisor with multiplicity $2d\mu$, so always even. Hence $C_{\C}(x) \subseteq \Omega$.
	\end{itemize}
	We deduce that if $x \in \Omega$ is semisimple, then $C_{\C}(x) \subseteq \Omega$ if, and only if, $x+\mathbf{1}_V$ is non-singular. Hence, $t+1$ is not an elementary divisor of $x$; so the conjugacy class of $x$ in $\C$ splits into two distinct classes in $\Omega$ with representatives $x$ and $x^z$, with $z \in \C \setminus \Omega$.\\
	\\
	\textit{Case $q$ odd}. For each $v \in V$, $Q(v) \neq 0$, the \textit{reflection}\index{reflection} $r_v$ is defined as
	$$
	ur_v = u-\frac{\beta(v,u)}{Q(v)}v \: \mbox{ for } u \in V.
	$$
	It is easy to check that $r_v \in \C$ for all $v \in V$, $Q(v) \neq 0$. Let $\chi_{r_v}$ be the Wall form of $r_v$, defined in (\ref{WallForm}). If $r_v$ is a reflection, then $V_{r_v} = \langle v \rangle$ and $\chi_{r_v}(v,v)= Q(v)$. In fact, if $w \in V$ is such that $v = w-wr_v$, then
	$$
	\chi_{r_v}(v,v) = -\beta(w,wr_v-w)=Q(w)+Q(wr_v-w,w)-Q(wr_v) = Q(-v) = Q(v),
	$$
	using that $r_v \in \C$. It follows that $\theta(r_v) = Q(v)\mathbb{F}_q^{* 2}$. Since $\C$ is generated by reflections (see \cite[2.5.6]{KL}), this leads to a second equivalent definition of the spinor norm:
	$$
	\theta(x) = \left\{ \begin{array}{ll} 0 & \mbox{if }\prod_{i=1}^k Q(v_i) \in \mathbb{F}_q^{* 2},\\ 1 & \mbox{if }\prod_{i=1}^k Q(v_i) \notin \mathbb{F}_q^{* 2}, \end{array}\right.
	$$
	where $x = r_{v_1} \cdots r_{v_k}$.
	
	For every semisimple $x \in \Omega$, we search for $y \in C_{\Spec}(x)$ such that $\theta(y)=1$. As usual, suppose that $x$ has a unique \ged $f \in \Phi$ and distinguish the three cases.\\
	\begin{itemize}
		\item $f \in \Phi_1$. Let $x = \pm\mathbf{1}_V$. The centralizer of $x$ in $\Spec$ is $\Spec$, so it obviously contains $y \notin \Omega$.
		\item $f \in \Phi_2$, $f = g g^*$, with $g\neq g^*$ irreducible of degree $d$ and multiplicity $m$. Write $V = U \oplus W$, with $U = \ker{g(x)}$ and $W = \ker{g^*(x)}$. Thus $U$ and $W$ are totally isotropic and the set of block diagonal matrices $K = \{ Y\oplus Y^{*-1} \, : \, Y \in \GL(U) \}$ is a subgroup of $\C$  isomorphic to $\GL(U) \cong \GL(n/2,q)$. Denote by $\frac{1}{2}K$ the unique subgroup of $K$ of index 2 consisting of the matrices of the form $Y\oplus Y^{*-1}$, with $Y$ an element of the unique subgroup of $\GL(U)$ of index 2, that is the set of matrices $Y\oplus Y^{*-1}$ such that $\det{Y} \in \mathbb{F}_q^{* 2}$. There are only two possibilities: either $K \subseteq \Omega$ or $K \cap \Omega = \frac{1}{2}K$. Write $U = \langle e_1, \dots, e_k \rangle$ and $W = \langle f_1, \dots, f_k \rangle$, with $Q(e_i)=Q(f_j)=0$ and $\beta(e_i,f_j) = \delta_{ij}$ for every $i,j$. Put $\lambda \in \mathbb{F}_q^* \setminus \mathbb{F}_q^{* 2}$: now $z = r_{e_1+f_1}r_{e_1+\lambda f_1}$ is in $K \setminus \Omega$ because $Q(r_{e_1+f_1})Q(r_{e_1+\lambda f_1}) = \lambda \notin\mathbb{F}_q^{* 2}$. This proves that $K \cap \Omega = \frac{1}{2}K$. Therefore $y=Y\oplus Y^{*-1} \in K$ belongs to $\Omega$ if, and only if, $\det{Y}$ is a square in $\mathbb{F}_q$. Since $C_{\Spec}(x) \subseteq K$ and there are elements of $C_{\Spec}(x) \cong \GL(m,q^d)$ with non-square determinant, we have proved that $C_{\Spec}(x) \not\subseteq \Omega$.
		\item $f \in \Phi_3$. Let $f$ have degree $2d$ and multiplicity $m$ as a \ged of $x$. We can suppose $m=1$, since the general case is a direct sum of $m$ cases of multiplicity 1. Thus, $C_{\Spec}(x)$ is a cyclic group isomorphic to $\U(1,q^d) \leqslant \GL(1,q^{2d})$. For every $y \in C_{\Spec}(x)$ with matrix $Y$, either $Y=1$ or $1-Y$ is non-singular. For $Y \neq 1$, $\chi_y$ can be defined by $\chi_y(u,v):=\beta(u(1-y)^{-1},v)$ for every $u,v \in V_y$. So, if $B$ is the matrix of $\beta$, then the matrix of $\chi_y$ is $(1-Y)^{-1}B$, whose determinant is $\theta(y) = \det(1-Y)^{-1}\det{B}$. This is a square if, and only if, $\det(1-Y) \equiv \det{B} \pmod{\mathbb{F}_q^{* 2}}$. Now let $\widetilde{y}$ be the element of $\GL(1,q^{2d}) \cong \mathbb{F}_{q^{2d}}^*$ corresponding to $y$. Thus $1-\widetilde{y}$ is the corresponding element of $1-y$. But $y \in C_{\Spec}(x) \cong \U(1,q^d) = \langle \omega^{q^d-1} \rangle$, where $\omega$ is a primitive element of $\mathbb{F}_{q^{2d}}$. So $1-\widetilde{y} = 1-\omega^{\ell (q^d-1)}$ for some $\ell$, and $\det(1-Y) = \mathrm{N}_{q^{2d}|q}(1-\widetilde{y})$ is a square in $\mathbb{F}_q$ if, and only if, $1-\widetilde{y}$ is a square in $\mathbb{F}_{q^{2d}}$. So, the problem to find $y \in C_{\Spec}(x)$ with $\theta(y)=0$ (resp.\ $1$) reduces to finding an integer $\ell$ such that $1-\omega^{\ell (q^d-1)}$ is a square (resp.\ non-square) in $\mathbb{F}_{q^{2d}}$. Observe that $1+\omega^{q^d-1} = (\omega+\omega^{q^d})\omega^{-1} = \mathrm{Tr}_{q^{2d}|q^d}(\omega)\cdot \omega^{-1}$ is a non-square in $\mathbb{F}_{q^{2d}}$, being a product of a square ($\mathrm{Tr}_{q^{2d}|q^d}(\omega) \in \mathbb{F}_{q^d}$, so is a square) and a non-square. Writing $(1-\omega^{2(q^d-1)}) = (1-\omega^{q^d-1})(1+\omega^{q^d-1})$, we see that exactly one of $1-\omega^{2(q^d-1)}$ and $1-\omega^{q^d-1}$ is a square, since their quotient is a non-square. Hence, it is sufficient to take $\ell = 1$ or $2$, as needed.
	\end{itemize}
	In all three cases, we have proved that if $x \in \Omega$ is semisimple, then $C_{\Spec}(x) \not\subseteq \Omega$, so ${|C_{\Spec}(x) : C_{\Omega}(x)|} =2$. Moreover,
	$$
	|x^{\Omega}| = \frac{|\Omega|}{|C_{\Omega}(x)|} = \frac{|\Spec|}{|C_{\Spec}(x)|} = |x^{\Spec}|,
	$$
	so two semisimple elements of $\Omega$ are conjugate if, and only if, they are conjugate in $\Spec$.
\end{proof}
We state the following observations about membership in $\Omega$. This characterization allows us to determine which semisimple classes of $\Spec$ are in $\Omega$.
\begin{lemma}   \label{oddOrderinOmega}
	Let $\Omega = \Omega^{\epsilon}(n,q)$ be the Omega group of $\C(Q)$.
	\begin{description}
		\item{$\mathrm{(a)}$} Every element of odd order in $\Spec$ is also in $\Omega$.
		\item{$\mathrm{(b)}$} Let $q$ be odd and let $x \in \Omega$ be semisimple. Assume that $x$ has a unique \ged $f \in \Phi$ of multiplicity $m$.
		\begin{itemize}
			\item If $f = t-1$, then $x \in \Omega$.
			\item If $f = t+1$, then $x \in \Omega$ if, and only if, $Q$ has square discriminant.
			\item If $f \in \Phi_2$, $f=gg^*$, then $x \in \Omega$ if, and only if, either $m$ is even or $\det{C}$ is a square in $\mathbb{F}_q^*$, where $C$ is the companion matrix of $g$.
			\item If $f \in \Phi_3$, then $x \in \Omega$ if, and only if, either $m$ is even or the order of $C$ divides $(q^{d/2}+1)/2$, where $C$ is the companion matrix of $f$ and $d=\deg{f}$.
		\end{itemize}
	\end{description}
\end{lemma}
\begin{proof}
	(a) By contradiction, if $x \in \Spec \setminus \Omega$ has odd order, then let $H = \langle x \rangle$. Since $|\Spec:\Omega| = 2$, $\Omega H = \Spec$. Since $\Omega H / \Omega \cong H/ (\Omega \cap H)$, we deduce that $|H| = |\Spec:\Omega| \cdot |\Omega \cap H|$, but this is impossible because $H$ would have even cardinality.
	
	(b) The case $f=t-1$ is trivial. For $f=t+1$, see \cite[Prop.\ 2.5.13]{KL}. The case $f \in \Phi_2 \cup \Phi_3$ follows from the discussion above.
\end{proof}
If a semisimple $x \in \Spec$ has \geds $f_1, \dots, f_k$, then we can decide its membership in $\Omega$ by applying the lemma to each $f_i$. Observe that if $q$ is odd (resp.\ even), then every unipotent (resp.\ semisimple) element lies in $\Omega$.
\subsection{Representatives for conjugacy classes in classical groups} \label{formsandelements}
We show how to write down explicitly a semisimple element of a classical group having a given list of generalized elementary divisors. We must write both the matrix of the element and the form in an appropriate basis. The solutions for symplectic and unitary cases were personally communicated by Don Taylor.

Let $\beta$ be a sesquilinear form (resp.\ let $Q$ be a quadratic form) and $\C = \C(\beta)$ or $\C(Q)$. Let $x \in \C$ be semisimple. We suppose that $x$ has a unique \ged $f \in \Phi$ of multiplicity $1$. If $x$ has more generalized elementary divisors, the matrices $B$ and $X$ can be written as the block diagonal sum of matrices obtained for every single generalized elementary divisor.

If $f \in \Phi_1$, then $X$ is a scalar matrix, so an arbitrary basis can be chosen for the matrix $B$. If $f \in \Phi_2$, $f=gg^*$, then in a suitable basis
$$
X = \begin{pmatrix}
R & \\ & R^{*-1}
\end{pmatrix}, \quad B = \begin{pmatrix}
\mathbb{O} & \mathbb{I} \\ \varepsilon \mathbb{I} & \mathbb{O}
\end{pmatrix},
$$
where $R$ is the companion matrix of $g$ and $\varepsilon$ is either $-1$ (symplectic case), $0$ (quadratic case) or $1$ (symmetric and unitary cases). Now consider the case $f \in \Phi_3$.\\
\\
\textit{Symplectic case}. Let $f(t) = 1+a_1t+ a_2t^2 + \cdots + a_dt^d+a_{d-1}t^{d+1}+ \cdots + a_1t^{2d-1}+t^{2d}$. Let $X$ be the transpose of the companion matrix of $f$ and let $B = \left( \begin{array}{cc} \mathbb{O} & -P^{\tr} \\ P & \mathbb{O} \end{array}\right)$, where $P$ is the $d \times d$ upper triangular matrix with constant upper diagonals
$$
\left( \begin{array}{cccccc} 1 & a_1 & a_2 & \cdots & a_{d-2} & a_{d-1} \\ & 1 & a_1 & \ddots & \ddots & a_{d-2} \\ && \ddots & \ddots & \ddots & \vdots \\ &&  & \ddots & \ddots & a_2 \\ &&&& 1 & a_1 \\ &&&&& 1 \end{array}\right).
$$
A straightforward computation shows that $X \in \C(B)$.\\
\\
\textit{Quadratic case}. If $\deg{f}=2$, so $f=t^2+at+1$, then we take $X$ to be the companion matrix of $f$ and 
$$
A = \left(\begin{array}{cc} 1 & -a \\ 0 & 1 \end{array}\right).
$$
Suppose $\deg{f}>2$ and let $f(t) = 1+a_1t+ a_2t^2 + \cdots + a_dt^d+a_{d-1}t^{d+1}+ \cdots + a_1t^{2d-1}+t^{2d}$. Let $X$ be the companion matrix of $f$. Let $A$ be the upper triangular matrix with constant upper diagonals
$$
\left( \begin{array}{ccccccc} && 1 & b_0 & b_1 & \cdots & b_{d-1} \\ &&& 1 & \ddots & \ddots & \vdots \\ &&&& \ddots & \ddots & b_1 \\ &&&&& \ddots & b_0 \\ &&&&&& 1 \\ &&&&&& \\ &&&&&& \end{array}\right),
$$
where the coefficients $b_i$ are defined in the following way. If $p$ is odd, then the vector of the $b_i$'s satisfies the linear system
$$
\left(\begin{array}{ccc} b_0 & \cdots & b_{d-1} \end{array}\right) \left( \begin{array}{ccccc} 2 & a_1 & a_2 & \cdots & a_{d-1} \\ & 1 & a_1 & \ddots & \vdots \\ && \ddots & \ddots & a_2 \\ &&& \ddots & a_1 \\ &&&& 1 \end{array}\right) = \left(\begin{array}{cccccc} 2a_1 & a_2-1 & a_3 & a_4 & \cdots & a_d \end{array}\right).
$$
If $p=2$, then the vector of the $b_i$'s satisfies the linear system
$$
\left(\begin{array}{ccc} b_0 & \cdots & b_{d-1} \end{array}\right) \left( \begin{array}{ccccc} c_1 & a_1 & a_2 & \cdots & a_{d-1} \\ c_2 & 1 & a_1 & \ddots & \vdots \\ \vdots && \ddots & \ddots & a_2 \\ \vdots &&& \ddots & a_1 \\ c_d &&&& 1 \end{array}\right) = \left(\begin{array}{cccccc} \delta & a_2-1 & a_3 & a_4 & \cdots & a_d \end{array}\right),
$$
where
\begin{eqnarray*}
	c_i & = & \sum_{j=1}^{d+1-i} a_{j-1}a_{d+1-j} \quad \forall \, i=1, \dots, d;\\
	\delta & = & \sum_{j=0}^d a_ja_{d+j-1}
\end{eqnarray*}
with $a_0=1$ and $a_{d+j}=a_{d-j}$ for all $j=1, \dots, d$. A direct computation shows that $X \in \C(Q)$, where $Q$ is the quadratic form with matrix $A$.\\
\\
\textit{Unitary case}. Let $f(t) = a_0+a_1t+ \cdots +a_dt^d+a_0t^{d+1}(\overline{a}_d+\overline{a}_{d-1}t+ \cdots +\overline{a}_0t^d)$ be an irreducible polynomial in $\mathbb{F}_{q^2}[t]$ with $n=2d+1$ and $a_0\overline{a}_0 =1$. Define
\begin{eqnarray*}
	b_0 & := & \sqrt[q-1]{a_0^{-1}}; \\
	b_i & := & \overline{b}_0 \sum_{j=0}^i a_j \: \mbox{ for } 1 \leq i \leq d; \\
	c & := & \sqrt[q+1]{b_d+\overline{b}_d}.
\end{eqnarray*}
The matrix
$$
\left(\begin{array}{cccc|c|cccc} & 1 &&&&&&& \\ && \ddots &&&&&& \\ &&& 1 &&&&& \\ &&&&&&&& 1/\overline{b}_0 \\ \hline &&&& 1 &&&& -\overline{c}/\overline{b}_0 \\ \hline b_0 & b_1 & \cdots & b_{d-1} & c &&&& - \overline{b}_d/\overline{b}_0 \\ &&&&& 1 &&& -\overline{b}_{d-1}/\overline{b}_0 \\ &&&&&& \ddots && \vdots \\ &&&&&&& 1 & -\overline{b}_1/\overline{b}_0 \end{array}\right)
$$
has characteristic polynomial $f$ and preserves the hermitian form described by the matrix
$$
\begin{pmatrix}
 && 1 \\ & \adots & \\ 1 &&
\end{pmatrix}.
$$

If we need the matrix of the form preserved by the companion matrix of $f$, then we can get it by a change of basis: if $X$ and $B$ are the matrices described above, $C$ is the companion matrix of $f$ and $P \in \GL(V)$ satisfies $PXP^{-1}=C$, then $C$ preserves  the form $PBP^*$.

\section{Centralizer of a semisimple element}  \label{SSCentrOrder}
In this section we summarize the centralizer structure of a semisimple element in a classical group.
\begin{theorem} \label{CentrSem}
Let $V$ be an $n$-dimensional vector space over $\mathbb{F}_{q^2}$ (in unitary case) or $\mathbb{F}_q$ (otherwise) with a non-degenerate reflexive sesquilinear form $\beta$, and let $\C =\C(\beta)$ be a symplectic, orthogonal or unitary group. Let $x \in \C$ be semisimple. For every polynomial $f$ in $\Phi$, let $m_f$ be the multiplicity of $f$ as a \ged of $x$, let
$$
d_f = \left\{ \begin{array}{ll} \deg{f} & \mbox{ if $\C$ is unitary,}\\ \deg{f}/2 & \mbox{ if $\C$ is symplectic or orthogonal} \end{array}\right.
$$
and let $\beta_f$ be the restriction of the form $\beta$ to the eigenspace $\ker{f(x)}$.
\begin{eqnarray}\label{CardofCentralizer}
C_{\C}(x) \cong \prod_{f \in \Phi_1} \!\! \C(\beta_f) \times \prod_{f \in \Phi_2}\!\!\GL(m_f,q^{2d_f}) \times \prod_{f \in \Phi_3}\!\! \U(m_f,q^{d_f}),
\end{eqnarray}
where the products run over all \geds of $x$. If $Q$ is a non-singular quadratic form on $V$, then $(\ref{CardofCentralizer})$ holds on replacing $\beta$ by $Q$.
\end{theorem}
From the structure given in Equation (\ref{CardofCentralizer}) it is easy to compute the centralizer order and consequently the class size of a semisimple element of $\C$.

In the quadratic case, if $f=t\pm 1$ and $m_f$ is even, then $\C(Q_f)$ may be an orthogonal group of plus or minus type. If only one of $t+1$ and $t-1$ is an elementary divisor, then the type of $Q_f$ can be deduced from the type of $Q$ and the multiplicity of the \geds in $\Phi_3$. If both $t-1$ and $t+1$ are elementary divisors, then the Jordan form of $x$ is not sufficient to establish the cardinality of $C_{\C(Q)}(x)$ and it is necessary to investigate the type of $Q_{t+1}$ or $Q_{t-1}$.\\

If $\Spec$ is the special unitary subgroup $\Spec(\beta)$ or the special orthogonal subgroup $\Spec(Q)$, then $C_{\Spec}(x) = C_{\C}(x) \cap \Spec$. Analogously, if $\Omega=\Omega(Q)$ is the Omega subgroup of $\C(Q)$, then $C_{\Omega}(x) = C_{\C}(x) \cap \Omega$. If $q$ is even, then the centralizer of $x$ in $\Omega$ can be obtained by replacing the first factor in Equation (\ref{CardofCentralizer}) by $\prod_{f \in \Phi_1} \Omega(\beta_f)$, since the factors corresponding to \geds $f \in \Phi_2 \cup \Phi_3$ already lie in $\Omega$. From the results of Sections \ref{sectSpecial} and \ref{OmegaClasses}, one can easily see that
$$
|C_{\Spec}(x)| = \left\{ \begin{array}{ll} \frac{|C_{\C}(x)|}{q+1} & \mbox{ if $\C$ is unitary;}\\ \\
\frac{|C_{\C}(x)|}{2} & \mbox{ if $\C$ is orthogonal, $q$ is odd and $x^2-\mathbf{1}_V$ is singular;}\\ \\
|C_{\C}(x)|  & \mbox{ otherwise;} \end{array} \right.
$$
and
$$
|C_{\Omega}(x)| = \left\{ \begin{array}{ll} |C_{\Spec}(x)| & \mbox{ if $q$ is even and $x+\mathbf{1}_V$ is non-singular};\\ \\
\frac{|C_{\Spec}(x)|}{2} & \mbox{ otherwise.} \end{array} \right.
$$

\newpage
\chapter{Centralizers and conjugacy classes of unipotent elements} \label{unipotentChap}
The problem of listing representative of the unipotent conjugacy classes in classical groups was solved by Gonshaw, Liebeck and O'Brien \cite{GLOB}. Here we summarize without proof the relevant results. We do not describe explicitly the representatives, but introduce some parameters in the description of the classes that will be helpful in the following sections. For an explicit descriptions of the blocks $V_b(k)$ and $W_b(k)$, see \cite{GLOB}. Following \cite[Chap.\ 7]{LS} and \cite[\S 2.6]{GEW} we also describe the structure of the centralizer of a unipotent element.
\section{Unipotent conjugacy classes in classical groups}\label{UnipotentRepr}

Let $q$ be a prime power, and let $V$ be a finite dimensional vector space over $F$, where $F=\mathbb{F}_{q^2}$ in the unitary case and $F = \mathbb{F}_q$ in the other cases. Let $J_{\ell}$ be the unipotent Jordan block of dimension $\ell$. We write
$J_{\ell}^{\oplus k}$ to denote the diagonal join of $k$ copies of $J_{\ell}$.

\subsection{Unitary case}
Let $z \in \U(n,q)$ with $\det{z} = \omega^{q-1}$, where $\omega$ is a primitive element of $F$.

\begin{proposition} \label{uniconjU}
There exists only one conjugacy class in $\U(n,q)$ whose representative $x$ has Jordan form $\bigoplus_1^s (J_{n_i})^{\oplus r_i}$. The conjugacy class of such $x$ in $\U(n,q)$ splits into $t$ distinct classes in $\SU(n,q)$, where
$$
t = \gcd(n_1, \dots, n_s, q+1).
$$
Representatives for such classes are given by $x, x^z, x^{z^2}, \dots, x^{z^{t-1}}$, with $z$ defined above.
\end{proposition}
In other words, two unipotent elements are conjugate in $\U(n,q)$ (resp.\ $\SU(n,q)$) if, and only if, they are conjugate in $\GL(n,q)$ (resp.\ $\SL(n,q)$).

\subsection{Symplectic case, $q$ odd}      \label{SectUniRepsSp}
Let $\alpha$ be a non-square in $F^*$. For every positive integers $k$, let $V_1(2k) \in \Sp(2k,q)$ have Jordan form $J_{2k}$. If the symplectic space is $W \oplus W^{\bot}$, with $W$ and $W^{\bot}$ totally isotropic, let $z_b$ the element defined by $z_b(w)=bw$ and $z_b(w')=w'$ for every $w \in W$ and $w' \in W^{\bot}$ and let $V_b(2k)$ be the conjugate of $V_1(2k)$ by $z_b$. Finally, let $W(2l+1) \in \Sp(4l+2,q)$ have Jordan form $J_{2l+1}^{\oplus 2}$ for every $l \geq 0$.

Recall that, by Theorem \ref{elementsofCG}, all Jordan block of odd dimension occur with even multiplicity, so the Jordan form of every unipotent element of $\Sp(2m,q)$ is
\begin{eqnarray} \label{unirepSp}
\bigoplus_{i=1}^r (J_{2k_i})^{\oplus a_i} \oplus \bigoplus_{j=1}^s (J_{2l_j+1})^{\oplus 2c_j},
\end{eqnarray}
where the $k_i$ and the $l_j$ are distinct and $2(\sum k_ia_i+ \sum(2l_j+1)c_j) = 2m$.
\begin{proposition} \label{uniconjSp}
There are $2^r$ unipotent classes in $\Sp(2m,q)$ having Jordan form $(\ref{unirepSp})$. Representatives are
\begin{eqnarray*}
\bigoplus_{i=1}^r (V_{b_i}(2k_i) \oplus V_1(2k_i)^{\oplus a_i-1}) \oplus \bigoplus_{j=1}^s W(2l_j+1)^{\oplus c_j}
\end{eqnarray*}
for $b_i \in \{1, \alpha \}$, where $\alpha$ is a fixed non-square in $\mathbb{F}_q$.
\end{proposition}

\subsection{Orthogonal case, $q$ odd} \label{SectUniRepsGO}
For every $b \in \mathbb{F}_q^*$, $k \geq 0 $ and $l \geq 1$, let $V_b(2k+1) \in \SO(2k+1)$ have Jordan form $J_{2k+1}$ and fix the symmetric form with matrix
$$
\begin{pmatrix}
 & & \mathbb{I}_k \\ & 2b & \\ \mathbb{I}_k & & 
\end{pmatrix}.
$$
For $l \geq 1$, let $W(2l) \in \SO^+(4l,q)$ have Jordan form $J_{2l}^{\oplus 2}$.

By Theorem \ref{elementsofCG} all Jordan block of even dimension occur with even multiplicity, so the Jordan form of every unipotent element of $\Or^{\epsilon}(n,q)$ is
\begin{eqnarray} \label{unirepGO}
\bigoplus_{i=1}^r (J_{2k_i+1})^{\oplus a_i} \oplus \bigoplus_{j=1}^s (J_{2l_j})^{\oplus 2c_j},
\end{eqnarray}
where the $k_i$ and the $l_j$ are distinct and $\sum a_i(2k_i+1)+ 4\sum l_jc_j = n$, with $\epsilon \in \{ +, -, \circ \}$.

Recall that every unipotent element of $\SO^{\epsilon}(n,q)$ lies in $\Omega^{\epsilon}(n,q)$ (see Lemma \ref{oddOrderinOmega}).

\begin{proposition} \label{ThmUnirepsGO}
Suppose $q$ is odd. Let $\alpha$ be a fixed non-square in $\mathbb{F}_q$.
\begin{enumerate}
\item If $n$ is even, then the unipotent elements with Jordan form $(\ref{unirepGO})$ fall into $2^{r-1}$ classes in each of $\Or^+(n,q)$ and $\Or^-(n,q)$, with the exception that if $r=0$, there is one class in $\Or^+(n,q)$ and none in $\Or^-(n,q)$. Representatives are
\begin{align} \label{unirepsGO}
\bigoplus_{i=1}^r (V_{b_i}(2k_i+1) \oplus V_1(2k_i+1)^{\oplus a_i-1}) \oplus \bigoplus_{j=1}^s W(2l_j)^{\oplus c_j},
\end{align}
where $b_i \in \{1, \alpha \}$. If $u$ is such a representative, then the conjugacy class of $u$ in $\Or^{\epsilon}(n,q)$ splits into two distinct classes in $\SO^{\epsilon}(n,q)$ if, and only if, $r=0$ (namely if $u$ is sum of Jordan blocks of even dimension).
\item If $n$ is odd, there are $2^{r-1}$ classes in $\Or(n,q)$ with Jordan form $(\ref{unirepGO})$; representatives are as in $(\ref{unirepsGO})$, half of these fix an orthogonal form of square discriminant, and half fix a form of non-square discriminant.

\item A conjugacy class in $\SO^{\epsilon}(n,q)$ with representative $u$ as in $(\ref{unirepsGO})$ splits into two distinct classes in $\Omega^{\epsilon}(n,q)$ if, and only if, either $r=0$, or $r \geq 1$ and the following hold:
\begin{description}
\item[$\mathrm{(a)}$] $a_i=1$ for all $i$,
\item[$\mathrm{(b)}$] the $b_i(-1)^{k_i}$ are mutually congruent modulo $\mathbb{F}_q^{*2}$.
\end{description}
In case of splitting, representatives of the $\Omega^{\epsilon}(n,q)$ classes are $u$ and $u^z$, where $z \in \SO^{\epsilon}(n,q) \setminus \Omega^{\epsilon}(n,q)$.
\end{enumerate}
\end{proposition}

\subsection{Symplectic and orthogonal case, $q$ even}       \label{SectUniRepsChar2}
Let $q$ be a power of 2 and let $V$ be a vector space of dimension $2k$ over $F=\mathbb{F}_q$, with basis $e_1, \dots, e_k, f_k, \dots, f_1$. Let $\beta(u,v) = (u,v)$ be a non-degenerate alternating form on $V$ such that $(e_i,e_j)=(f_i,f_j)=0$, $(e_i,f_j)=\delta_{ij}$ for all $i,j$, where $\delta_{ij}$ is the Kronecker delta: $\delta_{ij}=1$ if $i=j$, 0 otherwise\index{$\delta_{ij}$}.

For $b \in F$, let $Q_b$ be the quadratic form on $V$ with associated bilinear form $\beta$ and satisfying
$$
Q_b(e_k)=b, \quad Q_b(f_k)=1, \quad Q_b(e_i)=Q_b(f_i)=0 \: \mbox{for } i \neq k.
$$
One can check that $\C(Q_0) = \Or^+(2k,q)$ and $\C(Q_{\alpha}) = \Or^-(2k,q)$ whenever the polynomial $t^2+t+\alpha$ is irreducible in $F[t]$. Let $V(2k) \in \C(Q_0)$ and $V_{\alpha}(2k) \in \C(Q_{\alpha})$ have Jordan form $J_{2k}$.

For $k \geq 1$ let $Q$ be the quadratic form associated to $\beta$ and such that $Q(e_i)=Q(f_i)=0$. Let $\Omega^+(2k,q)$ be the Omega group corresponding to $\C(Q)$ and let $W(k) \in \Omega^+(2k,q)$ with Jordan form $J_{k}^{\oplus 2}$. If $k=2l$ is even, then let $W'(k)$ be the conjugate of $W(k)$ by the reflection in $e_k+f_k$.
If $k=2l+1$ is odd, then let $Q_b'$ be the quadratic form associated to $\beta$ such that
$$
Q_b'(e_l)=Q_b'(e_{l+1}) = Q_b'(f_{l+1}) = b
$$
and $Q_b'(e_i)=Q_b'(f_i)=0$ for all other values of $i$. Let $W_b(2l+1)$ be an element of $\Omega^{\epsilon}(4l+2,q)$ (the Omega group corresponding to $Q_b'$) with Jordan form $J_{k}^{\oplus 2}$. In particular, $\epsilon = +$ if $b=0$ and $\epsilon=-$ if $t^2+t+b$ is irreducible in $F[t]$. (There is a typographical error in the definition of $W_b(2l+1)$ in \cite{GLOB}; see \cite[p.\ 95]{thesis} for correction.) 

The main result is the following.
\begin{proposition} \label{unirepsGO2}
Let $\alpha \in \mathbb{F}_q$ such that $t^2+t+\alpha$ is irreducible. Let $\C$ be a symplectic group $\Sp(2m,q)$ or an orthogonal group $\Or^{\epsilon}(2m,q)$, with $m \geq 1$. Every unipotent element of $\C$ is $\C$-conjugate to exactly one element of the form
\begin{align} \label{unireps2}
\bigoplus_i W(m_i)^{\oplus a_i} \oplus \bigoplus_j V(2k_j)^{\oplus c_j} \oplus \bigoplus_r W_{\alpha}(m_r') \oplus \bigoplus_s V_{\alpha}(2k_s')
\end{align}
satisfying the following conditions:
\begin{enumerate}
\item $\sum a_im_i + \sum c_jk_j + \sum m_r'+ \sum k_s' = m$,
\item the $m_r'$ are odd and distinct, and the $k_s'$ are distinct,
\item $c_j \leq 2$, and $c_j \leq 1$ if there exist $j,s$ such that $k_j = k_s'$,
\item there exist no $j,s$ such that $k_s'-k_j=1$ or $k_s'-k_j'=1$,
\item there exist no $j,r$ such that $m_r' = 2k_j \pm 1$ or $m_r' = 2k_j' \pm 1$,
\item for $\C=\Sp(2m,q)$, each $m_r' \geq 3$ and each $k_s' \geq 2$.
\end{enumerate}
In the orthogonal case, an element of the form $(\ref{unireps2})$ lies in $\Or^{\epsilon}(2m,q)$, where $\epsilon = (-1)^t$ and $t$ is the total number of $W_{\alpha}$- and $V_{\alpha}$-blocks; the element lies in $\Omega= \Omega^{\pm}(2m,q)$ if, and only if, the total number of $V$- and $V_{\alpha}$-blocks is even; moreover, the only $\C$-classes which split into two distinct classes in $\Omega$ are those of the form $\bigoplus W(m_i)^{a_i}$ with all $m_i$ even, and for these a second class representative can be obtained by replacing one summand $W(m_i)$ by $W(m_i)'$.
\end{proposition}

\section{Centralizer of a unipotent element} \label{CentrUni}
The next theorems describe the centralizer structure of a unipotent element. For convenience, we consider separately the symplectic and orthogonal cases of even characteristic.
\begin{theorem} \label{CentrUniOdd}
Let $\C$ be a unitary group over $F=\mathbb{F}_{q^2}$ or a symplectic or orthogonal group over $F=\mathbb{F}_q$, where $q$ is odd if $\C$ is symplectic or orthogonal. Let $x \in \C$ be unipotent with Jordan form $\bigoplus_{i=1}^s (J_{n_i})^{\oplus r_i}$, with $n_1 < n_2 < \cdots < n_s$. Define
$$
\gamma = \sum_{i<j} n_ir_ir_j + \frac{1}{2}\sum_i (n_i-1)r_i^2+\frac{\delta}{2} \sum_{i \in Z} r_i,
$$
where $\delta = 0$ (unitary), $1$ (symplectic) or $-1$ (orthogonal), and $Z = \{ i \, | \, n_i \mbox{ even}\}$. The centralizer of $x$ in $\C$ is $ U \rtimes R$, where $|U|=|F|^{\gamma}$ and $R$ is defined as follows.
\begin{itemize}
\item Let $x \in \U(n,q)$, with
$$
x = \bigoplus_{i=1}^s V(n_i)^{\oplus r_i}
$$
as in Proposition $\ref{uniconjU}$. Then $R \cong \prod_{i=1}^s \U(r_i,q)$. The centralizer of $x$ in $\Spec = \SU(n,q)$ is $U \rtimes R'$, where $R' = R \cap \Spec$.
\item Let $x \in \Sp(n,q)$, with
$$
x = \bigoplus_{i=1}^r (V_{b_i}(2k_i) \oplus V(2k_i)^{\oplus a_i-1}) \oplus \bigoplus_{j=1}^s W(2l_j+1)^{\oplus c_j}
$$
as in Proposition $\ref{uniconjSp}$. Then
$$
R \cong \prod_{i=1}^r \Or^{\epsilon}(a_i,q) \times \prod_{j=1}^s \Sp(2c_j,q),
$$
where $\Or^{\epsilon}(a_i,q)$ is the orthogonal group preserving the form with $a_i \times a_i$ diagonal matrix
$$
\begin{pmatrix}
b_i &&& \\ & 1 && \\ && \ddots & \\ &&& 1
\end{pmatrix}.
$$
\item Let $x \in \Or^{\epsilon}(n,q)$, with
$$
x = \bigoplus_{i=1}^r (V_{b_i}(2k_i+1) \oplus V_1(2k_i+1)^{\oplus a_i-1}) \oplus \bigoplus_{j=1}^s W(2l_j)^{\oplus c_j},
$$
as in Proposition $\ref{ThmUnirepsGO}$. Then
$$
R \cong \prod_{i=1}^r \Or^{\epsilon}(a_i,q) \times \prod_{j=1}^s \Sp(2c_j,q),
$$
where $\Or^{\epsilon}(a_i,q)$ is the orthogonal group preserving the form with $a_i \times a_i$ diagonal matrix
$$
\begin{pmatrix}
b_i &&& \\ & 1 && \\ && \ddots & \\ &&& 1
\end{pmatrix}.
$$
If $x \in \Spec = \SO^{\epsilon}(n,q)$ or $\Omega = \Omega^{\epsilon}(n,q)$, then the centralizer of $x$ in $\Spec$ (resp.\ $\Omega$) is $U \rtimes R'$, where $R' = R \cap \Spec$ (resp.\ $R \cap \Omega$).
\end{itemize}
\end{theorem}
\begin{proof}
For the groups of isometries, a proof can be found in \cite[Thm.\ 7.1]{LS} and \cite[\S 2.6]{GEW}. The results for special and Omega subgroups follow from the fact that the subgroup $U$ contains only unipotent elements, and all of these have determinant $1$ and spinor norm $0$.
\end{proof}

\begin{theorem} \label{CentrUniEven}
Let $\C$ be a symplectic or orthogonal group of dimension $n$ over $\mathbb{F}_q$, with $q$ even. Let $x \in \C$ be unipotent,
$$
x = \sum_{i=1}^r W(m_i)^{\oplus a_i} \oplus \sum_{j=1}^s V(2k_j)^{\oplus b_j},
$$
where the sums over the $W$-blocks and the $V$-blocks include also the $W_{\alpha}$-blocks and the $V_{\alpha}$-blocks respectively in $(\ref{unireps2})$. Suppose $k_1>k_2> \cdots >k_s$. Let $N = 2\sum_ia_i+\sum_jb_j$ be the total number of Jordan blocks of $x$. Let $L$ be the sequence of the dimensions of the Jordan blocks of $x$
$$
L = [l_{\nu}: 1 \leq \nu \leq N] = [ \cdots, \overbrace{m_i, \cdots, m_i}^{2a_i\mbox{\footnotesize{-times}}}, \cdots, \overbrace{2k_j, \cdots, 2k_j}^{b_j\mbox{\footnotesize{-times}}},\cdots]
$$
sorted by decreasing order, so that $l_1 \geq l_2 \geq \cdots \geq l_N$. Define $\mathcal{S} \subseteq \{1, \dots, r \}$ to be
$$
\mathcal{S} = \{i \, | \; m_i \mbox{ even }\} \cup \{i\,| \; m_i = 2k_j \pm1 \mbox{ for some } j\} \cup \{ i \,| \; m_i=1 \mbox{ and } \C=\Sp(n,q) \}.
$$
Let
$$
\gamma = \sum_{\nu=1}^N (\nu l_{\nu}-\chi(l_{\nu}))-2\sum_i a_i^2 - \sum_{i \in \mathcal{S}}a_i+\sum_{i \notin \mathcal{S}}a_i,
$$
where $\chi(2k_j)=k_j$ (symplectic) or $k_j+1$ (orthogonal) and
$$
\chi(m_i) = \frac{1}{2} \left\{ \begin{array}{ll} m_i+\sigma+1 & \mbox{if $m_i$ even, $m_i=2k_j$ for some $j$;}\\ m_i+\sigma & \mbox{if $m_i$ odd;}\\ m_i+\sigma-1 & \mbox{if $m_i$ even, $m_i \neq 2k_j$ for all $j$,} \end{array}\right.
$$
with $\sigma = -1$ (symplectic) or $\sigma =1$ (orthogonal). The centralizer of $x$ in $\C$ is $U \rtimes R$, with $|U| = q^{\gamma}$ and
$$
R \cong \prod_{i \in \mathcal{S}} \Sp(2a_i,q) \times \prod_{i \notin \mathcal{S}} \Or^{\epsilon}(2a_i,q) \times \mathbb{Z}_2^{t+\delta},
$$
where $t$ is the number of values of $j$ such that $k_j-k_{j+1} \geq 2$, $\delta \in \{0,1\}$ with $\delta=0$ if, and only if, $s=0$ or $\C = \Sp(n,q)$ and $k_s=1$, and $\Or^{\epsilon}(2a_i,q)$ has minus type if, and only if, there is a component $W_{\alpha}(m_i')$ in $(\ref{unireps2})$.
\end{theorem}
\begin{proof}
See \cite[Lemma 6.2 and Theorem 7.3]{LS}.
\end{proof}
\begin{remark} \label{alg}
In all cases, $R$ is easy to generate, being isomorphic to a direct product of classical groups. The generation of $U$ is much harder. Liebeck and O'Brien \cite{IP} provide an algorithm to construct $U$ (\emph{Algorithm} 3)\index{Algorithm 3}.
\end{remark}
\section{Conjugating element in the unipotent case}  \label{ConjElUni}
Let $\C$ be a group of isometries on the vector space $V$. Given $x \in \C$ unipotent, there is an algorithm (\emph{Algorithm} $4$)\index{Algorithm 4} of Liebeck and O'Brien \cite{IP} that computes a basis for $V$ such that the matrices for $x$ and the form preserved by $\C$ with respect to this basis are those described in Section \ref{UnipotentRepr}. Given unipotent matrices $X,Y \in \C$, the algorithm can compute matrices $P_X$ and $P_Y$ in $\C$ such that $P_XXP_X^{-1}= P_YYP_Y^{-1}= J$, where $J$ is the standard representative of the conjugacy class of $X$ and $Y$. If $Z=P_X^{-1}P_Y$, then $X^Z=Y$.

\chapter{Centralizers and conjugacy classes: the general case}        \label{generalChap}
In this chapter we use the results in the semisimple and unipotent cases to solve the three problems for classical groups in the general case: describe all conjugacy classes, describe the centralizers and compute explicitly a conjugating element.
\section{Conjugacy classes in classical groups}   \label{SectListGen}
\subsection{Conjugacy classes in isometry groups}
Let $F=\mathbb{F}_{q^2}$ in the unitary case, $F=\mathbb{F}_q$ otherwise and let $V$ be an $n$-dimensional vector space over $F$. Let $\C$ be $\C(\beta)$ or $\C(Q)$, where $\beta$ is a non-degenerate alternating, hermitian or symmetric form and $Q$ is a non-singular quadratic form on $V$.

Let $x \in \C$. We know by Lemma \ref{splitSU} that $x=su=us$, with $s$ semisimple and $u$ unipotent. If $x_1 = s_1u_1=u_1s_1$ and $x_2=s_2u_2= u_2s_2$ are conjugate in $\C$, then there exists $z \in \C$ such that $x_1=z^{-1}x_2z$, so $s_1u_1 = (z^{-1}s_2z)(z^{-1}u_2z)$. The terms $z^{-1}s_2z$ and $z^{-1}u_2z$ are semisimple and unipotent respectively, so by uniqueness of the Jordan decomposition, this implies $s_1=z^{-1}s_2z$ and $u_1=z^{-1}u_2z$. In other words, if $x_1$ and $x_2$ are conjugate in $\C$, then their corresponding semisimple and unipotent parts are. So, the strategy to list all conjugacy classes of $\C$ is to list all semisimple classes and, for each representative, list all classes having that fixed semisimple part.

Recall the notation of Definition \ref{ThreeCases}:
\begin{eqnarray*}
	\Phi_1 & = & \{f: f \in F[t] \; | \; f=f^* \mbox{ monic irreducible}, \, \deg{f}=1 \}, \\
	\Phi_2 & = & \{f: f \in F[t] \; | \; f=gg^*, \, g \neq g^* \mbox{ monic irreducible}  \}, \\ 
	\Phi_3 & = & \{f: f \in F[t] \; | \; f=f^* \mbox{ monic irreducible}, \, \deg{f}>1 \}.
\end{eqnarray*}
Each semisimple class of $\C$ can be identified with a pair $(S,B)$, where $S$ is an isometry for the form $B$. Let $f_1, \dots, f_h$ be the \geds of $S$, with $f_i \in \Phi$ and $m_i$ the multiplicity of $f_i$ for every $i=1, \dots, h$. Thus
\begin{eqnarray}  \label{matricessemisimple}
S = \left( \begin{array}{ccc} S_1 && \\ & \ddots & \\ && S_h\end{array}\right), \quad B = \left(\begin{array}{ccc} B_1 && \\ & \ddots & \\ && B_h \end{array}\right),
\end{eqnarray}
where $S_i$ and $B_i$ are the matrices of the restriction of $S$ and $B$ to $\ker(f_i(S))$.

If $f_i \in \Phi_1 \cup \Phi_3$, then $S_i$ is a diagonal join of $m_i$ companion matrices of $f_i$; if $f_i = g_ig_i^* \in \Phi_2$, then
$$
S_i = \left( \begin{array}{cc} Y_i & \\ & Y_i^{*-1} \end{array}\right),
$$
where $Y_i$ is a diagonal join of $m_i$ companion matrices of $g_i$. For the matrix $B$, we choose the following form:
\begin{itemize}
	\item If $f_i \in \Phi_1$, then $S_i$ is a scalar matrix, so every $B_i$ can be chosen.
	\item If $f_i \in \Phi_2$, then choose
	$$
	B_i = \left(\begin{array}{cc} \mathbb{O} & \mathbb{I} \\ \varepsilon\mathbb{I} & \mathbb{O} \end{array}\right), \quad \varepsilon = \left\{ \begin{array}{rl} 1 & \mbox{ if $B$ is hermitian or symmetric}; \\ -1 & \mbox{ if $B$ is alternating}; \\ 0 & \mbox{ if $B$ is quadratic}. \end{array}\right.
	$$
	\item If $f_i \in \Phi_3$, then choose
	\begin{eqnarray}  \label{choiceforBi}
	B_i = \left( \begin{array}{ccc} && B_{f_i} \\ & \adots & \\ B_{f_i} \end{array}\right),
	\end{eqnarray}
	where $B_{f_i}$ is the matrix of a form preserved by the companion matrix of $f_i$, as shown in Section \ref{formsandelements}, and it appears $m_i$ times.
\end{itemize}
Let $x=su=us \in \C$, with $s$ semisimple and $u$ unipotent. Choose a basis such that $s$ and the form, $\beta$ or $Q$, have matrices $S$ and $B$ respectively as described in (\ref{matricessemisimple}). Let $U$ be the matrix of $u$ in such a basis. We know that $U$ belongs to $C_{\C}(S) = \bigoplus_{i=1}^h C_{\C(B_i)}(S_i)$. Thus $U$ is a block diagonal matrix
$$
\left( \begin{array}{ccc} U_1 && \\ & \ddots & \\ && U_h \end{array}\right)
$$
where $U_i$ is the matrix of the restriction of $u$ to $\ker(f_i(x)^{m_i})$.

\begin{proposition} \label{semiuniconj}
	Let $x=su_x$ and $y=su_y$  be elements of $\C$, where $\C = \C(\beta)$ or $\C(Q)$ and $u_x,u_y \in C_{\C}(s)$. Let $f_1, \dots, f_h$ be the \geds of $s$ with multiplicity $m_1, \dots, m_h$. Let $S$ and $B$ be matrices of $s$ and $\beta$ (or $Q$) as in $(\ref{matricessemisimple})$. Let $U_{x,i}, \, U_{y,i}$ be the matrices of the restrictions of $u_x$, $u_y$ respectively to $\ker(f_i(s))$. Now $x$ and $y$ are conjugate in $\C$ if, and only if, $U_{x,i}$ and $U_{y,i}$ are conjugate in $C_{\C(B_i)}(S_i)$ for every $i=1, \dots, h$.
\end{proposition}
\begin{proof}
	$(\Rightarrow)$ Let $z \in \C(B)$ such that $x = z^{-1}yz$. This implies $su_x = z^{-1}szz^{-1}u_yz$. Now $z^{-1}sz$ and $z^{-1}u_yz$ are semisimple and unipotent respectively; thus, by the uniqueness of the Jordan decomposition, $s=z^{-1}sz$ and $u_x=z^{-1}u_yz$. In particular, the first relation implies that $z \in C_{\C}(s)$, so $z$ has matrix
	\begin{eqnarray}  \label{formsofZ}
	Z = \left( \begin{array}{ccc} Z_1 && \\ & \ddots & \\ && Z_h \end{array} \right),
	\end{eqnarray}
	where $Z_i$ is the restriction of $z$ to $\ker(f_i(s))$. This leads immediately to the relations $S_i = Z_i^{-1}S_iZ_i$ and $U_{x,i} = Z_i^{-1}U_{y,i}Z_i$ for each $i=1, \dots, h$.
	
	$(\Leftarrow)$ If there exist $Z_i \in C_{\C(B_i)}(S_i)$ such that $U_{x,i} = Z_i^{-1}U_{y,i}Z_i$ and we take $z \in \C(B)$ with matrix $Z$ defined as in (\ref{formsofZ}), then it is immediate to verify that $z \in C_{\C(B)}(s)$ and $x=z^{-1}yz$.
\end{proof}
\begin{theorem}    \label{CCgen}
	A complete set of representatives for conjugacy classes of $\C$ is described by all pairs of matrices $(SU,B)$, defined by 
	$$
	S = \left(\begin{array}{ccc} S_1 && \\ & \ddots & \\ && S_h\end{array}\right), \quad U = \left( \begin{array}{ccc} U_{1,j} && \\ & \ddots & \\ && U_{h,j} \end{array}\right), \quad B = \left(\begin{array}{ccc} B_1 && \\ & \ddots & \\ && B_h \end{array}\right).
	$$
	where $(S,B)$ runs over all representatives of semisimple conjugacy classes of $\C$ and, for every such $(S,B)$, $U_{i,j}$ runs over all representatives of unipotent conjugacy classes of $C_{\C(B_i)}(S_i)$ for $i=1, \dots, h$.
\end{theorem}
\begin{proof}
	Proposition \ref{semiuniconj} implies that the problem of listing all conjugacy classes of $\C$ with fixed semisimple part $S$ can be reduced to listing representatives for all unipotent classes of $C_{\C(B_i)}(S_i)$. Let us distinguish the three cases.
	\begin{itemize}
		\item $f_i \in \Phi_1$. Now $S_i$ is a scalar matrix and $C_{\C(B_i)}(S_i)$ coincides with $\C(B_i)$. The representatives for unipotent classes of $\C(B_i)$ are given in Section \ref{UnipotentRepr}, and we are free to choose the form $B_i$.
		\item $f_i \in \Phi_2$, $f_i=g_i{g_i}^*$. Now $C_{\C(B_i)}(S_i)$ is isomorphic to $\GL(m_i,E)$, with $E = F[t]/(g_i)$, via the isomorphism
		$$
		y \mapsto \left(\begin{array}{cc} Y & \\ & Y^{*-1} \end{array}\right), \: \forall y \in \GL(m_i,E),
		$$
		where $Y$ is the embedding of $y$ into $\GL(m_id_i,F)$. Two elements of $\GL(m_i,E)$ are conjugate if, and only if, they have the same generalized elementary divisors, so the list of representatives of unipotent classes of $C_{\C(B_i)}(S_i)$ is just the list of the isomorphic images in $C_{\C(B_i)}(S_i)$ of representatives of the unipotent classes of $\GL(m_i,E)$. We choose the diagonal join of unipotent Jordan blocks as our preferred form.
		\item $f_i \in \Phi_3$. Let $E = F[t]/(f_i)$. By the results of Section \ref{semisimplesesquilinear}, $C_{\C(B_i)}(S_i)$ is the set of embeddings into $\GL(m_id_i,F)$ of all matrices of the group $\U(m_i,E)$ preserving the hermitian form with matrix
		$$
		\left( \begin{array}{ccc} && 1\\ & \adots & \\ 1 && \end{array} \right).
		$$
		This follows from our choice of $B_i$ in (\ref{choiceforBi}). Representatives for all unipotent classes of $\U(m_i,E)$ are given in Section \ref{UnipotentRepr}. \qedhere
	\end{itemize}
\end{proof}

\subsection{Conjugacy classes in special groups}
Let $\C = \Or^{\epsilon}(n,q)$ or $\U(n,q)$, with $q$ odd in the orthogonal case, $\epsilon \in \{ -,\circ,+\}$ and let $\Spec$ be the corresponding special group. If $\C$ is unitary, let $\omega$ be a primitive element of $\mathbb{F}_{q^2}$. Let $x \in \Spec$. We proved in Section \ref{sectSpecial} that the conjugacy class of $x$ in $\C$ splits into $r$ distinct classes in $\Spec$ with representatives $x, x^z, x^{z^2}, \dots, x^{z^{r-1}}$, where $z \in \C$ has determinant $-1$ (in the orthogonal case) or $\omega^{q-1}$ (in the unitary case) and 
$$
r = \left\{ \begin{array}{rl} \frac{2}{|C_{\C}(x) : C_{\Spec}(x)|} & \mbox{ if $\C$ is orthogonal;} \\ \\ \frac{q+1}{|C_{\C}(x) : C_{\Spec}(x)|} & \mbox{ if $\C$ is unitary.} \end{array}\right.
$$
Thus, we reduce the problem to finding the index $|C_{\C}(x): C_{\Spec}(x)|$.
\begin{theorem} \label{CCinSpecial}
	Let $x \in \Spec$. Let $f_1^{m_1}, \dots, f_k^{m_k}$ be the \geds of $x$, with $f_i \in \Phi$ not necessarily distinct.
	
	If $\C$ is orthogonal, then the conjugacy class of $x$ in $\C$ splits into two distinct classes in $\Spec$ if, and only if, $x$ has no elementary divisors $(t \pm 1)^m$ with $m$ odd.
	
	If $\C$ is unitary, then the conjugacy class of $x$ in $\C$ splits into $r$ distinct classes in $\Spec$, where
	$$
	r = \gcd (m_1, \dots, m_k, q+1).
	$$
\end{theorem}
\begin{proof}
	Suppose first that $\C$ is orthogonal. The class of $x$ in $\C$ splits into two distinct classes in $\Spec$ if, and only if, $C_{\C}(x) \subseteq \Spec$. Let $x=su$ be the Jordan decomposition of $x$. If $x^2-\mathbf{1}_V$ is non-singular, then $C_{\C}(x) \subseteq C_{\C}(s) \subseteq \Spec$ by results of Section \ref{sectSpecial}. So, we can suppose that $x$ has powers of $(t-1)$ and $(t+1)$ as its unique generalized elementary divisors. If all of these powers are even, then $C_{\C}(x) \subseteq C_{\GL(V)}(x) \subseteq \SL(V)$ by results of Section \ref{SLconjclasses}. Conversely, if $x$ has an elementary divisor $(t \pm1)^m$ with $m$ odd, then $C_{\C}(x)$ contains elements of determinant $-1$. In an appropriate basis, $x$ has matrix $X$ and preserves the symmetric form $B$ with
	$$
	X= \left( \begin{array}{ccc} X_1 && \\ & \ddots & \\ && X_k \end{array}\right), \quad B= \left( \begin{array}{ccc} B_1 && \\ & \ddots & \\ && B_k \end{array}\right),
	$$
	where $X_i$ is the matrix of the restriction of $x$ to the cyclic submodule of $V$ corresponding to the \ged $f_i^{m_i}$. Suppose without loss of generality that $f_1(t)^{m_1} = (t \pm 1)^m$ with $m$ odd. The matrix
	$$
	\left(\begin{array}{cc} -\mathbb{I}_m & \\ & \mathbb{I}_{m'} \end{array}\right) \quad (m' = m_2+\cdots + m_k)
	$$
	belongs to $C_{\C(B)}(x)$ and has determinant $(-1)^m = -1$.
	
	Now suppose that $\C$ is unitary. For convenience, suppose that $x$ has a unique \ged $f^m$. Let $r = \gcd(m,q+1)$. We prove that $C_{\C}(x)$ contains elements with determinant $\omega^{m(q-1)}$, so $|C_{\C}(x) : C_{\Spec}(x)| = \frac{q+1}{r}$.
	
	If $f \in \Phi_1$, $f(t) = t-\lambda$, then we may suppose that $x$ has matrix $X$ and each element of $C_{\C}(x)$ has upper triangular matrix $Y$ with constant diagonals:
	$$
	X = \left(\begin{array}{cccc} \lambda & 1 && \\ & \ddots & \ddots & \\ && \ddots & 1 \\ &&& \lambda \end{array}\right), \quad Y = \left(\begin{array}{cccc} \mu & * & * & * \\ & \ddots & * & * \\ && \ddots & * \\ &&& \mu \end{array}\right).
	$$
	The condition $Y \in \C$ implies that $\mu$ is a multiple of $\omega^{q-1}$, so if we choose $\mu = \omega^{q-1}$, then $\det{Y}=\omega^{m(q-1)}$.
	
	If $f \in \Phi_2$, and $y \in C_{\C}(x)$, then in an appropriate basis, $x$ and $y$ have matrices $X$ and $Y$ respectively, with
	$$
	X = \left( \begin{array}{cc} X_1 & \\ & X_1^{*-1} \end{array}\right), \quad Y= \left(\begin{array}{cc} Y_1 & \\ & Y_1^{*-1} \end{array}\right).
	$$
	As we have seen in Section \ref{SLconjclasses}, $\det{Y_1}$ can assume every multiple of $\omega^m$; if we choose $\det{Y_1} = \omega^{-m}$, then $Y$ has determinant $\omega^{-m} \cdot \omega^{mq} = \omega^{m(q-1)}$.
	
	If $f \in \Phi_3$ and $y \in C_{\C}(x)$, then $y$ is the embedding into $\GL(V)$ of a certain $\widetilde{y} \in \U(m,q^d)$, with $d=\deg{f}$. If $\alpha$ is a primitive element of $\mathbb{F}_{q^{2d}}$, then $\widetilde{y}$ has determinant a multiple of $\alpha^{m(q^d-1)}$ by the same argument as for the case $f \in \Phi_1$. So, by choosing an appropriate $\alpha$,
	\begin{eqnarray*}
		\det{y} = \N_{q^{2d}|q^2}(\det{\widetilde{y}}) = \left(\alpha^{m(q^d-1)}\right)^{(q^{2d}-1)/(q^2-1)}=\omega^{m(q-1)}.
	\end{eqnarray*}\\[-1.32cm]
\end{proof}

\subsection{Conjugacy classes in $\Omega^{\epsilon}(\MakeLowercase{n,q})$, $q$ odd}
Let $x \in \Omega^{\epsilon}(n,q)$. Write $x=su=us$, with $s$ semisimple and $u$ unipotent. Let $f_1, \dots, f_h$ be the \geds of $s$ with multiplicities $m_1, \dots, m_h$. Choose a basis such that $x$ has matrix $X$ and the form has matrix $B$ defined by
$$
X= \left( \begin{array}{ccc} X_1 && \\ & \ddots & \\ && X_h \end{array}\right), \quad B= \left( \begin{array}{ccc} B_1 && \\ & \ddots & \\ && B_h \end{array}\right),
$$
where $X_i$ and $B_i$ are the matrices of the restrictions of $X$ and $B$ respectively to $\ker(f_i(s))$. Write $X_i = S_iU_i$, with $S_i$ and $U_i$ matrices of the semisimple and unipotent parts respectively. Abbreviate $\C(B)$, $\Spec(B)$ and $\Omega(B)$ by $\C$, $\Spec$ and $\Omega$ respectively.

\begin{theorem}\label{ConjClassinOmega}
	Let $x \in \Spec$, $x=su=us$. Then $x$ lies in $\Omega$ if, and only if, $s$ does. Moreover, the conjugacy class of $x$ in $\Spec$ splits into two distinct classes in $\Omega$ if, and only if, the following conditions hold:
	\begin{itemize}
		\item if $f_i(t) = t \pm 1$, then $X_i$ has shape
		$$
		\pm\left(\bigoplus_{i=1}^r (V_{b_i}(2k_i+1) \oplus V_1(2k_i+1)^{\oplus a_i-1}) \oplus \bigoplus_{j=1}^s W(2l_j)^{\oplus c_j}\right)
		$$
		as in $(\ref{unirepsGO})$, with either $r=0$, or $r \geq 1$ and the following hold:
		\begin{description}
			\item[$\mathrm{(a)}$] $a_i=1$ for all $i$,
			\item[$\mathrm{(b)}$] the $b_i(-1)^{k_i}$ are mutually congruent modulo $\mathbb{F}_q^{*2}$;
		\end{description}
		\item if both $t+1$ and $t-1$ occur in the list $\{ f_1, \dots, f_h \}$ and in each case $r>0$, then the values of the $(-1)^{k_i}b_i \pmod{\mathbb{F}_q^{* 2}}$ must be the same for both $t+1$ and $t-1$;
		\item if $f_i \in \Phi_2 \cup \Phi_3$, then $x$ has no \geds $f_i^m$ with $m$ odd.
	\end{itemize}
\end{theorem}
\noindent The proof of this theorem requires some preliminary work.

Its first assertion follows from the fact that every unipotent element belongs to $\Omega$, since it has odd order. So $x=su$ lies in $\Omega$ if, and only if, $s$ does. Once we list all conjugacy classes of $\Spec$ lying in $\Omega$ (see Lemma \ref{oddOrderinOmega}), we need to establish which of these splits into two distinct classes in $\Omega$. The strategy is that used in the special case: compute the index $|C_{\Spec}(x): C_{\Omega}(x)|$. If the index is 1, then the conjugacy class of $x$ in $\Spec$ splits into two distinct classes in $\Omega$, with representatives $x$ and $x^z$, where $z \in \Spec \setminus \Omega$.
\begin{lemma} \label{lemmaPM}
	If $|C_{\Spec(B_i)}(X_i): C_{\Omega(B_i)}(X_i)|=2$ for at least one $i$, then $|C_{\Spec}(x): C_{\Omega}(x)|=2$.
\end{lemma}
\begin{proof}
	Suppose without loss of generality that $|C_{\Spec(B_1)}(X_1): C_{\Omega(B_1)}(X_1)|=2$. In such a case there exists $Y_i \in C_{\Spec(B_1)}(X_1)$ with spinor norm 1. Thus
	$$
	\left( \begin{array}{cc} Y_1 & \\ & \mathbb{I}_{m'} \end{array}\right) \quad (m'=m_2+m_3+ \cdots +m_h)
	$$
	belongs to $C_{\Spec}(x) \setminus C_{\Omega}(x)$.
\end{proof}
\noindent For convenience, let us consider separately the two cases where all \geds of $x$ belongs to $\Phi_2 \cup\Phi_3$ or to $\Phi_1$. After that, we will analyze the general case.\\

Suppose first that $x^2-\mathbf{1}_V$ is non-singular, that is $f_i \in \Phi_2 \cup \Phi_3$ for every $i$. If $f_i \in \Phi_2$, $f_i=g_i{g_i}^*$, $d_i=\deg{g_i}$, then $C_{\C(B_i)}(S_i) = C_{\Spec(B_i)}(S_i) \cong \GL(m_i,q^{d_i})$ by Theorem \ref{CCinSpecial}. In the analysis of the case $q$ odd, $f_i \in \Phi_2$ in Section \ref{OmegaClasses}, we found that a certain $y \in C_{\C(B_i)}(S_i)$, the image of $\widetilde{y} \in \GL(m_i,q^{d_i})$, belongs to $\Omega(B_i)$ if, and only if, $\widetilde{y}$ belongs to the unique subgroup of $\GL(m_i,q^{d_i})$ of index 2, that is, if $\det{\widetilde{y}}$ is a square in $\mathbb{F}_{q^{d_i}}$. The centralizer of $U_i$ in $C_{\C(B_i)}(S_i)$ contains elements of non-square determinant if, and only if, $U_i$ has at least one elementary divisor of the form $(t-1)^m$ with $m$ odd. If $f_i \in \Phi_3$ and $d_i = \deg{f_i}$, then exactly the same argument holds, substituting $\GL(m_i,q^{d_i})$ by $\U(m_i,q^{d_i/2})$. We conclude that if $x^2-\mathbf{1}_V$ is non-singular, then $|C_{\Spec}(x): C_{\Omega}(x)|=2$ if, and only if, $x$ has at least one \ged $f^m$, with $f \in \Phi_2 \cup \Phi_3$ and $m$ odd.\\

Now suppose that the only \geds of $X$ are powers of $t+1$ and $t-1$. For convenience, use the notation
$$
X = \left( \begin{array}{cc} X_+ & \\ & X_- \end{array}\right), \quad B= \left( \begin{array}{cc} B_+ & \\ & B_- \end{array}\right),
$$
where $(X_+,B_+)$ and $(X_-, B_-)$ are the restrictions to the eigenspaces $\ker(x-\mathbf{1}_V)^{m_+}$ and ${\ker(x+\mathbf{1}_V)^{m_-}}$ respectively. Let $\epsilon \in \{+,-\}$. The Jordan decomposition of $X_{\epsilon}$ is $S_{\epsilon}U_{\epsilon}$, with $S_{\epsilon} = {\epsilon}\mathbb{I}_{m_{\epsilon}}$ and $U_{\epsilon}$ a unipotent element with shape
$$
\epsilon\left(\bigoplus_{i=1}^r \left(V_{b_i}(2k_i+1) \oplus V_1(2k_i+1)^{\oplus a_i-1}\right) \oplus \bigoplus_{j=1}^s W(2l_j)^{\oplus c_j}\right)
$$
described in Proposition \ref{ThmUnirepsGO}. It implies that $C_{\Spec(B_{\epsilon})}(U_{\epsilon}) = C_{\Omega(B_{\epsilon})}(U_{\epsilon})$ if, and only if, either $r=0$, or both $a_i=1$ for all $i$ and the $(-1)^{k_i}b_i$ are mutually congruent modulo $\mathbb{F}_q^{* 2}$.

If $C_{\C(B_{\epsilon})}(U_{\epsilon}) = C_{\Spec(B_{\epsilon})}(U_{\epsilon}) = C_{\Omega(B_{\epsilon})}(U_{\epsilon})$ for at least one $\epsilon \in \{+,-\}$, say $\epsilon = +$, then
$$
C_{\Omega(B)}(X) = \left( \begin{array}{cc} C_{\Omega(B_+)}(U_+) & \\ & C_{\Omega(B_-)}(U_-) \end{array}\right),
$$
and $|C_{\Spec(B)}(X) : C_{\Omega(B)}(X)| = |C_{\Spec(B_-)}(U_-) : C_{\Omega(B_-)}(U_-)|$.

If $|C_{\Spec(B_{\epsilon})}(U_{\epsilon}) : C_{\Omega(B_{\epsilon})}(U_{\epsilon}) | =2$ for at least one $\epsilon$, then, by Lemma \ref{lemmaPM},
$$
|C_{\Spec}(X):C_{\Omega}(X)|=2.
$$

The situation $C_{\C(B_{\epsilon})}(U_{\epsilon}) = C_{\Spec(B_{\epsilon})}(U_{\epsilon})$ and $| C_{\Spec(B_{\epsilon})}(U_{\epsilon}): C_{\Omega(B_{\epsilon})}(U_{\epsilon})| = 2$ can never occur: $C_{\C(B_{\epsilon})}(U_{\epsilon}) = C_{\Spec(B_{\epsilon})}(U_{\epsilon})$ only when $r=0$ in (\ref{unirepsGO}), but in such a case $C_{\Spec(B_{\epsilon})}(U_{\epsilon}) = C_{\Omega(B_{\epsilon})}(U_{\epsilon})$.

Finally, suppose that
\begin{eqnarray*}
	|C_{\C(B_+)}(U_+):C_{\Spec(B_+)}(U_+)| = 2, & & C_{\Spec(B_+)}(U_+) = C_{\Omega(B_+)}(U_+); \\
	|C_{\C(B_-)}(U_-):C_{\Spec(B_-)}(U_-)| = 2, & & C_{\Spec(B_-)}(U_-) = C_{\Omega(B_-)}(U_-).
\end{eqnarray*}
This occurs when both $U_+$ and $U_-$ satisfy the following conditions: $r>0$, $a_i=1$ and $(-1)^{k_i}b_i \equiv (-1)^{k_j}b_j \bmod{\mathbb{F}_q^{* 2}}$ for all $i,j$. Let $K$ be the group of matrices with shape
$$
\begin{pmatrix}
W_+ & \\ & W_-
\end{pmatrix},
$$
where $W_{\epsilon} \in C_{\Spec(B_{\epsilon})}(U_{\epsilon})$. The centralizer of $X$ in $\Spec$ is given by
\begin{eqnarray}  \label{CentralizerSpecOmega}
C_{\Spec}(X) = K\cdot \left\langle \left( \begin{array}{cc} Z_+ & \\ & Z_- \end{array}\right) \right\rangle,
\end{eqnarray}
where $Z_{\epsilon} \in C_{\C(B_{\epsilon})}(U_{\epsilon}) \setminus C_{\Spec(B_{\epsilon})}(U_{\epsilon})$. The centralizer in $\Spec$ coincides with the centralizer in $\Omega$ if, and only if, the matrix
\begin{eqnarray} \label{matrixZ}
Z:=\left( \begin{array}{cc} Z_+ & \\ & Z_- \end{array}\right)
\end{eqnarray}
belongs to $\Omega$, and this happens if, and only if, $Z_+$ and $Z_-$ have the same spinor norm. If such $Z$ does not exist, then the centralizer of $x$ in $\Omega$ is just the group $K$ in (\ref{CentralizerSpecOmega}), so $|C_{\Spec}(X) : C_{\Omega}(X)| = 2$. Thus, the problem is reduced to finding $Z_{\epsilon} \in C_{\C(B_{\epsilon})}(U_{\epsilon})$ with determinant $-1$ and an appropriate spinor norm.
\begin{proposition}
	Let
	$$
	U=\bigoplus_{i=1}^r V_{b_i}(2k_i+1)\oplus \bigoplus_{j=1}^s W(2l_j)^{\oplus c_j}
	$$
	be a unipotent element as in $(\ref{unirepsGO})$, with the $(-1)^{k_i}b_i$ mutually congruent modulo $\mathbb{F}_q^{* 2}$. The centralizer of $U$ in $\C$ contains elements of non-square spinor norm if, and only if, $(-1)^{k_i}b_i$ is a non-square for at least one $i$ (so, for all of them).
\end{proposition}
\begin{proof}
	Without loss of generality we can assume that $s=0$, since $\bigoplus_{j=1}^s W(2l_j)^{\oplus c_j}$ does not affect the spinor norm. Thus, $U$ preserves the form 
	$$
	B = \left( \begin{array}{ccc} B_1 && \\ & \ddots & \\ && B_r \end{array}\right),
	$$
	where $B_i$ is the $(2k_i+1) \times (2k_i+1)$ matrix with $1$ on the antidiagonal except for $2b_i$ in the entry $(k_i+1,k_i+1)$, and the $k_i$ are all different.
	
	If $\theta: C_{\C(B)}(U) \rightarrow \mathbb{F}_q^*/\mathbb{F}_q^{* 2}$ is the spinor norm, then $\ker{\theta}$ has index 1 or 2 in $C_{\C(B)}(U)$. If the index is 2, then $C_{\Spec(B)}(U) = C_{\Omega(B)}(U) \subseteq \ker{\theta}$, and equality holds since they are both subgroups of index 2 in $C_{\C(B)}(U)$; equivalently, $y \in C_{\C(B)}(U)$ has non-square spinor norm if, and only if, $\det{y}=-1$. Hence, it is sufficient to compute the spinor norm of any element of $C_{\C(B)}(U)$ having determinant $-1$. Take, for example,
	$$
	Y = \left( \begin{array}{cc} -\mathbb{I}_{2k_1+1} & \\ & \mathbb{I}_{m'} \end{array}\right)
	$$
	with $m' = \sum_{i=2}^r (2k_i+1)$. Using the notation of Section \ref{OmegaClasses}, $V_Y$ coincides with the cyclic submodule of $V$ relative to the first block $V_{b_1}(2k_1+1)$ and $\chi_Y(u,v) = \frac{1}{2} \beta(u,v)$ for all $u,v \in V_Y$, so
	$$
	\theta(y) = \det{ \left( \frac{1}{2}B_1 \right)} = 2^{-(2k_1+1)}\cdot 2b_1 (-1)^{k_1} = 2^{-2k_1}b_1(-1)^{k_1} \equiv (-1)^{k_1}b_1 \bmod{\mathbb{F}_q^{* 2}}.
	$$
	We conclude that every element of $C_{\C(B)}(U)$ has square spinor norm if, and only if, the $(-1)^{k_i}b_i$ are squares.
\end{proof}
From the last proposition, the matrix $Z$ defined in (\ref{matrixZ}) exists in $\Omega$ if, and only if, the $(-1)^{k_i}b_i$ relative to the forms $B_+$ and $B_-$ are congruent modulo $\mathbb{F}_q^{* 2}$; only in such a case can $Z_+$ and $Z_-$ be taken with the same spinor norm.
\begin{proof}[Proof of Theorem $\ref{ConjClassinOmega}$]
	We saw what happens when every \ged of $x$ belongs to $\Phi_1$ or to $\Phi_2 \cup \Phi_3$. Now consider the general case. Write
	$$
	X = \left( \begin{array}{cc} X_{\pm} & \\ & X_{\circ} \end{array}\right), \quad B=\left( \begin{array}{cc} B_{\pm} & \\ & B_{\circ} \end{array}\right),
	$$
	where $X_{\pm}$ has \geds in $\Phi_1$ and $X_{\circ}$ has \geds in $\Phi_2 \cup \Phi_3$. By Lemma \ref{lemmaPM}, if either $|C_{\Spec(B_{\pm})}(X_{\pm}): C_{\Omega(B_{\pm})}(X_{\pm})|=2$ or $|C_{\Spec(B_{\circ})}(X_{\circ}): C_{\Omega(B_{\circ})}(X_{\circ})|=2$, then $|C_{\Spec(B)}(X): C_{\Omega(B)}(X)|=2$. If both indexes are 1, then
	$$
	C_{\Omega(B)}(X) = \left( \begin{array}{cc} C_{\Spec(B_{\pm})}(X_{\pm}) & \\ & C_{\Spec(B_{\circ})}(X_{\circ}) \end{array}\right),
	$$
	exactly the centralizer of $X$ in $\Spec(B)$ (since $C_{\C(B_{\circ})}(X_{\circ}) = C_{\Spec(B_{\circ})}(X_{\circ})$, there cannot be corrective factors like the $Z$ described in (\ref{matrixZ})).
\end{proof}

\subsection{Conjugacy classes in $\Omega^{\epsilon}(\MakeLowercase{n,q})$, $q$ even}
Let $\C = \Or^{\epsilon}(n,q)$ and $\Omega = \Omega^{\epsilon}(n,q)$, with $n$ and $q$ even and $\epsilon \in \{+,-\}$. Recall that the spinor norm $\theta: \C \rightarrow \mathbb{F}_2$ is defined by $\theta(x) = \rk(x+\mathbf{1}_V) \bmod{2}$. Every semisimple element of $\C$ belongs to $\Omega$ because it has odd order. Hence, listing the conjugacy classes of $\Omega$ has three steps:
\begin{enumerate}
	\item List all semisimple classes $(S,B)$ in $\C$, with
	$$
	S = \left( \begin{array}{ccc} S_1 && \\ & \ddots & \\ && S_h\end{array}\right), \quad B = \left(\begin{array}{ccc} B_1 && \\ & \ddots & \\ && B_h \end{array}\right)
	$$
	defined as in (\ref{matricessemisimple}).
	\item Establish which unipotent elements
	$$
	U = \left(\begin{array}{ccc} U_1 && \\ & \ddots & \\ && U_h \end{array}\right)
	$$
	belong to $\Omega$.
	\item For every such $U$, establish whether the conjugacy class of $U$ in $\C$ splits into two distinct classes in $\Omega$ and, in such a case, add $((SU)^Z,B)$ to the list of representatives, where $Z \in \C(B)$ with $\theta(Z)=1$.
\end{enumerate}
\begin{proposition} \label{splitOmega2}
	Let $f \in \Phi_2 \cup \Phi_3$ and let $x \in \C$ have powers of $f$ as generalized elementary divisors. Then $x \in \Omega$ and the conjugacy class of $x$ in $\C$ splits into two distinct classes in $\Omega$.
\end{proposition}
\begin{proof}
	Suppose for convenience that $x$ has a unique \ged $f^m$, with $m$ a positive integer. We prove that each element of $C_{\C}(x)$ has spinor norm $0$, so showing simultaneously that $x \in \Omega$ and $C_{\C}(x) = C_{\Omega}(x)$, or equivalently that the conjugacy class of $x$ in $\C$ splits into two distinct classes in $\Omega$.
	
	If $f \in \Phi_2$, $f=gg^*$, then we can choose a basis such that each element of the centralizer of $x$ in $\C$ has matrix $Y$ and preserves a quadratic form with matrix $B$, where
	$$
	Y=\left( \begin{array}{cc} Y_1 & \\ & Y_1^{*-1} \end{array}\right), \quad B = \left( \begin{array}{cc} \mathbb{O} & \mathbb{I} \\ \mathbb{O} & \mathbb{O} \end{array}\right).
	$$
	The ranks of $(Y_1-1)$ and $(Y_1^{*-1}-1)$ are the same, so
	$$
	\rk(Y-1) = \rk(Y_1-1)+\rk(Y_1^{*-1}-1) = 2\!\cdot\!\rk(Y_1-1)
	$$
	is even, thus $\theta(Y)=0$.
	
	If $f \in \Phi_3$, and $d=\deg{f}$, then each $y \in C_{\C}(x)$ is an image in $\GL(md,q)$ of a certain $\widetilde{y} \in \U(m,q^{d/2})$. Thus $\rk(y-\mathbf{1}_V) = d\cdot\rk(\widetilde{y}-1)$ is even because $d$ is even. It follows that $y \in \Omega$.
\end{proof}
\begin{theorem}  \label{ConjClassesinOmega2}
	Let $\C = \Or^{\epsilon}(n,q)$ and $\Omega = \Omega^{\epsilon}(n,q)$, with $n$ and $q$ even. A complete set of representatives for conjugacy classes of $\Omega$ is described by all pairs of matrices $(SU,B)$ defined by
	$$
	S = \left(\begin{array}{ccc} S_1 && \\ & \ddots & \\ && S_h\end{array}\right), \quad U = \left( \begin{array}{ccc} U_{1,j} && \\ & \ddots & \\ && U_{h,j} \end{array}\right), \quad B = \left(\begin{array}{ccc} B_1 && \\ & \ddots & \\ && B_h \end{array}\right),
	$$
	where $(S,B)$ runs over all representatives of semisimple conjugacy classes of $\C$ and, for every such $(S,B)$, $U_{i,j}$ runs over all representatives of unipotent conjugacy classes of $C_{\Omega(B_i)}(S_i)$ (if $f_i = t+1$) or in $C_{\C(B_i)}(S_i)$ (if $f_i \in \Phi_2 \cup \Phi_3$). Moreover, if $f_i \in \Phi_2 \cup \Phi_3$ for all $i$, then the element $((SU)^Z,B)$ must be added to the set of representatives for each $SU$, where $Z$ is a fixed element of $\C(B) \setminus \Omega(B)$.
\end{theorem}
\begin{proof}
	If $f_i \in \Phi_2 \cup \Phi_3$ for all $i$, then the theorem is a direct consequence of Proposition \ref{splitOmega2}. If $f_i = t+1$ for some $i$, then without loss of generality suppose $i=1$. A unipotent element
	$$
	U = \left( \begin{array}{ccc} U_1 && \\ & \ddots & \\ && U_h \end{array}\right) \in \C
	$$
	belongs to $\Omega$ if, and only if, $U_1 \in \Omega(B_1)$; the conjugacy class of $U$ in $\C$ splits into two distinct classes in $\Omega$ if, and only if, the conjugacy class of $U_1$ in $C_{\C(B_1)}(S_1) = \C(B_1)$ splits into two distinct classes in $\Omega(B_1)$. Therefore, if $U_1$ runs over the set of representatives of $\Omega(B_1)$, described in (\ref{unireps2}), and $U_j$ runs over all unipotent classes of $C_{\C(B_j)}(S_j)$ for all $j>1$, then $U$ runs over all unipotent conjugacy classes of $C_{\Omega(B)}(S)$.
\end{proof}

\section{Centralizers in classical groups}   \label{SectCentrGen}
We now describe the structure of the centralizer of an arbitrary element of a classical group and give a generating set.

We assume that the following algorithms are available.
\begin{itemize}
	\item (\textit{Algorithm} 1)\index{Algorithm 1} Given $X,Y \in \GL(V)$, we determine explicitly $Z \in \GL(V)$ such that $Z^{-1}XZ=Y$. If $J$ is the Jordan form of $X$ and $Y$ and $J=P_XXP_X^{-1}=P_YYP_Y^{-1}$ for $P_X, P_Y \in \GL(V)$, then $Z = P_X^{-1}P_Y$. This algorithm is described in \cite{JF}.
	\item (\textit{Algorithm} 2)\index{Algorithm 2} Given matrices $B_1,B_2$ of two non-degenerate sesquilinear or quadratic forms on $V$, we determine explicitly $T \in \GL(V)$ such that $TB_1T^*=B_2$ (or $TB_1T^*-B_2$ is alternating in the case of quadratic forms). This algorithm is described in \cite{GSM}.
	\item (\textit{Algorithm} 3) Given a unipotent $X \in \C$, return a generating set for $C_{\C}(X)$. It is referenced in Remark \ref{alg}.
\end{itemize}
\begin{theorem} \label{CentrGen}
	Let $\C$ be a classical group in characteristic $p$ preserving a non-degenerate sesquilinear or quadratic form $\beta$ and let $x \in \C$. The centralizer $C_{\C}(x)$ is a semidirect product $U \rtimes R$, where $U$ is a $p$-group and $R$ is (isomorphic to) a direct product of classical groups (here we identify $\mathbb{Z}_2$ with $\Or^+(2,2)$).
\end{theorem}

\begin{proof}
	Let
	\begin{eqnarray}    \label{ElDivCentr}
	[ f_i^{m_{i,j}}: \, i=1, \dots, h, \, 1 \leq j \leq k_i]
	\end{eqnarray}
	be the list of generalized elementary divisor, with $m_{i,1} \geq m_{i,2} \geq \cdots \geq m_{i,k_i}$. Let $x_i$ (resp.\ $\beta_i$) be the restriction of $x$ (resp.\ $\beta$) to $\ker{f_i(x)^{m_{i,1}}}$. Let $d_i = \deg{f_i}$ and let $m_i = \sum_j m_{i,j}$. Let $x=su=us$ and $x_i=s_iu_i=u_is_i$ be the Jordan decompositions of $x$ and $x_i$. As in the linear case, the centralizer of $x$ in $\C$ can be obtained by computing the centralizer of $u$ in $C_{\C}(s)$, and this can be done separately for each $x_i$. The centralizer of $x$ in $\C$ is the direct product of the $C_{\C(\beta_i)}(x_i)$. As usual, let us distinguish the three cases.
	\begin{itemize}
		\item $f_i \in \Phi_1$. Now $s_i$ is a multiple of the identity, so $C_{\C(\beta_i)}(s_i) = \C(\beta_i)$. The centralizer of $u_i$ in $\C(\beta_i)$ is $U_i \rtimes R_i$, with $U_i$ and $R_i$ described in Theorems \ref{CentrUniOdd} and \ref{CentrUniEven}.
		\item $f_i \in \Phi_2$, $f_i = g_i{g_i}^*$. Let $E = F[t]/(g_i)$ and $d_i' = \deg{g_i}$. In an appropriate basis, by Theorem \ref{CentrSem} $C_{\C(\beta_i)}(s_i)$ is the set of matrices
		\begin{eqnarray*}
			\begin{pmatrix}
				Y_i & \\ & Y_i^{*-1}
			\end{pmatrix},
		\end{eqnarray*}
		where $Y_i$ belongs to the embedding of $\GL(m_i,E)$ into $\GL(d_i'm_i,F)$, so $C_{\C(\beta_i)}(x_i) = U_i\rtimes R_i$, where $U_i$ and $R_i$ are the embeddings into $\GL(d_i'm_i,F)$ of the subgroups $U$ and $R$ described in Theorem \ref{GLCentralizer}.
		\item $f_i \in \Phi_3$. Let $E = F[t]/(f_i)$. By Theorem \ref{CentrSem}, $C_{\C(\beta_i)}(s_i)$ is the embedding of $\U(m_i,E)$ into $\GL(d_im_i,F)$, so $C_{\C(\beta_i)}(x_i) = U_i\rtimes R_i$, where $U_i$ and $R_i$ are the embeddings into $\GL(d_im_i,F)$ of the subgroups $U$ and $R$ described in Theorem \ref{CentrUniOdd}.
	\end{itemize}
	From the three cases, one can see easily that $C_{\C}(x) = U \rtimes R$, where
	\begin{empheq}{equation*}
	\begin{split}
	U = \prod_i U_i \mbox{ and } R = \prod_i R_i. \qedhere
	\end{split}
	\end{empheq}
\end{proof}
\begin{remark}
	If $\Spec$ is the special subgroup of $\C$, then $C_{\Spec}(x) = U \rtimes (R \cap \Spec)$. If $F$ has odd characteristic, $\beta$ is a symmetric form and $\Omega = \Omega(\beta)$, then $C_{\Omega}(x) = U \rtimes (R \cap \Omega)$. Both follow directly from the fact that every element of $U$ is unipotent, so it has determinant $1$ and spinor norm $0$.
	
	Assume that $F$ has even characteristic, $x \in \Omega=\Omega(\beta)$, $f_1 = t+1$ and $f_i \in \Phi_2 \cup \Phi_3$ for all $i\geq 2$. By Theorem \ref{semisimpleOmega}, for every $i \geq 2$ each element of $C_{\C(\beta_i)}(x_i)$ has spinor norm $0$, so
	$$
	C_{\Omega}(x) = \Omega \cap \prod_i C_{\C(\beta_i)}(x_i) = (\Omega \cap C_{\C(\beta_1)}(x_1)) \times \prod_{i \geq 2} C_{\C(\beta_i)}(x_i).
	$$
\end{remark}

\subsection{Generators for centralizers in isometry groups}      \label{GenCentrGen}
Let $x \in \C(\beta)$ have \geds as in (\ref{ElDivCentr}). Let $X$ and $B$ be the matrices of $x$ and $\beta$ respectively. If $P \in \GL(V)$, then $PXP^{-1} \in \C(PBP^*)$ and, if $Y_1, \dots, Y_r$ are generators for $C_{\C(PBP^*)}(PXP^{-1})$, then $P^{-1}Y_1P, \dots, P^{-1}Y_rP$ are generators for $C_{\C(B)}(X)$. Hence, using Algorithm 1, we can choose a basis such that $x$ and $\beta$ have matrices
$$
X=\begin{pmatrix}
X_1 && \\ & \ddots & \\ &&  X_h
\end{pmatrix} \quad \mbox{and} \quad B=\begin{pmatrix}
B_1 && \\ & \ddots & \\ && B_h
\end{pmatrix},
$$
where $X_i$ and $B_i$ are the matrices of the restriction of $x$ and $\beta$ respectively to $\ker(f_i(x)^{m_{i,1}})$. Let $m_i = \sum_{j=1}^{k_i} m_{i,j}$ for every $i$. We can suppose that $X_i$ is a Jordan form if $f_i \in \Phi_1 \cup \Phi_3$, or
\begin{eqnarray}    \label{formOfX}
X_i = \begin{pmatrix}
\widehat{X}_i & \\ & \widehat{X}_i^{*-1}
\end{pmatrix}
\end{eqnarray}
if $f_i = g_i{g_i}^* \in \Phi_2$, where $\widehat{X}_i$ is the Jordan form of the matrix of the restriction of $x$ to $\ker{g_i(x)^{m_{i,1}}}$. Let
$$
E_i = \left\{ \begin{array}{ll} F & \mbox{if } f_i \in \Phi_1 \\ F[t]/(g_i) & \mbox{if } f_i \in \Phi_2, \, f_i=g_i{g_i}^*\\ F[t]/(f_i) & \mbox{if } f_i \in \Phi_3. \end{array}\right.
$$
Finally, let 
$$
\begin{pmatrix}
X_1 && \\ & \ddots & \\ && X_h
\end{pmatrix} = \begin{pmatrix}
S_1 && \\ & \ddots & \\ && S_h
\end{pmatrix} \begin{pmatrix}
U_1 && \\ & \ddots & \\ && U_h
\end{pmatrix}
$$
be the Jordan decomposition of $X$, with
$$
\begin{pmatrix}
\widehat{X}_i & \\ & \widehat{X}_i^{*-1}
\end{pmatrix} = \begin{pmatrix}
\widehat{S}_i & \\ & \widehat{S}_i^{*-1}
\end{pmatrix}\begin{pmatrix}
\widehat{U}_i & \\ & \widehat{U}_i^{*-1}
\end{pmatrix}
$$
when $f_i \in \Phi_2$. A generating set for $C_{\C}(X)$ consists of the matrices
\begin{eqnarray}            \label{GenForGen}
y_{i,j} = \begin{pmatrix}
\mathbb{I}_< && \\ & Y_{i,j} & \\ && \mathbb{I}_>
\end{pmatrix},
\end{eqnarray}
where $\mathbb{I}_<$ and $\mathbb{I}_>$ are identity matrices of dimension $\sum_{l<i} m_ld_l$ and $\sum_{l>i} m_ld_l$ respectively, and $Y_{i,j}$ runs over a generating set for $C_{\C(B_i)}(X_i)$. These are obtained as follows.
\begin{itemize}
	\item $f_i \in \Phi_1$. The $Y_{i,j}$ are the generators for $C_{\C(B_i)}(U_i)$ returned by Algorithm $3$.
	\item $f_i \in \Phi_2$. Let $d_i' = \deg{f_i}/2$. By Lemma \ref{L26w}, the form preserved by $X_i$ is
	$$
	B_i = \begin{pmatrix}
	\mathbb{O} & A_i \\ \varepsilon A_i^* & \mathbb{O}
	\end{pmatrix},
	$$
	where $\varepsilon= -1$ in the symplectic case and $1$ otherwise. Now $\widehat{U}_i$ is the embedding of a unipotent $\widetilde{U}_i \in \GL(m_i,E_i)$. We take
	\begin{eqnarray}         \label{GeneratoriPhi2}
	Y_{i,j} = \begin{pmatrix}
	Z_{i,j} & \\ & A_i^*Z_{i,j}^{*-1}A_i^{*-1}
	\end{pmatrix},
	\end{eqnarray}
	where the $Z_{i,j}$ are the embeddings into $\GL(m_id_i',F)$ of the generators of $C_{\GL(m_i,E_i)}(\widetilde{U}_i)$, described in Section \ref{Cue}.
	\item $f_i \in \Phi_3$. Let $E = F[t] /(f_i)$. We follow the argument in the analysis of Case 3 in Section \ref{semisimplesesquilinear}. Let $R$ be the companion matrix of $f_i$ and let $\varepsilon = -1$ if $B$ is alternating, $\varepsilon=1$ otherwise. We can suppose that $S_i$ is the direct sum of $m_i$ copies of $R$. Using Algorithm 1 we find $T$ such that $R^*=T^{-1}R^{-1}T$, and by Lemma \ref{L210w} we can choose $T$ such that $T = \varepsilon T^*$. Let $\mathcal{T}$ be the direct sum of $m_i$ copies of $T$. The matrix $H_i = B_i\mathcal{T}^{-1}$ lies in the centralizer of $S_i$, so it is the embedding into $\GL(m_id_i,F)$ of $\widetilde{H}_i \in \GL(m_i,E)$. By Theorem \ref{L29w}, $\widetilde{H}_i$ is hermitian and $U_i$ is the embedding into $\GL(m_id_i,F)$ of a unipotent $\widetilde{U}_i \in \C(\widetilde{H}_i) \cong \U(m_i,E)$. So, $C_{\C(B_i)}(X_i)$ is generated by the embeddings into $\GL(m_id_i,F)$ of the generators of $C_{\C(\widetilde{H}_i)}(\widetilde{U}_i)$ returned by Algorithm 3.
\end{itemize}
In the analysis of the cases $f_i \in \Phi_2 \cup \Phi_3$, if $Q_i$ is a quadratic form, then it can be replaced by the associated bilinear form.
\begin{remark}      \label{Remarkutile}
	Use the notation of the three cases described above. Suppose $\C(B_i)$ is a unitary group and $f_i \in \Phi_2$. If $\widetilde{Z} \in \GL(m_i,E_i)$ and $Z$ is its embedding into $\GL(m_id_i',F)$, then
	$$
	\det \begin{pmatrix}
	Z & \\ & A^*Z^{*-1}A^{*-1}
	\end{pmatrix} = \det{Z}^{1-q} = \det{\widetilde{Z}}^{(1-q)(q^{2d_i'}-1)/(q^2-1)} = \det{\widetilde{Z}}^{-(q^{2d_i'}-1)/(q+1)}.
	$$
	It follows that $\det\left( \begin{smallmatrix} Z & \\ & A^*Z^{*-1}Z^{*-1} \end{smallmatrix}\right) = 1$ if, and only if, $\det{\widetilde{Z}}$ has order divisible by $q+1$. Hence, to generate $C_{\Spec(B_i)}(X_i)$, we need to take the generators of the centralizer of $\widetilde{U}_{i,j}$ in the subgroup of $\GL(m_i,E_i)$ of index $q+1$ and build up the matrices $Y_{i,j}$ as in (\ref{GeneratoriPhi2}). Similarly, if $f_i \in \Phi_3$, then to generate $C_{\Spec(B_i)}(X_i)$ we need to take the generators of the centralizer of $\widetilde{U}_{i,j}$ in the subgroup of $\C(\widetilde{H}_i)$ of index $q+1$ and take their embeddings into $\GL(m_id_i,F)$.
	
	If $\C$ is a general linear or a unitary group and $\Spec$ is the corresponding special subgroup, then in both cases we need to compute the centralizer of some unipotent $\widetilde{U}$ in a group $K$, where $\Spec \leqslant K \leqslant \C$. Algorithm 3 computes centralizers in $\C$ and $\Spec$. To generate $C_K(\widetilde{U})$ we take the generators of $C_{\Spec}(\widetilde{U})$ and add to the generating set an element of $C_{K}(\widetilde{U})$  with determinant of maximum order. This element can be chosen by considering the determinant map ${\det: C_{\C}(\widetilde{U}) \rightarrow F^*}$, computing the preimage $W$ of a generator of the cyclic group $\det(C_{\C}(\widetilde{U}))$, and taking an appropriate power of $W$.
\end{remark}
\subsection{Generators for centralizers in special groups}
We now consider how to obtain a generating set for the centralizer of an element of a special group. Suppose first that $\Spec(B)$ is the special orthogonal group. Two cases occur.
\begin{itemize}
	\item If $f_i \in \Phi_2 \cup \Phi_3$ for all $i$, or $m_{i,j}$ is even for every $f_i \in \Phi_1$, then $C_{\C(B)}(X) = C_{\Spec(B)}(X)$.
	\item Suppose without loss of generality that $f_1= t \pm 1$ and there is an odd $m_{1,j}$. In this case $C_{\C(B_1)}(U_1)$ contains elements of determinant $-1$. Define
	\begin{eqnarray}         \label{GenforSpecial1}
	y_{1,j} = \begin{pmatrix}
	Y_{1,j} & \\ & \mathbb{I}_>
	\end{pmatrix},
	\end{eqnarray}
	where $Y_{1,j}$ runs over the generators of $C_{\Spec(B_1)}(U_1)$ returned by Algorithm 3. For every $i>1$, define
	\begin{eqnarray}         \label{GenforSpecial2}
	y_{i,j} = \begin{pmatrix}
	H_{i,j} &&& \\ & \mathbb{I}_< && \\ && Y_{i,j} & \\ &&& \mathbb{I}_>
	\end{pmatrix},
	\end{eqnarray}
	where $\mathbb{I}_<$ and $\mathbb{I}_>$ are identity matrices of dimension $\sum_{1<l<i} m_ld_l$ and $\sum_{l>i} m_ld_l$ respectively, $Y_{i,j}$ runs over the generators of $C_{\C(B_i)}(X_i)$ described in Section \ref{GenCentrGen} and $H_{i,j} \in C_{\C(B_1)}(U_1)$ is chosen to ensure $\det{y_{i,j}}=1$ for every $i,j$ (note $H_{i,j}$ can be chosen among the generators of $C_{\C(B_1)}(U_1)$). Now $C_{\C(B)}(X)$ is generated by the $y_{i,j}$.
\end{itemize}

Suppose that $\Spec(B)$ is the special unitary group. The strategy for generating $C_{\Spec(B)}(X)$ is similar to the special linear case. For every $i$, define
\begin{eqnarray}      \label{definizione1}
y_{i,j} = \begin{pmatrix}
\mathbb{I}_< && \\ & Y_{i,j} & \\ && \mathbb{I}_>
\end{pmatrix},
\end{eqnarray}
where $Y_{i,j}$ runs over all generators for $C_{\Spec(B_i)}(X_i)$ returned by Algorithm 3 (see Remark \ref{Remarkutile} for the case $f_i \in \Phi_2 \cup \Phi_3$). Let $\omega$ be a primitive element of $F^* = \mathbb{F}_{q^2}^*$ and let
$$
d_i = \gcd(q^2-1, m_{i,1}, \dots, m_{i,k_i})
$$
for every $i$. Choose $H_i \in C_{\C(B_i)}(X_i)$ such that $\det{H_i} = \omega^{d_i}$ (they can be obtained as explained in Remark \ref{Remarkutile}). Let ${Z(0) \leqslant \mathbb{Z}_{q^2-1}^h}$ be the set of solutions of the equation $\sum_{i=1}^h x_id_i =0$ in $\mathbb{Z}_{q^2-1}$ and let
$$
(a_{1,1}, \dots, a_{1,h}), \: \dots \: , (a_{r,1}, \dots, a_{r,h})
$$
be generators for $Z(0)$. For every $\lambda=1, \dots, r$, define
\begin{eqnarray}      \label{definizione2}
Z_{\lambda} = \begin{pmatrix}
H_1^{a_{\lambda,1}} && \\ & \ddots & \\ && H_h^{a_{\lambda,h}}
\end{pmatrix}.
\end{eqnarray}
A generating set for $C_{\Spec(B)}(X)$ consists of all the $y_{i,j}$ and the $Z_{\lambda}$ defined in (\ref{definizione1}) and (\ref{definizione2}) respectively.

\subsection{Generators for centralizers in Omega groups}
If $\Omega(B)$ has even characteristic, then $C_{\Omega(B)}(X) = C_{\C(B)}(X)$ except when $f_i \in \Phi_1$ for some $i$, say $i=1$. In this case, $C_{\Omega(B)}(X)$ is generated by the $y_{i,j}$, defined as follows:
$$
y_{1,j} = \begin{pmatrix}
Z_{1,j} & \\ & \mathbb{I}_{>}
\end{pmatrix},
$$
where the $Z_{1,j}$ are the generators for $C_{\Omega(\beta_1)}(U_1)$ given by Algorithm 3, and $y_{i,j}$ are as in (\ref{GenForGen}) for $i>1$. \\

Now let us consider odd characteristic. We first analyze two particular cases.\\
\\
\textbf{Case A.} $h=1$, $f_1 \in \Phi_2 \cup \Phi_3$ and there exists an odd $m_{1,j}$. If $f_1 \in \Phi_2$, $f_1=g_1g_1^*$ and
$$
B_1 = \begin{pmatrix}
\mathbb{O} & A_1 \\ \varepsilon A_1^* & \mathbb{O}
\end{pmatrix},
$$
then $U_1 = \left( \begin{smallmatrix} \widehat{U}_1 & \\ & \widehat{U}_1^{*-1} \end{smallmatrix}\right)$ and $\widehat{U}_1$ is the embedding into $\GL(m_1d_1,F)$ of a unipotent $\widetilde{U}_1 \in \GL(m_1, E_1)$, with $E_1 = F[t]/(g_1)$. Let $K_1$ be the unique subgroup of $\GL(m_1,E_1)$ of index 2. In the proof of Theorem \ref{semisimpleOmega} we have shown that $C_{\Omega(B_1)}(X_1)$ is the set of matrices
$$
\begin{pmatrix}
Y & \mathbb{O} \\ \mathbb{O} & A_1^*Y^{*-1}A_1^{*-1}
\end{pmatrix},
$$
where $Y$ is the embedding into $\GL(m_1d_1,F)$ of an element of $C_{K_1}(\widetilde{U}_1)$.

If $f_1 \in \Phi_3$, let $\widetilde{H}_1$ be the hermitian matrix obtained in the third case of Section \ref{GenCentrGen}, and let $K_1$ be the unique subgroup of index 2 of $\C(\widetilde{H}_1)$. Hence, $C_{\Omega(B_1)}(X_1)$ is the embedding into $\GL(m_1d_1,F)$ of $C_{K_1}(\widetilde{U}_1)$.

In both cases, we can construct a generating set for $C_{K_1}(\widetilde{U}_1)$ as explained in Remark \ref{Remarkutile}.\\
\\
\textbf{Case B.} $h=1$ or $2$, $f_i \in \Phi_1$ for $i=1,2$. The possible cases are the following.
\begin{itemize}
	\item $h=1$. A generating set for $C_{\Omega(B)}(X) = C_{\Omega(B)}(U)$ is returned by Algorithm 3.
	\item $h=2$, $C_{\Spec(B_i)}(U_i) \not\subseteq \Omega(B_i)$ for at least one $i$, say $i=1$. We saw in the proof of Theorem \ref{ConjClassinOmega} that $C_{\C(B_i)}(U_i)$ contains elements of every determinant and spinor norm (recall that $m_1 >1$ because it is even). Hence, $C_{\Omega(B)}(X)$ is generated by matrices with shape
	$$
	\begin{pmatrix}
	Y_{1,j} & \\ & \mathbb{I}_>
	\end{pmatrix} \quad \mbox{ and } \quad \begin{pmatrix}
	H_j & \\ & Y_{2,j}
	\end{pmatrix},
	$$
	where $Y_{1,j}$ runs over the generators of $C_{\Omega(B_1)}(U_1)$, $Y_{2,j}$ runs over the generators of $C_{\C(B_2)}(U_2)$, and $H_j \in C_{\C(B_1)}(U_1)$ are chosen to have the same determinant and spinor norm as $Y_{2,j}$. We can readily construct such $H_j$ from the generating set returned by Algorithm 3.
	\item If $h=2$ and $C_{\Spec(B_i)}(U_i) \subseteq \Omega(B_i)$ for both $i=1,2$, then the generating set has matrices with shape
	$$
	\begin{pmatrix}
	Y_{1,j} & \\ & \mathbb{I}_>
	\end{pmatrix} \mbox{ and } \begin{pmatrix}
	\mathbb{I}_< & \\ & Y_{2,j}
	\end{pmatrix},
	$$
	where $Y_{i,j}$ runs over the generators for $C_{\Omega(B_i)}(U_i)$. If $|C_{\C(B_i)}(U_i): C_{\Spec(B_i)}(U_i)| = 2$ for both $i$, then take arbitrary $Z_i \in {C_{\C(B_i)}(U_i) \setminus C_{\Spec(B_i)}(U_i)}$; if $Z_1$ and $Z_2$ have different spinor norm, then the generating set is complete, otherwise we must add the matrix
	$$
	\begin{pmatrix}
	Z_1 & \\ & Z_2
	\end{pmatrix}.
	$$
\end{itemize}

Now consider the general case. Suppose there exists $f_i \in \Phi_2 \cup \Phi_3$ and $m_{i,j}$ odd. Suppose without loss of generality that $i=h$. Define
$$
\widetilde{y}_{h,j} = \begin{pmatrix}
\mathbb{I}_< & \\ & Y_{h,j}
\end{pmatrix},
$$
where $Y_{h,j}$ runs over the generators of $C_{\Omega(B_h)}(X_h)$ described in Case A. For $i=1, \dots, h-1$, define $y_{i,j}$ as in (\ref{GenforSpecial1}) and (\ref{GenforSpecial2}), and define
$$
z_{i,j} = \begin{pmatrix}
\mathbb{I}_< & \\ & H_{i,j}
\end{pmatrix},
$$
where $H_{i,j} \in C_{\C(B_h)}(X_h)$ has the same spinor norm as $y_{i,j}$. Define $\widetilde{y}_{i,j} = y_{i,j}z_{i,j}$. Now $C_{\Omega(B)}(X)$ is generated by the $\widetilde{y}_{i,j}$.

Now suppose that $m_{i,j}$ is even for every $i,j$ such that $f_i \in \Phi_2 \cup \Phi_3$. For convenience, write
$$
X = \begin{pmatrix}
X_{\pm} & \\ & X_{\circ}
\end{pmatrix}, \quad B = \begin{pmatrix}
B_{\pm} & \\ & B_{\circ}
\end{pmatrix},
$$
where $X_{\pm}$ has \geds in $\Phi_1$ and $X_{\circ}$ has \geds in $\Phi_2 \cup \Phi_3$. A generating set for $C_{\Omega(B)}(X)$ consists of matrices
$$
\begin{pmatrix}
Y_{\pm,j} & \\ & \mathbb{I}_>
\end{pmatrix} \quad \mbox{ and } \quad \begin{pmatrix}
\mathbb{I}_< & \\ & Y_{\circ,j}
\end{pmatrix},
$$
where $Y_{\pm,j}$ runs over the generators of $C_{\Omega(B_{\pm})}(X_{\pm})$ described in Case B, and $Y_{\circ,j}$ runs over the generators of $C_{\C(B_{\circ})}(X_{\circ})$ described in Section \ref{GenCentrGen}.

\section{Conjugating element in classical groups}      \label{SectConjGen}
In this section we first state a theorem which decides when two elements are conjugate in a classical group; it is the result of our analysis in previous sections. We then show how to construct explicitly a conjugating element.\\

Let $F = \mathbb{F}_{q^2}$ in the unitary case, $F = \mathbb{F}_q$ otherwise and let $V \cong F^n$. Let $\omega$ be a primitive element of $F$. Let $\C$ be a classical group on $V$. For every $x \in \C$, let $x_i$ be the restriction of $x$ to $\ker(f_i(x)^{m_i})$, where $\prod_{i=1}^h f_i^{m_i}$ is the minimal polynomial of $x$, with $f_i \in \Phi$. We write $x \sim y$ if $x$ and $y$ are conjugate in $\GL(V)$.
\begin{theorem} \label{ConjGen}
	Let $\C$ be a classical group on $V$. Let $x,y \in \C$, $x \sim y$ with minimal polynomial $\prod_{i=1}^h f_i^{m_i}$, where $f_i \in \Phi$.
	\begin{itemize}
		\item If $\C = \U(n,q)$, then $x$ and $y$ are conjugate in $\C$ if, and only if, $x \sim y$.
		\item If $\C= \Sp(n,q)$ or $\C = \Or^{\epsilon}(n,q)$, then $x$ and $y$ are conjugate in $\C$ if, and only if, $x \sim y$ and, for every $i$ such that $f_i(t) = t \pm 1$, the unipotent parts of $x_i$ and $y_i$ are conjugate in the corresponding symplectic or orthogonal group (see Sections $\ref{SectUniRepsSp}$-$\ref{SectUniRepsChar2}$).
		\item If $\C= \U(n,q)$, $\Spec = \SU(n,q)$, $x,y \in \Spec$, then $x$ and $y$ are conjugate in $\Spec$ if, and only if, $x$ and $y$ are conjugate in $\SL(V)$: namely $x$ and $y$ are conjugate in $\C$ and the conjugating element in $\C$ has determinant a multiple of $\omega^d$, where $d$ is the greatest common divisor of the dimensions of the Jordan blocks of $x$ and $y$.
		\item If $\C = \Or^{\epsilon}(n,q)$, $\Spec = \SO^{\epsilon}(n,q)$ with $q$ odd, $x,y \in \Spec$, then $x$ and $y$ are conjugate in $\Spec$ if, and only if, $x$ and $y$ are conjugate in $\C$ and either they have an elementary divisor $(t \pm 1)^e$ with $e$ odd, or the conjugating element in $\C$ has determinant $1$.
		\item If $\Omega = \Omega^{\epsilon}(n,q)$ with $q$ odd, $x,y \in \Omega$, then $x$ and $y$ are conjugate in $\Omega$ if, and only if, they are conjugate in $\Spec = \SO^{\epsilon}(n,q)$ and either their class in $\Spec$ does not split into two distinct classes in $\Omega$ (see Theorem $\ref{ConjClassinOmega}$) or the conjugating element in $\Spec$ has spinor norm $0$.
		\item If $\Omega = \Omega^{\epsilon}(n,q)$ with $q$ even, $x,y \in \Omega$, then $x$ and $y$ are conjugate in $\Omega$ if, and only if, they are conjugate in $\C=\Or^{\epsilon}(n,q)$ and one of the following holds:
		\begin{enumerate}
			\item $f_i(t) = t+1$ for some $i$ and the conjugacy class of $x_i$ and $y_i$ in $\C(B_i)$ does not split into two distinct classes in $\Omega(B_i)$ (see Proposition $\ref{unirepsGO2}$), where $B_i$ is the restriction of the form of $\C$ to $\ker{(x+\mathbf{1}_V)^{m_i}}$.
			\item The conjugating element in $\C$ has spinor norm $0$.
		\end{enumerate}
	\end{itemize}
\end{theorem}
\subsection{Conjugating element in isometry groups}
Let $V$ be a vector space over the finite field $F$ and let $\C$ be a classical group on $V$. Given conjugate matrices $X,Y \in \C$, we wish to construct explicitly $Z \in \C$ such that $X^Z= Z^{-1}XZ=Y$. We use Algorithms 1 and 2 defined in Section \ref{SectCentrGen} and Algorithm 4, defined in Section \ref{ConjElUni}.

Let $B$ be the matrix of a non-degenerate sesquilinear form on $V$ and let $\C=\C(B)$. Let $X,Y$ be conjugate elements of $\C$. Let $[f_i^{m_{i,j}} \, : \, 1 \leq i \leq h, \, 1 \leq j \leq k_i]$ be the list of \geds of $X$ and $Y$, with $f_i \in \Phi$. Let $m_i = m_{i,1}+ \cdots + m_{i,k_i}$ and let $d_i = \deg{f_i}$. The strategy to compute $Z \in \C$ such that $Z^{-1}XZ=Y$ is the following.
\begin{enumerate}
	\item Compute matrices $P_X$ and $P_Y$ in $\GL(V)$ such that $P_XXP_X^{-1} = P_YYP_Y^{-1} =J$, where
	\begin{eqnarray}          \label{Jstdform}
	J = \begin{pmatrix}
	J_1 && \\ & \ddots & \\ && J_h
	\end{pmatrix}
	\end{eqnarray}
	and $J_i$ is the matrix of the restriction of $J$ to the eigenspace $\ker(f_i(J))$ for every $i$. Note that $J$ is not necessarily the Jordan form of $X$: we can choose the form of each $J_i$ completely freely. Hence, $J$ is an isometry for the sesquilinear forms with matrices $B_X = P_XBP_X^*$ and $B_Y = P_YBP_Y^*$.
	\item Compute $W \in C_{\GL(V)}(J)$ such that $WB_YW^* = B_X$.
	\item Let $Z = P_X^{-1}WP_Y$. A computation shows that $Z \in \C(B)$ and $Z^{-1}XZ=Y$, so $Z$ is the desired conjugating element.
\end{enumerate}
The first step is solved using Algorithm 1. Let us describe the second step. By Lemma \ref{L26w}, the forms $B_X$ and $B_Y$ have block diagonal shape
\begin{eqnarray*}
	B_X = \begin{pmatrix}
		B_{X,1} && \\ & \ddots & \\ && B_{X,h}
	\end{pmatrix}, \quad B_Y = \begin{pmatrix}
		B_{Y,1} && \\ & \ddots & \\ && B_{Y,h}
	\end{pmatrix},
\end{eqnarray*}
and $J_i$ is an isometry for $B_{X,i}$ and $B_{Y,i}$ for every $i$. So we need to find $W_i$ in the centralizer of $J_i$ such that $W_iB_{Y,i}W_i^*=B_{X,i}$ for every $i$, and take $W = W_1 \oplus \cdots \oplus W_h$. Let us distinguish the three cases.
\begin{itemize}
	\item $f_i \in \Phi_1$. Now $J_i$ is a product of a scalar and a unipotent element, so we can suppose that the scalar is the identity (so $J_i$ is unipotent) because it does not affect the computation. Using Algorithm 2, we get $W_{B,i}$ such that $W_{B,i}B_{Y,i}W_{B,i}^* = B_{X,i}$. Now $J_i$ and $W_{B,i}J_iW_{B,i}^{-1}$ are unipotent elements of $\C(B_{X,i})$, so using Algorithm 4 we get ${W_{J,i} \in \C(B_{X,i})}$ such that $W_{J,i}W_{B,i}J_i{W_{B,i}}^{-1}{W_{J,i}}^{-1}=J_i$. Hence, we take $W_i = W_{J,i}W_{B,i}$.
	\item $f_i \in \Phi_2$, $f_i = g_i{g_i}^*$. If we put
	$$
	J_i = \begin{pmatrix}
	\widehat{J}_i & \\ & \widehat{J}_i^*
	\end{pmatrix},
	$$
	where $\widehat{J}_i$ is the restriction of $J_i$ to $\ker{g_i(J)}$ then, by Lemma \ref{L26w},
	$$
	B_{X,i} = \begin{pmatrix}
	\mathbb{O} & A_{X,i} \\ \varepsilon A_{X,i}^* & \mathbb{O}
	\end{pmatrix}, \quad B_{Y,i} = \begin{pmatrix}
	\mathbb{O} & A_{Y,i} \\ \varepsilon A_{Y,i}^* & \mathbb{O}
	\end{pmatrix},
	$$
	for some $A_{X,i},A_{Y,i}$ in the centralizer of $\widehat{J}_i$, where $\varepsilon= -1$ if $B$ is alternating, $1$ otherwise. We take
	$$
	W_i = \begin{pmatrix}
	A_{X,i}A_{Y,i}^{-1} & \mathbb{O} \\ \mathbb{O} & \mathbb{I}
	\end{pmatrix}.
	$$
	\item $f_i \in \Phi_3$. Following the steps of the semisimple case in Section \ref{semisimplesesquilinear}, we work in $\GL(m_i,E)$, with $E = F[t]/(f_i)$. Let $R$ be the companion matrix of $f_i$ and let $\varepsilon = -1$ if $B$ is alternating, $\varepsilon=1$ otherwise. Let $J_i = S_iU_i$ be the Jordan decomposition of $J_i$. We can suppose that $S_i$ is the direct sum of $m_i$ copies of $R$. Using Algorithm 1, we find $T$ such that $R^*=T^{-1}R^{-1}T$, and by Lemma \ref{L210w} we can choose $T$ such that $T = \varepsilon T^*$. Let $\mathcal{T}$ be the direct sum of $m_i$ copies of $T$. The matrices $H_{X,i} = B_{X,i}\mathcal{T}^{-1}$ and $H_{Y,i} = B_{Y,i}\mathcal{T}^{-1}$ lie in the centralizer of $S_i$, so they are the embeddings into $\GL(m_id_i,F)$ of matrices $\widetilde{H}_{X,i}$ and $\widetilde{H}_{Y,i}$ in $\GL(m_i,E)$. These two matrices are hermitian by Theorem \ref{L29w}. Also $U_i$ commutes with $S_i$, so it is the embedding into $\GL(m_id_i,F)$ of some unipotent $\widetilde{U}_i \in \GL(m_i,E)$. Moreover, $\widetilde{U}_i$ is an isometry for both $\widetilde{H}_{X,i}$ and $\widetilde{H}_{Y,i}$. Using Algorithm 2, find $\widetilde{W}_{B,i} \in \GL(m_i,F)$ such that $\widetilde{W}_{B,i}\widetilde{H}_{Y,i}\widetilde{W}_{B,i}^{\dag} = \widetilde{H}_{X,i}$. Hence, $\widetilde{W}_{B,i}\widetilde{U}_i\widetilde{W}_{B,i}^{-1}$ and $\widetilde{U}_i$ are unipotent elements lying into $\C(\widetilde{H}_{X,i}) \cong \U(m_i,E)$. So, using Algorithm 4 we can find $\widetilde{W}_{J,i} \in \C(\widetilde{H}_{X,i})$ such that $\widetilde{W}_{J,i}\widetilde{W}_{B,i}\widetilde{U}_i\widetilde{W}_{B,i}^{-1}\widetilde{W}_{J,i}^{-1} = \widetilde{U}_i$. A straightforward computation shows that if $W_{B,i}$ and $W_{J,i}$ are the embeddings of $\widetilde{W}_{B,i}$ and $\widetilde{W}_{J,i}$ respectively into $\GL(m_id_i,F)$, then we can take $W_i = W_{J,i}W_{B,i}$.
\end{itemize}
Now suppose that $\C=\C(Q)$, where $Q$ is a non-singular quadratic form in even characteristic. If $f_i \in \Phi_1$, then we repeat the argument given above for the case $f_i \in \Phi_1$ replacing sesquilinear forms by quadratic forms; if $f_i \in \Phi_2 \cup \Phi_3$, then by Theorem \ref{ConjQuadratic} we can replace the quadratic form $Q$ by the associated sesquilinear form $\beta_Q$ and repeat the argument given above for the case $f_i \in \Phi_2\cup\Phi_3$.

\subsection{Conjugating element in special and Omega groups}
Here we assume that $X$ and $Y$ are conjugate in a special or in an Omega group. Using the notation of the previous section, we can compute a conjugating element $Z=P_X^{-1}WP_Y$ in the isometry group and $Z$. If such $Z$ already lies in the special or in the Omega group, then we have finished. Otherwise, the strategy is to replace $Z$ by $P_X^{-1}DWP_Y$, where $D \in C_{\C(B_X)}(J)$ has appropriate determinant and spinor norm.\\

Suppose first that $\Spec$ is the special orthogonal group. If $X$ and $Y$ are conjugate in $\Spec$, then we may suppose that $f_1 \in \Phi_1$ and there exists an odd $m_{1,j}$, so $D$ can be chosen with shape
\begin{eqnarray}  \label{formadiD}
\begin{pmatrix}
D_1 & \\ & \mathbb{I}_>
\end{pmatrix},
\end{eqnarray}
where $D_1 \in C_{\C(B_{X,1})}(J_1)$ has the same determinant as $Z$ and $\mathbb{I}_>$ is the identity matrix of dimension $\sum_{l>1} m_ld_l$.

Now let $\Spec$ be the special unitary group. For every $i=1, \dots, h$, let
$$
d_i = \gcd({q^2-1}, m_{i,1}, \dots, m_{i,k_i})
$$
and let $d = \gcd(q^2-1, d_1, \dots, d_h)$. If $\lambda_1, \dots, \lambda_h$ are integers such that $\sum_i \lambda_id_i=d$, then let 
$$
H = \begin{pmatrix}
H_1^{\lambda_1} && \\ & \ddots & \\ && H_h^{\lambda_h}
\end{pmatrix},
$$
where $H_i \in C_{\C(B_{X,i})}(J_i)$ has determinant of maximum order (they can be obtained as explained in Remark \ref{Remarkutile}). Since $X$ and $Y$ are conjugate in $\Spec$, there exists an integer $\ell$ such that $\det{Z} = \det{H}^{\ell}$, so we take $D = H^{-\ell}$.\\

Finally, let us consider the case $X,Y \in \Omega(\beta)$. In even characteristic, we can suppose that $f_1 \in \Phi_1$ and there exists $D_1 \in C_{\C(B_{X,1})}(J_1)$ with $\theta(D_1) = \theta(Z)$, so we choose
$$
D = \begin{pmatrix}
D_1 & \\ & \mathbb{I}_>
\end{pmatrix}.
$$
Suppose we are in odd characteristic. Assume that $X$ and $Y$ are conjugate in $\Omega$, ${C_{\Spec}(X) \neq C_{\Omega}(X)}$ and we got a conjugating element $Z \in \C$ with inappropriate determinant or spinor norm. One of the following cases applies.
\begin{itemize}
	\item There exists $f_i \in \Phi_2 \cup \Phi_3$ and $m_{i,j}$ odd. Now $C_{\C(B_{X,i})}(J_i)$ contains elements of non-square spinor norm. Hence, we take $D = D_sD_o$, where either $D_s=\mathbf{1}_V$ (if $\det{Z}=1$) or it is defined as in (\ref{formadiD}) (if $\det{Z}=-1$ and $f_1 \in \Phi_1$), and
	$$
	D_o = \begin{pmatrix}
	\mathbb{I}_< && \\ & D_i & \\ && \mathbb{I}_>
	\end{pmatrix},
	$$
	where $D_i \in C_{\C(B_{X,i})}(J_i)$ has the same spinor norm as $P_X^{-1}D_sWP_Y$, and $\mathbb{I}_<$ and $\mathbb{I}_>$ are identity matrices of dimension $\sum_{l<i} m_ld_l$ and $\sum_{l>i} m_ld_l$ respectively.
	\item There exists $f_i \in \Phi_1$ with $C_{\Spec(B_{X,i})}(J_i) \not\subseteq \Omega(B_{X,i})$. We saw in Theorem \ref{ConjClassinOmega} that we can pick $D_i \in C_{\C(B_{X,i})}(J_i)$ with the same determinant and spinor norm as $Z$, so we take
	$$
	D = \begin{pmatrix}
	\mathbb{I}_< && \\ & D_i & \\ && \mathbb{I}_>
	\end{pmatrix}.
	$$
	\item Suppose $f_i \in \Phi_1$ and $C_{\Spec(B_{X,i})}(J_i) \subseteq \Omega(B_{X,i})$ for $i=1,2$. Now every conjugating element $Z \in \C$ satisfies either $\det{Z}=1$, or $\theta(Z)=0$ and ${C_{\C(B_{X,i})}(J_i)\not\subseteq \Spec(B_{X,i})}$ for at least one $i$, say $i=1$ (otherwise $X$ and $Y$ could not be conjugate in $\Omega$). If $\det{Z}=1$ and $\theta(Z)=1$, then ${C_{\C(B_{X,i})}(J_i)\not\subseteq \Spec(B_{X,i})}$ for both $i$ and we take
	$$
	D = \begin{pmatrix}
	D_1 && \\ & D_2 & \\ && \mathbb{I}_>
	\end{pmatrix}
	$$
	with $D_i \in C_{\C(B_{X,i})}(J_i) \setminus C_{\Spec(B_{X,i})}(J_i)$ for $i=1,2$. If $\det{Z} =- 1$ and $\theta(Z) =0$, then we take  $D_1 \in C_{\C(B_{X,1})}(J_1) \setminus C_{\Spec(B_{X,1})}(J_1)$ and $D_2 \in C_{\Spec(B_{X,2})}(J_2)$.
\end{itemize}

\newpage
\addcontentsline{toc}{chapter}{Notation}
\printindex

\end{document}